\documentclass[11pt]{article}
\usepackage{amsmath,amssymb,amsthm}
\usepackage{graphicx,enumerate,a4wide}
\usepackage{pdfsync,epstopdf,comment,url}
\usepackage[x11names,usenames,dvipsnames,svgnames,table,rgb]{xcolor}
\usepackage{multirow}
\usepackage{float} % Use H in figure as a location specifier
\numberwithin{equation}{section}
\usepackage{longtable} % for tables that span over more than two pages
\usepackage{caption} % To change the caption name
\usepackage{hyperref}

\usepackage[capitalize]{cleveref}

\usepackage{tabularx,booktabs}

\usepackage{booktabs} % For table
\usepackage{multirow} % For table

\includecomment{mytodos} 
\usepackage[obeyDraft]{todonotes}
\usepackage{ifdraft,varwidth}
\makeatletter
\tikzstyle{diaanotestyle} = [
draw=\@todonotes@currentbordercolor,
fill=\@todonotes@currentbackgroundcolor,
line width=0.5pt,
inner sep = 0.8 ex,
rounded corners=4pt,align=left,
]

\ifdraft{
	\renewcommand{\@todonotes@drawInlineNote}{%
		\noindent{\begin{tikzpicture}[remember picture,baseline={(0,-0.3)}]%
				\draw node[diaanotestyle,font=\@todonotes@sizecommand,anchor=base west]{%
					\begin{varwidth}[t]{6.4in}
						\if@todonotes@authorgiven%
						{\@todonotes@sizecommand \@todonotes@author:\,\@todonotes@text}%
						\else%
						{\@todonotes@sizecommand \@todonotes@text}%
						\fi
				\end{varwidth}};%
			\end{tikzpicture}~}%
	}%
}

\numberwithin{equation}{section}

\newtheorem{theorem}{Theorem}[section]

\newtheorem{corollary}{Corollary}[section]
\newtheorem{proposition}{Proposition}[section]
\newtheorem{lemma}{Lemma}[section]
\theoremstyle{definition}
\newtheorem{remark}{Remark}[section]
\newtheorem{definition}{Definition}[section]

\newtheorem{assumption}{Assumption}
\crefname{assumption}{Assumption}{Assumptions}

\newcommand{\nn}{\mathbb{N}} %Nonnegative integers
\newcommand{\erl}{\left(-\infty , +\infty\right]} %Extended real line
\newcommand{\dom}[1]{\mathrm{dom}\,{#1}} %Domain`
\newcommand{\idom}[1]{\mathrm{int\,(dom}\,{#1})} %Interior of the Domain
\newcommand{\conv}[1]{\mathrm{conv}\,{#1}} %Convex hull
\newcommand{\spn}[1]{\mathrm{span}\,{#1}} %Linear hull
\DeclareMathOperator{\aff}{aff}
\newcommand{\ri}[1]{\mathrm{ri}\,{#1}} %Relative interior
\newcommand{\cl}[1]{\mathrm{cl}\,{#1}} %Closure
\newcommand{\cconv}[1]{\mathrm{cl\,(conv}\,{#1})}
\newcommand{\bdr}[1]{\mathrm{bd}\,{#1}} %Boundary of a set
\newcommand{\argmin}{\mathop{\mathrm{argmin}}}
\newcommand{\argmax}{\mathop{\mathrm{argmax}}}

\newcommand{\prox}{\mathrm{prox}} %Proximal map
\newcommand{\BBB}{\mathcal{B}}
\newcommand{\CCC}{\mathcal{C}}

\newcommand{\FFF}{\mathcal{F}}

\newcommand{\MMM}{\mathcal{M}}

\newcommand{\PPP}{\mathcal{P}}
\newcommand{\QQQ}{\mathcal{Q}}

\newcommand{\EE}{\mathbb{E}}

\newcommand{\bbS}{\mathbb{S}}

\newcommand{\real}{\mathbb{R}} %Real numbers
\newcommand{\intg}{\mathbb{Z}} %All Integer numbers
\newcommand{\nnn}{\ensuremath{{k\in{\mathbb N}}}}
\newcommand{\inte}{\ensuremath{\operatorname{int}}}
\newcommand{\inner}[2]{{\langle#1,#2\rangle}}

\newcommand{\KL}[2]{{D_{\scriptscriptstyle\text{KL}}(#1~||~#2)}}

\newcommand{\pr}{\color{purple}}

\newcommand{\withsmallskip}{\smallskip}

\newcommand\borderline{{\pr\noindent\leavevmode\xleaders\hbox{$=$}\hfill\kern0pt}\\}
\newcommand\todoborderline{{\pr\noindent\leavevmode\xleaders\hbox{\todo{}}\hfill\kern0pt}\\}

\newcommand*\samethanks[1][\value{footnote}]{\footnotemark[#1]}

\makeatletter
\newcommand*{\rom}[1]{\expandafter\@slowromancap\romannumeral #1@}
\makeatother

\newcommand\xqed[1]{%
  \leavevmode\unskip\penalty9999 \hbox{}\nobreak\hfill
  \quad\hbox{#1}}
\newcommand*{\rqed}{\xqed{$\Diamond$}}
\title{
Maximum Entropy on the Mean and the Cram\'er Rate Function in Statistical Estimation and Inverse Problems: Properties, Models, and Algorithms
}
\author{Yakov Vaisbourd\thanks{Department of Mathematics and Statistics, McGill University},  Rustum Choksi\samethanks, Ariel Goodwin\samethanks, Tim Hoheisel\samethanks, \\ \& Carola-Bibiane Sch\"{o}nlieb\thanks{Department of Applied Mathematics and Theoretical Physics, University of Cambridge.}}

\date{}

\begin{document}

\maketitle

\begin{abstract}
	We explore a method of statistical estimation called {\it Maximum Entropy on the Mean} (MEM) which is based on an information-driven criterion that quantifies the compliance of a given point with a reference prior probability measure. 
	At the core of this approach lies the {\it MEM function} which is a partial minimization of the Kullback-Leibler divergence over a linear constraint. In many cases, it is known that this function admits a simpler representation (known as the {\it Cram\'{e}r rate function}). Via the connection to exponential families of probability distributions, we study general conditions under which this representation holds.  We then address how the associated {\it MEM estimator}  gives rise to a wide class of MEM-based regularized linear models for solving inverse problems. Finally, we propose an algorithmic framework to solve these problems efficiently based on the Bregman proximal gradient method, alongside proximal operators for commonly used reference distributions. The article is complemented by a software package for experimentation and exploration of the MEM approach in applications.
  \end{abstract}

{\small

}

\section{Introduction}

Many models for modern applications in various disciplines are based on some form of {\it statistical estimation}, for example, the very common 
\emph{maximum likelihood (ML)} principle. In this study,  we consider an alternative approach known as the \emph{maximum entropy on the mean} (MEM).  
%\cite{rietsch1977maximum}. 
At its core lies the MEM function $\kappa_P$ induced by some \emph{reference distribution} $P$ and defined as
\begin{gather*}
%	\label{mem:eq:mem}
	\kappa_P(y):=\inf\left\{\KL{Q}{P}: \mathbb{E}_Q=y, Q\in\PPP(\Omega)\right\},%\EE_Q[Y]=y, Q\in\PPP(\Omega)\right\}.
\end{gather*}
where $\mathcal P(\Omega)$ stands for the set of probability measures on $\Omega\subseteq\real^d$, $\mathbb{E}_Q$ is the expected value of $Q\in \mathcal P(\Omega)$ and $\KL{Q}{P}$
stands for the Kullback-Leibler (KL) divergence of $Q$ with respect to $P$ \cite{kullback1951information} (see \Cref{sec:pre} for precise definitions).  
Thus, the MEM modeling paradigm stems from the principle of minimum discrimination information \cite{kullback1968information} which generalizes the well-known principle of maximum entropy \cite{jaynes1957information}.
In the context of information theory %\cite{Cover-Thomas}
\cite{cover1999elements}, the argmin of $\kappa_P(y)$ is often referred to as the {\it information projection} of $P$ onto the set $\{ Q\in \mathcal P(\Omega) \, : \,  \mathbb{E}_Q=y\}$, the {\it closest} member of the set to $P$. 
 \withsmallskip
 
Various forms and interpretations of  MEM have been studied (see, for example, %\cite{Gamboa-0, Gamboa-1, le1999new, marechal1997unification, Gamboa-2, Gamboa-3}
\cite{dacunha1990maximum, gamboa1989methode, gamboa1997bayesian, gamboa2021maximum, gzyl2003maximum, le1999new, marechal1997unification}) and found applications in various disciplines, including earth sciences \cite{fermin2006bayesian,navaza1985maximum,navaza1986use,rietsch1977maximum,urban1996retrieval}, and 
medical imaging \cite{amblard2004biomagnetic,cai2022diffuse,chowdhury2013meg,grova2006evaluation,heers2016localization}. 
%For example, in \cite{chowdhury2013meg} models that were implemented within the MEM framework provided superior result \CBS{Superior to what?} in localizing the generators of epileptic activity in the brain using Electroencephalography (EEG) or Magnetoencephalography (MEG) signals. 
A version of the  MEM method was recently explored for blind deblurring of images possessing some form of fixed symbology (for example, in barcodes)  \cite{rioux2019blind,rioux2020maximum}. There one exploited the ability of 
 of the MEM framework to facilitate the incorporation of nonlinear constraints via the introduction of a prior distribution.\withsmallskip

Despite its many interesting properties in both theory and applications, the MEM methodology has yet to find its place as a mainstream tool for statistical estimation, particularly as it pertains to solving inverse problems. One factor that might have contributed to this centers on the practical issue that there are no dedicated optimization algorithms designed to tackle models based on the MEM methodology.  Indeed, the MEM function is defined by means of an infinite-dimensional optimization problem. Previous attempts to solve models involving the MEM function relied on its finite-dimensional dual problem. To the best of the authors' knowledge, there are no dedicated optimization algorithms designed to tackle models based on the MEM methodology. Therefore, any researcher or practitioner wishing to employ the MEM framework must first overcome a notable barrier of deriving an appropriate optimization algorithm for its solution.
%  \CBS{Please check that what we are saying here is still true and makes sense after my modifications ;)}\\
In this work, our goal is to fill in this gap, providing an access gate to the MEM methodology. 
\withsmallskip

Our approach is based on the fundamental work by Brown \cite[Chapter 6]{brown1986fundamentals} and complements  \cite{le1999new}
% Our approach follows and complements the important paper of Le Besnerais et.~al \cite{le1999new}  
 by first proving the equivalence of the MEM function to the \emph{Cram\'er's rate} function, mostly known from its role in \emph{large deviation theory}. Cram\'er's rate function is defined by means of a finite-dimensional optimization problem as it is simply the convex conjugate of the log-normalizer (aka the cumulant generating function) of the reference distribution $P$. In many cases (i.e., choices of $P$) it admits a closed-form expression while in others it can still be evaluated efficiently. The connection between these seemingly different functions is well established in the large deviations \cite{donsker1976asymptotic}, statistics \cite{brown1986fundamentals}, and information theory \cite{le1999new} literature. Nonetheless, various assumptions imposed in the aforementioned works limit the scope of existing results. 
Employing the framework of exponential families of probability distributions \cite{brown1986fundamentals}, 
we establish the equivalence between the two functions under very mild and natural conditions, allowing us to cover many distributions of practical interest. Thus, 
models involving MEM functions can be explicitly stated using the corresponding Cram\'er functions. 
\withsmallskip
%\CBS{To do so (?)}, we focus on a modeling paradigm motivated by \cite{brown1986fundamentals} for exponential families and extend it to distributions beyond its original scope. 

Central to our study is  {\it the MEM estimator} which is shown to be well-defined under very mild conditions. 
We further recall an insightful connection between the MEM and ML estimators as presented in \cite{brown1986fundamentals} for the case of a reference distribution from an exponential family. As with the ML counterpart, the MEM estimator has vast applications, and hence we restrict the remainder of the paper to a wide class of regularized linear models for solving inverse problems. Each model in this class involves two MEM functions, one in the role of a fidelity term and another as a regularizer (comparable to the {\it maximum a priori (MAP) estimation} framework which extends ML). 
Let us provide an example: given a measurement matrix $A\in\real^{m\times d}$, an observation vector $\hat{y}\in\real^m$ and an additional vector $p\in[0,1]^d$ representing some prior knowledge, the following optimization problem 
\begin{gather*}
%	\label{intro:eq:normal_with_bernoulli_prior}
	\min\left\{\frac{1}{2}\|Ax-\hat{y}\|_2^2+\sum_{i=1}^d\left[x_i\log\left(\frac{x_i}{p_i}\right)+(1-x_i)\log\left(\frac{1-x_i}{1-p_i}\right)\right]:x\in[0,1]^d\right\},\\[-0.4cm]
\hspace{-1.4cm}\underbrace{\qquad\qquad\quad}_{Fidelity}\quad\underbrace{\qquad\qquad\qquad\qquad\qquad\qquad\qquad\qquad\qquad}_{Regularization}
\end{gather*}
fits the MEM framework with normal (Gaussian) and Bernoulli reference distributions of the fidelity and regularization terms, respectively. 
%\CBS{I wonder whether here we could / should mention the equivalent interpretation of the variational model below as a MAP estimator with Gaussian likelihood and regulariser being the negative log of a Gibbs prior for a Bernoulli distribution (or another distribution? not sure if this exactly corresponds...).}
Other choices of reference distributions will lead to additional models that admit similar additive composite structure. 
 Moreover, the closed-form expressions of the two functions in our example follow from the definition of Cram\'er's rate function. In models of these forms,  concrete expressions and structures with distinct geometry can be exploited to customize appropriate optimization strategies. 
Here we highlight the class of \emph{Bregman proximal gradient} (BPG) methods as an especially suitable choice for this family of models. Nevertheless, other methods are also viable alternatives; for example, adaptive and scaled, accelerated variants and dual decomposition methods which are defined by means of the same operators developed here. 
% Note that variational models such as (\ref{VP-1}) can also be derived as {\it Maximum A Posteriori (MAP)} estimators with the fidelity term corresponding to the log-likelihood and the regularizer corresponding to the log of a Gibbs prior. The MEM interpretation gives rise to a new modelling paradigm with possibly new interpretations that open new perspectives and result in so-far unseen variational regularization models.
\withsmallskip

Our overall aim is to provide a self-contained, mathematically sound toolbox for working with the MEM methodology for a wide variety of models. For this reason, we provide a comprehensive list of Cram\'er functions and operators used in the algorithms and complement it with a software package. We believe this sets the basis for (and hopefully triggers) further experimentation and exploration of the MEM approach in contemporary applications.
\withsmallskip

The paper is organized as follows. In \Cref{sec:pre}, we recall some concepts and preliminary results from convex analysis and probability theory which will be used in this work. In \Cref{sec:mem}, we study the MEM and Cram\'er rate functions and establish the equivalence between the two under very mild and natural conditions. This allows us to use the accessible definition of the Cram\'er function and derive tractable expressions for a wide class of possible reference distributions which closes this section (see \cref{mem:tbl:cramer}). \Cref{sec:model} is devoted to the MEM models considered in this work, and in \Cref{sec:algos}, we present the algorithms for solving such models. We end with a few concrete examples of problems and corresponding algorithms crafted from the operators derived in this work. 
An appendix provides %deferred proofs and 
the details of a variety of Cram\'er rate function computations.

\section{Preliminaries}
\label{sec:pre}
%We will use standard notation that can be found in various textbooks such as 
%\cite{rockafellar1970convex}.
% We establish some preliminary results and notation that will be used throughout the paper. 

\subsection{Convex Analysis}

We recall here some definitions and results from convex analysis. Further details and proofs can be found in various textbooks such as \cite{bauschke2011convex,beck2017first,rockafellar1970convex}. \withsmallskip

The \emph{affine hull} of a set $S\subseteq\real^d$ is the smallest affine subspace containing $S$. For any point $y\in S$, we have the following relation
\begin{gather}
	\label{mem:eq:affine_hull}
	\aff S =y+\spn{(S-y)},
\end{gather}
where $\spn{S}$ stands for the linear hull of $S$.%We will denote the subspace $\LLL(S):=\spn (S-x_0)$.\footnote{Observe that $\LLL(S)$ is independent of the choice of $x_0\in S$.} 
The dimension of $\aff{S}$ is defined as $\dim(\aff{S}):=\dim\left(\spn{(S-y)}\right)$. The interior, closure, and boundary of a set are denoted as $\inte{S},~\cl{S}$ and $\bdr{S}$, respectively.%\dim\left(\LLL(S)\right)$.\\ %\spn{(S-x_0)}$ where $x_0\in S$. 
\withsmallskip

The (Fenchel) conjugate of $\psi:\real^d\rightarrow[-\infty, \infty]$ is defined as 
\begin{gather*}
	\psi^*(y):=\sup\{\inner{y}{x}-\psi(x):x\in\real^d\}.
\end{gather*}
The function $\psi$ is proper if $\psi(x)>-\infty$ for all $x\in\real^d$ and $\dom \psi := \{x\in\real^d:\psi(x)<\infty\}\neq\emptyset$. In addition, $\psi$ is closed, if its epigraph $\{(x,\alpha)\in\real^d\times\real:\psi(x)\leq \alpha\}$ is a closed set.
\withsmallskip

If $\psi$ is proper and convex then $\psi^*$ is closed, proper, and convex. For a proper function $\psi:\real^d\rightarrow\erl$, the \emph{Fenchel-Young inequality} states that 
$
	\psi(x)+\psi^*(y)\geq\inner{y}{x}.$
 If $\psi$ is proper, closed and convex then we obtain that \cite[Theorem 4.20]{beck2017first}
\begin{gather}
	\label{pre:eq:fechel_in_equality_equivs}
	\psi(x)+\psi^*(y)=\inner{y}{x}\quad\Longleftrightarrow\quad y\in\partial \psi(x) \quad\Longleftrightarrow\quad x\in\partial \psi^*(y),
\end{gather}
where $\partial \psi(x):=\{g\in\real^d:\psi(y)\geq\psi(x)+\inner{g}{y-x}~(y\in\real^d)\}$ is the \emph{subdifferential} of $\psi$ at $x\in\real^d$.
\withsmallskip

The \emph{indicator function} of a set $S\subseteq\real^d$ is denoted by $\delta_S$ and defined as $\delta_S(x)=0$ if $x\in S$ and $\delta_S(x)=+\infty$ otherwise. Its convex conjugate is known as the \emph{support function} $\sigma_S(y):=\delta_S^*(y)=\sup\{\inner{y}{x}:x\in S\}$. 

\begin{definition}[Essential smoothness and Legendre type]
%[The Legendre transformation {\cite[Chapter 26]{rockafellar1970convex}}]
	\label{pre:def:Legendre}
	Let $\psi:\real^d\rightarrow\erl$ be proper and convex. Then, $\psi$ is called \emph{essentially smooth} if it satisfies the following conditions:
	\begin{enumerate}
		\item $\idom{\psi}\neq\emptyset$;
		\item $\psi$ is differentiable on $\idom{\psi}$;
		\item $\|\nabla \psi(x^k)\|\rightarrow\infty$ for any sequence $\{x^k\in\idom{\psi}\}_\nnn\rightarrow \bar{x}\in\bdr{(\dom{\psi})}$.
	\end{enumerate}
	The last condition listed above is called \emph{steepness}. %and any function that satisfies it is called \emph{steep}. 
	An essentially smooth function $\psi$ is said to be of  \emph{Legendre type} if it is strictly convex on $\idom{\psi}$.
\end{definition}
%Let $\psi:\real^d\rightarrow\erl$ be proper and convex. Then, $\psi$ is called \emph{essentially smooth} if it satisfies the following three conditions \cite[Chapter 26]{rockafellar1970convex}:
%\begin{enumerate}
%	\item $\idom{\psi}$ is not empty;
%	\item $\psi$ is differentiable for any point in $\idom{\psi}$;
%	\item $\|\nabla \psi(x^k)\|\rightarrow\infty$ for any sequence $\{x^k\}_\nnn\subset\idom{\psi}$ such that $x^k\rightarrow \bar{x}\in\bdr{(\dom{\psi})}$.
%\end{enumerate}
%The last condition listed above is called \emph{steepness}. %and any function that satisfies it is called \emph{steep}. 
%An essentially smooth function $\psi$ is called of \emph{Legendre type} if it strictly convex on $\idom{\psi}$.\\
For  $\psi:\real^d\rightarrow\erl$ closed and of Legendre type, the following hold  \cite[Theorem 26.5]{rockafellar1970convex}: 
\begin{enumerate}
	\item $\psi^*$ is of Legendre type.
	\item $\nabla\psi:\idom{\psi}\rightarrow\idom{\psi^*}$ is a bijection with $(\nabla \psi)^{-1}=\nabla \psi^*$.
\end{enumerate}
The \emph{Bregman distance} induced by a function $\psi$ of Legendre type is defined as \cite{bregman1967relaxation}
\begin{gather*}
	D_\psi(y,x) = \psi(y)-\psi(x)-\inner{\nabla \psi(x)}{y-x}\qquad (x\in\idom \psi, y\in\dom\psi). 
\end{gather*}
For any $(x,y)\in\idom \psi\times\dom\psi$, the Bregman distance is nonnegative $D_{\psi}(y,x)\geq0$, and equality holds if and only if $x=y$ due to strict convexity of $\psi$ \cite{bregman1967relaxation}. However, in general, $D_\psi$ is not symmetric, unless $\psi=(1/2)\|\cdot\|^2$ \cite[Lemma 3.16]{bauschke2001joint}. %\footnote{This observation can be attributed to A.N.~Iusem (see \cite[Lemma 3.16]{bauschke2001joint}).} 
The Bregman distance induced by a function $\psi$ of Legendre type satisfies the following additional properties \cite[Theorem 3.7]{bauschke1997legendre}:  For any $x,y\in\idom{\psi}$ it holds that
\begin{gather}
	\label{pre:eq:Bregman_duality}
	D_\psi(y,x) = D_{\psi^*}(\nabla \psi(x),\nabla \psi(y)).
\end{gather}
The Bregman distance is strictly convex with respect to its first argument. Moreover, for two functions $\psi_1$ and $\psi_2$ differentiable at $x\in\idom {\psi_1}\cap\idom{\psi_2}$ %and $\alpha,\beta\in\real$
\begin{gather}
\label{pre:eq:Bregman_linear_additivity}
D_{\alpha\psi_1+\beta\psi_2}(y,x)=\alpha D_{\psi_1}(y,x)+\beta D_{\psi_2}(y,x)\quad (y\in\dom \psi_1\cap\dom \psi_2,~\alpha,\beta\in\real).
\end{gather} 

% \begin{description}
% 	\item[Duality:] For any $x,y\in\idom{\psi}$ it holds that
% 		\begin{gather}
% 			\label{pre:eq:Bregman_duality}
% 			D_\psi(y,x) = D_{\psi^*}(\nabla \psi(x),\nabla \psi(y)).
% 		\end{gather}
% 	\item[Convexity:] 
% 	The Bregman distance is strictly convex with respect to its first argument.
% 	\item[Linear additivity:] For two functions $\psi_1$ and $\psi_2$ differentiable at $x\in\idom {\psi_1}\cap\idom{\psi_2}$ %and $\alpha,\beta\in\real$
% 	\begin{gather}
% 		\label{pre:eq:Bregman_linear_additivity}
% 		D_{\alpha\psi_1+\beta\psi_2}(y,x)=\alpha D_{\psi_1}(y,x)+\beta D_{\psi_2}(y,x)\quad (y\in\dom \psi_1\cap\dom \psi_2,~\alpha,\beta\in\real).
% 	\end{gather} 
% \end{description}

\subsection{Probability Theory and Exponential Families}

We recall some concepts from probability theory with an emphasis on exponential families. For further detail, see e.g. \cite{barndorff1978information,brown1986fundamentals}.%,schilling2017measures}.\\
\withsmallskip

Let $\MMM(\Omega)$ be the set of $\sigma$-finite measures defined over a measurable space $(\Omega,\Sigma)$ where $\Omega\subseteq\real^d$ and $\Sigma$ is a \emph{$\sigma$-algebra} on $\Omega$. %set $\Omega\subseteq\real^d$. 
The \emph{support} of $\rho$, namely the minimal closed measurable set $A\in\Sigma$ such that $\rho(\Omega\setminus A)=0$, is denoted by $\Omega_\rho$. We denote by $\Omega_\rho^{cc}:=\cconv{\Omega_\rho}$ the closure of the convex hull of the support $\Omega_\rho$, which is known as the \emph{convex support} of $\rho$. Recall further that, if $\mu$ is another %$\sigma$-finite 
measure defined over $(\Omega,\Sigma)$, then $\mu$ is \emph{absolutely continuous} with respect to $\rho$ (denoted by $\mu\ll \rho$) if for every $A\in\Sigma$ %measurable subset $A\subseteq\Omega$ 
such that $\rho(A)=0$ it holds that $\mu(A)=0$. In this case, the \emph{Radon-Nikodym derivative} is the unique function $h=\frac{d\mu}{d\rho}$ such that $\mu(A) = \int_A hd\rho$ for any $A\in\Sigma$. %measurable subset $A\subseteq\Omega$
% \begin{gather*}
% 	\mu(A) = \int_A hd\rho.
% \end{gather*}
For a measurable space $(\Omega,\Sigma)$ 
%set $\Omega\subseteq\real^d$ 
we denote by $\nu\in \MMM(\Omega)$ the \emph{dominating measure}.
%Let $\PPP(\Omega)$ be the set of probability measures defined over $\Omega$. 
Throughout, we restrict ourselves to two scenarios: either $\Omega=\real^d$ and $\nu$ is the Lebesgue measure or $\Omega$  is a countable subset of $\real^d$ and $\nu$ is the counting measure. Let $\PPP(\Omega)$ be the set of probability measures defined over $\Omega$ and absolutely continuous with respect to $\nu$. 
We emphasize that for $P\in\PPP(\Omega)$ the support $\Omega_P$ might be a proper subset of $\Omega$, and thus there is no loss of generality in our setting even when $\Omega=\real^d$.
Furthermore, for any set $A\subseteq\real^d$ the expression $P(A)$ should be understood as $P(A\cap\Omega)$.
%any probability measure $P\in\PPP(\real^d)$ is absolutely continuous with respect to the Lebesgue measure. In the second case, any probability measure $P\in\PPP(\Omega)$ is concentrated on $\Omega$ and is absolutely continuous with respect to the counting measure. In any case, the \emph{dominating measure} (Lebesgue or counting) will be denoted by $\nu$, and consequently, it holds that $P\ll\nu$ for all $P\in\PPP(\Omega)$. 
For $P\in\PPP(\Omega)$, the Radon-Nikodym derivative $f_{P}:=\frac{dP}{d\nu}$ is either a probability density or mass function, depending on the set $\Omega$. In both cases, we will refer to $f_P$ as the density of the distribution.\footnote{We will interchangeably refer to $P\in\PPP(\Omega)$ as either a distribution or measure.} The  \emph{expected value} (if it exists) and  \emph{moment generating function} of ${P\in\PPP(\Omega)}$ are given by
\begin{gather*}
	\mathbb{E}_P := \int_\Omega ydP(y)\in\real^d\qquad \text{and}\qquad M_P[\theta]:=\int_\Omega \exp(\inner{\cdot}{\theta})dP,
\end{gather*}
respectively.  For  $P\in\MMM(\Omega)$  absolutely continuous with respect to $\nu$, we define
\begin{gather*}
	\Theta_P:=\left\{\theta\in\real^d:\int_{\Omega} \exp(\inner{\cdot}{\theta})dP<\infty \right\},
\end{gather*}
and consider the function $\psi_P:\real^d\rightarrow\erl$ given by 
\begin{gather}
	\label{pre:eq:log_normalizer}
	\psi_P(\theta):=\begin{cases}
		\displaystyle\log \int _{\Omega}\exp\left(\inner{\cdot}{\theta}\right)dP,& \theta\in\Theta_P,\\
		+\infty, & \theta\notin\Theta_P.
	\end{cases}
\end{gather}
Then $\FFF_P:=\left\{ f_{P_\theta}(y):=
%\frac{dP_\theta}{dP}(y)=
\exp\left(\inner{y}{\theta}-\psi_P(\theta)\right):\theta\in\Theta_P\right\}$,
% \begin{gather}
% 	\label{pre:eq:exp_family}
% 	\FFF_P:=\left\{ f_{P_\theta}(y):=
% 	%\frac{dP_\theta}{dP}(y)=
% 	\exp\left(\inner{y}{\theta}-\psi_P(\theta)\right):\theta\in\Theta_P\right\},
% \end{gather}
is a \emph{standard exponential family} generated by $P$. 
Note that, the probability measure $P_\theta$ satisfying $dP_\theta=f_{P_\theta}dP$
%Note that, by construction, $P_\theta$ 
is, by construction, a probability measure such that $P_\theta$ and $P$ are mutually absolutely continuous, hence $\Omega_{P_{\theta}}=\Omega_P$ for all $\theta\in\Theta_P$ \cite[Section 8.1]{barndorff1978information}. %In addition, the probability measures $\{P_\theta:\theta\in\Theta_P\}$ are mutually absolutely continuous. 
The function $\psi_P$ is called the \emph{log-normalizer} (also known as the \emph{log-partition} or \emph{log-Laplace transform} of $P$). The vector $\theta\in\real^d$ is known as the \emph{natural parameter} and the set $\Theta_P=\dom\psi_P$ is called the \emph{natural parameter space}.\footnote{It is possible  to define the exponential family $\FFF_P$ over a subset of the natural parameter space \cite[Definition 1.1]{brown1986fundamentals},  but this is not needed for our study.}  
\withsmallskip

The following results summarize some well-known properties of the log-normalizer $\psi_P$. 

\begin{proposition}[Convexity, {\cite[Theorem 1.13]{brown1986fundamentals}}]
	\label{pre:prop:exp_fam_convexity}
	Let $\FFF_P$ be an exponential family generated by $P\in\MMM(\Omega)$. Then, the natural parameter space $\Theta_P$ is a convex set, and the log-normalizer function $\psi_P:\real^d\rightarrow \erl$ is closed, proper, and convex. 
\end{proposition}

\begin{proposition}[Differentiability, {\cite[Theorem 2.2, Corollary 2.3]{brown1986fundamentals}}]
	Let $\FFF_P$ be an exponential family generated by $P\in\MMM(\Omega)$ and let $\theta\in\inte{\Theta_P}$. Then, the log normalizer $\psi_P:\real^d\rightarrow \erl$ is infinitely differentiable at $\theta$ and it holds that 
	%\begin{gather*}
	$\nabla \psi_P(\theta) = \mathbb{E}_{P_\theta}$.%\EE_P~\EE_P[\Omega].%\EE_{P_\theta}[Y].
	%\end{gather*}
\end{proposition}
\noindent The dimension of a convex set $S\subseteq\real^d$, denoted by $\dim{S}$, is equal to the affine dimension of $\aff{S}$. We 
assume that the exponential family generated by $P\in\MMM(\Omega)$ is \emph{minimal}, i.e.,  $\dim \Theta_P=\dim \Omega_P^{cc} = d$ or, equivalently, $\inte\Theta_P\neq\emptyset$ and $\inte\Omega_P^{cc}\neq\emptyset$.
This is not restrictive as a non-minimal exponential family can be always reduced to a minimal form \cite[Theorem 1.9]{brown1986fundamentals}. 
The following result strengthens \cref{pre:prop:exp_fam_convexity} for minimal exponential families. 
\begin{proposition}[Strict convexity, {\cite[Theorem 1.13]{brown1986fundamentals}}]
	Let $\FFF_P$ be a minimal exponential family generated by $P\in\MMM(\Omega)$. Then, the log-normalizer function $\psi_P:\real^d\rightarrow \erl$ is strictly convex over $\Theta_P$. 
\end{proposition}
\noindent If the log-normalizer  $\psi_P$ is essentially smooth (or  'steep' in the exponential family terminology, see, e.g., \cite[Theorem 5.27]{barndorff1978information} and \cite[Definition 3.2]{brown1986fundamentals}),  we say that the exponential family $\FFF_P$ is \emph{steep}. This condition is automatically satisfied when $\Theta_P$ is open \cite[Theorem 8.2]{barndorff1978information}. While most exponential families encountered in practice have this property, there are relevant cases when this assumption is too restrictive (e.g., \cite[Example 3.4]{brown1986fundamentals}). Thus, in order to cover all examples provided in this work, we will assume that the exponential family is steep. Summarizing the above discussion and recalling \cref{pre:def:Legendre} we have the following corollary. 
\begin{corollary}\label{pre:cor:Legendre}
	Let $\FFF_P$ be a minimal and steep exponential family generated by $P\in\MMM(\Omega)$. Then, the log normalizer function $\psi_P$ is of Legendre type. 
\end{corollary}
\noindent From the last corollary we can see that $\nabla \psi_P$ forms a bijection between $\idom {\psi_P}=\inte \Theta_P$ and $\idom {\psi_P^*}$. This relation provides a dual representation of the log-normalizer $\psi_P$ and, consequently, the distribution in question. The so-called \emph{mean value parametrization} is obtained by applying a change of variables where the natural parameter $\theta$ is replaced by $\mu\in\real^d$ such that $\mu=\mathbb{E}_{P_\theta} = \nabla \psi_P(\theta)$, i.e., $\theta=\nabla \psi_P^*(\mu)$.

The \emph{Kullback-Leibler (KL) divergence} (also known as the relative entropy) of a probability measure $Q\in\PPP(\Omega)$ with respect to $P\in\PPP(\Omega)$
is given by (see \cite{kullback1951information})%\cite[Definition 6.1]{brown1986fundamentals}
\begin{gather*}
%	\label{eq:KL}
%	\tag{KL}
	\KL{Q}{P}:=\begin{cases}
		\displaystyle\int_\Omega \log\left(\frac{dQ}{dP}\right)dQ,& Q\ll P,\\
		+\infty, & \text{otherwise}.
	\end{cases}% = {\ed add the }.
\end{gather*}
%If $P\in\PPP(\Omega)$ then 
It holds that $\KL{Q}{P}\geq 0$ with equality if and only if $Q=P$ \cite[Lemma 3.1]{kullback1951information}. Thus, the Kullback-Leibler information quantifies the dissimilarity between two probability measures. We note that, in general, $\KL{Q}{P}$ is not symmetric.
Furthermore, $\KL{Q}{P}$ is jointly convex in $(Q|P)$. 
We record a special case for which the KL divergence is of particular interest.

\begin{remark}[Kullback-Leibler divergence for exponential family]
	\label{pre:rmrk:KL_special_cases}
Let $\FFF_P$ be an exponential family generated by $P\in\MMM(\Omega)$. Let $\theta_1\in\Theta_P$ and $\theta_2\in\inte\Theta_P$, thus for $i=1,2$ we have that $f_{P_{\theta_i}}\in \FFF_P$. In this case, the KL divergence between the two measures $P_{\theta_i}\in\PPP(\Omega)$ such that $dP_{\theta_i}:=f_{P_{\theta_i}}dP$ ($i=1,2$) satisfies $\KL{P_{\theta_2}}{P_{\theta_1}} = D_{\psi_P}(\theta_1,\theta_2)$ \cite[Proposition 6.3]{brown1986fundamentals}.\rqed
		% \begin{gather*}
		% 	\KL{P_{\theta_2}}{P_{\theta_1}} = D_{\psi_P}(\theta_1,\theta_2).
		% \end{gather*}
\end{remark}

\section{Maximum entropy on the mean and Cram\'er's rate function} 
\label{sec:mem}

For $y\in\mathbb R^d$, the density
\begin{gather}
	f_P(y):=\frac{dP}{d\nu}(y)
\end{gather}
provides an indication of the likelihood of $y$ under the distribution $P\in\PPP(\Omega)$. 
%Thus, this function quantifies the compliance of the point $y$ with the distribution defined by the probability measure $P$. 
The method of \emph{Maximum Entropy on the Mean} (MEM)
suggests an alternative, information-driven function $\kappa_P:\real^d\rightarrow\erl$ given by
\begin{gather}
	\label{mem:eq:mem}
	\kappa_P(y):=\inf\left\{\KL{Q}{P} : \mathbb{E}_Q=y, Q\in\PPP(\Omega)\right\}.%\EE_Q[Y]=y, Q\in\PPP(\Omega)\right\}.
\end{gather}
Here,  $\kappa_P$   measures  how  $y$ complies   with the distribution $P$, by seeking a  distribution $Q$ with expected value $y$ that minimizes $\KL{\cdot}{P}$. The distance, in terms of the KL divergence (the information gain) between the resulting and the original distributions, quantifies the compliance of $y$ with $P$. We will refer to  $\kappa_P$ as the \emph{MEM function} and to $P$ as the %\emph{reference measure}, or 
\emph{reference distribution}. 
Since $\KL{Q}{P}\geq 0$ and $\KL{Q}{P}=0$ if and only if $Q=P$, we find that the MEM function satisfies $\kappa_P(y)\geq 0$ for any $y\in\real^d$ and $\kappa_P(y)= 0$ if and only if $y=\mathbb{E}_P$. %$y=\EE_P[Y]$. 
\withsmallskip

In most cases of interest, the MEM function admits an alternative representation that sheds light on many of its additional properties (cf. \cref{mem:thrm:MEM_properties}). More precisely, under suitable conditions (cf. \cref{mem:thrm:Cramer_and_MEM_equiv}), the MEM function coincides with the  \emph{Cram\'er rate function} \cite{cramer1938nouveau},
to which we turn now. For a given reference distribution $P\in\PPP(\Omega)$, recall the log-normalizer previously defined for a general measure in \eqref{pre:eq:log_normalizer}: 
\begin{gather*}
	\label{mem:eq:log_Laplace}
	\psi_P(\theta):=\log M_P[\theta]
	%\log  M_Y[\theta]
	=\log \int_{\Omega} \exp\left(\inner{\cdot}{\theta}\right)dP.
\end{gather*}
 In the context of probability measures $P$, $\psi_P$ is often known as the {\it cumulant generating function}. 
The \emph{Cram\'er rate function} $\psi_P^* $ associated with $P$ is the conjugate of  $ \psi_P$, that is, 
\[ \psi_P^* (y) \, = \, \sup\{\inner{y}{\theta}-\psi_P(\theta):\theta\in\real^d\}.\]
%The terminology and notation we used above is justified by the fact that \eqref{mem:eq:log_Laplace} coincides with the definition of the log-normalizer function $\psi_P$ as given in \eqref{pre:eq:log_normalizer} when $P\in\PPP(\Omega)$. Thus, the probability measure $P$ generates the exponential family $\FFF_P$. 
Our central assumption (which is not too restrictive in view of our discussion above) on the prior $P$ and its exponential family $\FFF_P$ is provided below. The additional condition $0\in\inte \Theta_P$ ensures the existence of $\EE_P$. 
%We will assume that the reference distribution generates a minimal and steep exponential family. 
\begin{assumption}
	\label{mem:asmp:blanket_minimal_and_steep}
	The reference distribution $P\in\PPP(\Omega)$ generates a minimal and steep exponential family $\FFF_P$ such that $0\in\inte \Theta_P$.
\end{assumption}

%It turns out that, under conditions studied below the convex conjugate $\psi_P^*$ of the log-normalizer function, which in this case is known as the Cram\'er rate function, coincides with the MEM function $\kappa_P$. %which we record in Theorem \ref{mem:thrm:Cramer_and_MEM_equiv}. 
The equivalence between the two seemingly different functions\footnote{$\psi_P^*$ appears   %\cite{ellis2006entropy, Olivieri} 
in {\it Cram\'er's Theorem} central in large deviations theory \cite{ellis2006entropy}. A more general form of $\kappa_P$  appears in   {\it Sanov's Theorem}.} $\psi_P^*$ and $\kappa_P$
was previously established  under various assumptions: the authors of \cite[Theorem 5.2]{donsker1976asymptotic} (see also \cite{ellis2006entropy})  impose the (restrictive) assumption that $\psi_P$ is finite. On the other hand, the results in \cite[Theorem 6.17]{brown1986fundamentals} and \cite[Proposition 1]{le1999new} (see also \cite{ben1977dual} and a closely related result in \cite[Theorem 3.4]{wainwright2008graphical}) do not address  
the challenging case when $y$ resides on the boundary of the domain. This scenario turns out to be important if (and only if) the reference distribution is defined over a countable set. 
%Thus, without addressing the boundary case, the equivalence between the two functions is not fully established. 
Here, we provide complete proof that overcomes these assumptions previously imposed. Our approach emphasizes the role played by the convex support of the reference distribution and leads to natural and easy-to-verify conditions.  To this end, we will first need to examine the domains $\dom \kappa_P$ and $\dom \psi_P^*$. For Cram\'er's rate function $\psi_P^*$, a characterization of the domain is summarized in the following proposition.

\begin{proposition}[Domain of the Cram\'er rate function $\psi_P^*$ {\cite[Theorems 9.1, 9.4 and 9.5]{barndorff1978information}}]\label{mem:prop:Cramer_domain}
%	Let $\FFF_P$ be a minimal and steep exponential family generated by $P\in\PPP(\Omega)$. 
	Let $P\in\PPP(\Omega)$ be a reference distribution satisfying \cref{mem:asmp:blanket_minimal_and_steep}. Then, %$\dom \psi_P^*$, the domain of the Cram\'er rate function $\psi_P^*:\real^d\rightarrow\erl$ satisfies 
	$\inte\Omega_P^{cc}\subseteq\dom \psi_P^*\subseteq \Omega_P^{cc}$. Moreover, the following hold: 
	\begin{enumerate}%[(a)]
		\item[(a)] If $\Omega_P$ is finite, then $\dom \psi_P^*=\Omega_P^{cc}$.
		\item[(b)] If $\Omega_P$ is countable, then $\dom \psi_P^*\supseteq\conv\Omega_P$. 
		\item[(c)] If $\Omega_P$ is uncountable, then $\dom \psi_P^*=\inte\Omega_P^{cc}$.
	\end{enumerate}
\end{proposition}
\noindent In order to establish a similar characterization for the domain of the MEM function, we will need to make precise the relation between $\Omega_P$ and the expected value $\mathbb{E}_P$ for a given probability measure $P\in\PPP(\Omega)$. To this end, we first recall some additional definitions and results (see, for example, \cite[Section 6]{rockafellar1970convex}).
%We now turn to proving some useful results that will allow us to establish a similar characterization for the domain of the MEM function. \\ 
%
%We will use the following well-known definitions. Recall that the \emph{orthogonal projection} of a vector $x\in\real^d$ onto a nonempty closed set $S\subset\real^d$ is defined as 
%\begin{gather*}
%	P_S(x):=\argmin\{\|y-x\|:y\in S\}, 
%\end{gather*}
%and for any linear subspace $\LLL$ it holds that $v=P_{\LLL}(v)+P_{\LLL^{\perp}}(v)$, where $\LLL^\perp$ stands for the orthogonal complement of $\LLL$. As a consequence 
%\begin{gather}
%	\label{mem:eq:inner_product_decomposition}
%	\inner{v}{z} = \inner{P_{\LLL}(v)+P_{\LLL^{\perp}}(v)}{z} = \inner{P_{\LLL}(v)}{z} \qquad \forall z\in L.
%\end{gather} 
Consider two subsets $S,\hat S\subseteq\real^d$ and assume further that $S\subseteq\hat S$. Then  
	$\cl{S}\subseteq\cl{\hat S}, \inte{S}\subseteq\inte{\hat S}$ and $\conv{S}\subseteq\conv{\hat S}.$ \withsmallskip

Denote the closed Euclidean unit ball in $\real^d$ by $\BBB_d$. The \emph{relative interior} \cite[Section 6]{rockafellar1970convex} of a  convex set $S\subseteq\real^d$ is defined as 
\begin{gather*}
	\ri S :=\left\{x\in\real^d:\exists\tau>0 \text{ such that }(x+\tau\BBB_d)\cap\aff S\subseteq S\right\}.
\end{gather*}
E.g.,  for  the \emph{unit simplex} $\Delta_d:=\{y\in\real^d_+:\inner{e}{y}=1\}$ we have $\ri \Delta_d:=\{y\in\real^d_{++}:\inner{e}{y}=1\}$.
Some facts which will be used in the sequel are summarized in the following lemma. Further details and proofs can be found in \cite[Section 6, Theorem 13.1]{rockafellar1970convex}.
\begin{lemma}[On the relative interior]
	\label{mem:lem:on_ri}
	Let $S\subseteq\real^d$ be nonempty and convex. Then:
	\begin{enumerate}%[(a)]
		\item[(a)] It holds that $\ri(\cl{S})=\ri{S}$ and $\ri{S}\subseteq S\subseteq \cl{S}$.
		\item[(b)] If $\dim{S}=d$ then $\ri S=\inte S$ and, in particular, $\inte S\neq \emptyset$.
		\item[(c)] It holds that $x\in\ri S$ if and only if ${\sigma_{S-x}(v)\geq 0}$ where the last inequality is strict for every $v\in\real^d$ such that $-\sigma_{S}(-v)\neq \sigma_S(v)$.
	\end{enumerate}
\end{lemma}

\begin{lemma}[Domain of expected value]\label{mem:lem:expected_val_in_ri}
	Let $P\in\PPP(\Omega)$ and assume that $\mathbb{E}_P$ exists.  
Then $\mathbb{E}_P\in \ri \Omega_P^{cc}=\ri(\conv{\Omega_P})$.
\end{lemma}

\begin{proof}
	By definition of $\sigma_{\Omega_P}$, for any $v\in\real^d$, it holds that
$
		\label{mem:eq:expected_value_in_relint_proof1}
		-\sigma_{\Omega_P}(-v)\leq\inner{v}{y}\leq\sigma_{\Omega_P}(v).
	$
	As $P\in\PPP(\Omega)$, this  implies, for all $v\in\real^d$, that
	\begin{gather}
		\label{mem:eq:expected_value_in_relint_proof2}
		\inner{v}{\mathbb{E}_P} = \int_{\Omega_P}\inner{v}{y}dP(y)\leq \sigma_{\Omega_P}(v)\int_{\Omega_P}dP(y)=\sigma_{\Omega_P}(v).
	\end{gather} 
	If there exists some subset $A\subseteq\Omega_P$ such that %$\inner{v}{y}<\sigma_{\Omega_P}(v)$ for all $y\in A$, then 
	% \begin{gather*}
		$P(\{y\in A:\inner{v}{y}<\sigma_{\Omega_P}(v)\})>0$,
	% \end{gather*}
	then the inequality in \eqref{mem:eq:expected_value_in_relint_proof2} is strict. We will show that, for any $v\in\real^d$ such that $-\sigma_{\Omega_P}(-v)\neq \sigma_{\Omega_P}(v)$, such a subset exists;  the desired result then follows from  \cref{mem:lem:on_ri} (c) and the equivalence $\sigma_{\Omega_P^{cc}}(v)=\sigma_{\Omega_P}(v)$ \cite[Theorem 8.24]{rockafellar2009variational}. %In order to see that such a subset exists, observe first that
	Indeed, let $v\in\real^d$ such that $-\sigma_{\Omega_P}(-v)\neq \sigma_{\Omega_P}(v)$, i.e.
	% \begin{gather}
		% \label{mem:eq:expected_value_in_relint_proof3}
		$
		-\sigma_{\Omega_P}(-v)<\sigma_{\Omega_P}(v).$
	% \end{gather}
	Pick $\tau\in(-\sigma_{\Omega_P}(-v),\sigma_{\Omega_P}(v))$ and consider $A=\{y\in\Omega_P:\inner{v}{y}\leq\tau\}$. 
	As  $\tau<\sigma_{\Omega_P}(v)$, we have $A\subset \{y\in\Omega_P:\inner{v}{y}<\sigma_{\Omega_P}(v)\}$,   and 
	\begin{gather*}
		P(A) = P(\{y\in\Omega_P:\inner{-v}{y}\geq-\tau\})= P(\{y\in\Omega_P:\sigma_{\Omega_P}(-v)\geq\inner{-v}{y}\geq-\tau\})>0,
	\end{gather*}
	where the strict inequality follows from the definition of $\sigma_{\Omega_P}(-v)$ and $\sigma_{\Omega_P}(-v)>-\tau$. %Since $\tau<\sigma_{\Omega_P}(v)$ then $A\subset \{y\in\Omega_P:\inner{v}{y}<\sigma_{\Omega_P}(v)\}$ and 
	Hence, $A$ satisfies the desired conditions, which establishes the result.
\end{proof}
\noindent We are now in a position to present and prove a characterization for the domain of the MEM function, analogous to \cref{mem:prop:Cramer_domain}. We will use the following notation
\begin{gather*}
	\QQQ_P(y):=\{Q\in\PPP(\Omega):\mathbb{E}_Q=y,~Q\ll P\}. %\EE_Q[Y]=y,~Q\ll P\}, 
\end{gather*}
%where $Y$ is a random variable associated with the probability measure $Q$. 
Observe that $y\in\dom\kappa_P$ if and only if $\QQQ_P(y)\neq\emptyset$.

\begin{lemma}[Domain of the MEM function $\kappa_P$]
	\label{mem:lem:mem_domain}
	Let $P\in\PPP(\Omega)$ be a reference distribution satisfying \cref{mem:asmp:blanket_minimal_and_steep}.
	Then:
	\begin{enumerate}%[(a)]
		\item[(a)] If $\Omega_P$ is countable, then $\dom \kappa_P=\conv\Omega_P$. Hence, if $\Omega_P$ is finite, then ${\dom \kappa_P=\Omega_P^{cc}}$.
		\item[(b)] If $\Omega_P$ is uncountable, then $\dom \kappa_P=\inte{\Omega_P^{cc}}$.
	\end{enumerate}
\end{lemma}

\begin{proof}
%	We will prove the two cases separately.
	\begin{enumerate}%[(a)]
		\item[(a)] 
		Let $y\in \dom \kappa_P$, hence there exists $Q\in\QQQ_P(y)$. 
		As $Q\ll P$, we obtain $\Omega_Q\subseteq\Omega_P$, thus $\conv\Omega_Q\subseteq\conv\Omega_P$. Hence, by \cref{mem:lem:on_ri} (a) and \cref{mem:lem:expected_val_in_ri}, we know that $y=\mathbb{E}_Q\in \ri\Omega_Q^{cc}\subseteq \conv\Omega_Q\subseteq\conv\Omega_P$. %$y=\EE_Q[Y]\in \ri\Omega_Q^{cc}\subseteq \conv\Omega_Q\subseteq\conv\Omega_P$. 
		%Due to this arguments, we can conclude that, for all cases under consideration, the following inclusion holds true
		Thus,
%		\begin{gather*}
			%\label{mem:eq:dom_mem_in_conv_support}
			${\dom \kappa_P \subseteq \conv\Omega_P}$.
%		\end{gather*}
		%We will prove that $\conv\Omega_P\subseteq \dom\kappa_P$. Then, the result will follow from \eqref{mem:eq:dom_mem_in_conv_support}. 
		For the converse inclusion, let ${y\in \conv\Omega_P}$. By Carath\'eodory's theorem \cite{caratheodory1911variabilitatsbereich}, there exist $n\leq d+1$ points $p_1,\dots,p_{n}$ in $\Omega_P$ such that
%		\begin{gather*}
$			y = \sum_{i=1}^{n}\lambda_ip_i$
%		\end{gather*}
		for some $\lambda\in \Delta_{n}$. Consider a distribution $Q\in\PPP(\Omega)$ satisfying 
		$Q(\{p_i\})=\lambda_i$ for all $i=1,\dots,n$. 
		Then, $Q\in\QQQ_P(y)$ %$y=E_Q$ and $Q\ll P$ 
		by construction. Thus, $y\in \dom \kappa_P$, and we can conclude that $\conv\Omega_P\subseteq\dom \kappa_P$.
		
		\item[(b)] %Recall that by Remark \ref{mem:rmrk:expected_val_in_inte_for_uncountable}\\
%		We will first show that $\dom \kappa_P\subseteq \inte\Omega_P^{cc}$. Let $y\in \dom \kappa_P$, then there exists $Q\in\QQQ_P(y)$. According to \eqref{mem:eq:dom_mem_in_conv_support} we know that $y\in \conv\Omega_P\subseteq\Omega_P^{cc}$. Hence, in order to complete the proof we need to show that $y\notin\bdr(\Omega_P^{cc})=\Omega_P^{cc}\setminus \inte\Omega_P^{cc}$. %, i.e., $y$ does not reside on the relative boundary of $\Omega_P^{cc}$. 
%		Assume, by contradiction, that $y\in\bdr\Omega_P^{cc}$. In this case $\dim \Omega_Q^{cc}<d$ as otherwise 		
		First, let $y\in \dom \kappa_P$, then there exists $Q\in\QQQ_P(y)$. Since $Q\ll P$ which satisfies \cref{mem:asmp:blanket_minimal_and_steep}, it holds that $\dim \Omega_Q^{cc}=\Omega_P^{cc}=d$. %\dim \Omega_Q=\dim \Omega_P=d$. 
		Otherwise, the probability measure $Q$ ($Q(\Omega_Q)=1$) is concentrated on a lower dimensional affine subspace in contradiction to the absolute continuity of $Q$ with respect to $P$. Hence, using \cref{mem:lem:expected_val_in_ri} and \cref{mem:lem:on_ri} (b), we obtain that
		$y=\mathbb{E}_Q\in\ri{\Omega_Q^{cc}} = \inte{\Omega_Q^{cc}}\subseteq\inte{\Omega_P^{cc}}$.
		For the  converse inclusion,  by \cref{mem:prop:Cramer_domain}, $y\in \inte\Omega_P^{cc}=\dom{\psi_P^*}=\idom{\psi_P^*}=\dom\nabla \psi_P^*$, and we  conclude that $y=\mathbb{E}_{P_\theta}$ for $\theta = \nabla \psi_P^*(y)$. %and $Y_\theta$ is a random variable associated with the probability measure $P_\theta$. 
		Since $P_\theta\ll P$ for $P_\theta$ from the exponential family generated by $P$, we find that $P_\theta\in\QQQ_P(y)$ and therefore $y\in\dom\kappa_P$.
	\end{enumerate}
\end{proof}
\noindent Combining \cref{mem:lem:mem_domain} with \cref{mem:prop:Cramer_domain} yields the following corollary.

\begin{corollary}\label{mem:cor:equiv_domain_of_mem_and_cramer}
	%Let $\FFF_P$ be a minimal and steep exponential family generated by $P\in\PPP(\Omega)$. 
	Let $P\in\PPP(\Omega)$ be a reference distribution satisfying \cref{mem:asmp:blanket_minimal_and_steep}. Then, 
	\begin{enumerate}%[(a)]
		\item[(a)] If $\Omega_P$ is countable and $\conv\Omega_P$ is closed (i.e., $\conv\Omega_P=\Omega_P^{cc}$), then $\dom \kappa_P=\dom \psi_P^*=\Omega_P^{cc}$. In particular, $\dom \kappa_P=\dom \psi_P^*=\Omega_P^{cc}$ if $\Omega_P$ is finite.
		\item[(b)] If $\Omega_P$ is uncountable, then $\dom \kappa_P = \dom \psi_P^* = \inte \Omega_P^{cc}$.		
	\end{enumerate}		
	%	and assume that $\Omega_P$ is countable. Then, if $\conv\Omega_P$ is closed (i.e., $\conv\Omega_P=\Omega_P^{cc}$), then $\dom \psi_P^*=\Omega_P^{cc}$. In particular, $\dom \psi_P^*=\Omega_P^{cc}$ if $\Omega_P$ is finite.
\end{corollary}
\noindent The following lemma will be crucial for proving the equivalence between the MEM function $\kappa_P$ and Cram\'er's rate function $\psi_P^*$. The proof of the lower bound follows similar arguments as in \cite[Theorem 6.17]{brown1986fundamentals} and \cite[Proposition 1]{le1999new} and we include it here for completeness.

\begin{lemma}
	\label{mem:lem:mem_bounds}
	Let $P\in\PPP(\Omega)$ be a reference distribution satisfying \cref{mem:asmp:blanket_minimal_and_steep}. Then:
	\begin{gather*}
	\psi_P^*(y)\leq\kappa_P(y)\leq\psi_P^*(y) + \KL{Q}{P_\theta}-D_{\psi_P^*}\left(y,\nabla \psi_P(\theta)\right),%\quad (Q\in \QQQ_P(y),~\theta\in\Theta_P).
\end{gather*}
for any $y\in\dom\kappa_P$, $Q\in \QQQ_P(y)$ and $\theta\in\inte\Theta_P$.
\end{lemma}

\begin{proof}	
	For any $\theta\in\inte\Theta_P$ and $Q\in \QQQ_P(y)$ we obtain that $Q\ll P_\theta$ due to the mutual absolute continuity between $P_\theta$ and $P$. Hence, 
	\begin{multline}
		\label{mem:eq:KL_decomposition}
		% \def\arraystretch{2}
		% \begin{array}{rl}
			\KL{Q}{P} \displaystyle= \int_{\Omega} \log\left(\frac{dQ}{dP}\right)dQ = 	\int_{\Omega} \log\left(\frac{dQ}{dP_\theta}\right)dQ+\int_{\Omega} \log\left(\frac{dP_\theta}{dP}\right)dQ\\
			= \KL{Q}{P_\theta}+\int_{\Omega} [\inner{z}{\theta}-\psi_P(\theta)]dQ(z)
			= \KL{Q}{P_\theta}+\inner{y}{\theta}-\psi_P(\theta),
		% \end{array}
	\end{multline}
	% \begin{gather}
	% 	\label{mem:eq:KL_decomposition}
	% 	\def\arraystretch{2}
	% 	\begin{array}{rl}
	% 		\KL{Q}{P} &\displaystyle= \int_{\Omega} \log\left(\frac{dQ}{dP}\right)dQ \\ &\displaystyle= 	\int_{\Omega} \log\left(\frac{dQ}{dP_\theta}\right)dQ+\int_{\Omega} \log\left(\frac{dP_\theta}{dP}\right)dQ\\
	% 		&\displaystyle= \KL{Q}{P_\theta}+\int_{\Omega} [\inner{z}{\theta}-\psi_P(\theta)]dQ(z)\\
	% 		&\displaystyle= \KL{Q}{P_\theta}+\inner{y}{\theta}-\psi_P(\theta),
	% 	\end{array}
	% \end{gather}	
	where the last identity uses $y=\mathbb{E}_Q$. Since  \eqref{mem:eq:KL_decomposition} holds for all $\theta\in\inte\Theta_P$ and $\KL{Q}{P_\theta}\geq 0$, 
	\begin{gather}
		\label{mem:eq:lower_bound}
		\KL{Q}{P}\geq \sup\{\inner{y}{\theta}-\psi_P(\theta):\theta\in \inte\Theta_P\}= \psi_P^*(y),
		% \def\arraystretch{2}
		% \begin{array}{rl}
		% \KL{Q}{P} &= \sup\{\KL{Q}{P_\theta}+\inner{y}{\theta}-\psi_P(\theta):\theta\in \inte\Theta_P\}\\
		% &\geq \sup\{\inner{y}{\theta}-\psi_P(\theta):\theta\in \inte\Theta_P\}= \psi_P^*(y).%\\
		% %  &= \psi_P^*(y).
		% \end{array}
	\end{gather}
	due to the closedness of $\psi_P$, see \cref{pre:prop:exp_fam_convexity}. 
	%Observe that due to the steepness of $\psi_P$ the optimization problem above can be solved over $\Theta_P$ instead of $\inte\Theta_P$ which explains the equivalence to $\psi_P^*$ (see \cite[Lemma 5.4]{brown1986fundamentals}).
	The lower bound for $\kappa_P$ follows immediately from its definition and the above inequality.\withsmallskip
	
	As for the upper  bound: by \eqref{mem:eq:KL_decomposition} and \eqref{pre:eq:fechel_in_equality_equivs},  for any $Q\in\QQQ_P(y)$ and $\theta\in\inte\Theta_P$, we have
	\begin{gather*}
		\def\arraystretch{1.5}
		\begin{array}{rl}
			\KL{Q}{P} &\displaystyle= \KL{Q}{P_\theta}+\inner{y}{\theta}-\psi_P(\theta)\\
			&\displaystyle= \KL{Q}{P_\theta}+\inner{y-\nabla\psi_P(\theta)}{\theta}+\inner{\nabla\psi_P(\theta)}{\theta}-\psi_P(\theta)\\
			&\displaystyle= \KL{Q}{P_\theta}-\left[
			\psi_P^*(y)-\psi_P^*(\nabla \psi_P(\theta))-\inner{y-\nabla\psi_P(\theta)}{\theta}\right]+\psi_P^*(y)\\
			&\displaystyle= \KL{Q}{P_\theta}-D_{\psi_P^*}\left(y,\nabla \psi_P(\theta)\right)+\psi_P^*(y).
		\end{array}
	\end{gather*}
	Then the result follows due to the fact that $\kappa_P(y)\leq \KL{Q}{P}$ for all $Q\in\QQQ_P(y)$. %and the nonnegativity of the Bregman distance.
\end{proof}

\begin{theorem}[Equivalence between Cram\'er's rate function and the MEM function]
	\label{mem:thrm:Cramer_and_MEM_equiv}
%	Let $\FFF_P$ be a minimal and regular exponential family generated by $P\in\PPP(\Omega)$. 
	Let $P\in\PPP(\Omega)$ satisfy \cref{mem:asmp:blanket_minimal_and_steep}, and  assume that one of the following two conditions holds:
	\begin{enumerate}%[(i)]
		\item[(i)] $\Omega_P$ is uncountable.		
		\item[(ii)] $\Omega_P$ is countable and $\conv\Omega_P$ is closed (as is the case when $\Omega_P$ is finite).
	\end{enumerate}
	Then, $\kappa_P=\psi_P^*$. In particular, $\kappa_P$ is closed, proper, and convex.
\end{theorem}

\begin{proof}
	First, let $y\in \inte\Omega_P^{cc}$. By \cref{mem:asmp:blanket_minimal_and_steep}, $\nabla \psi_P$ is a bijection between 
	$\idom{\psi_{P}}$ $=\inte{\Theta_P}$ and $\idom{\psi_{P}^*}=\inte\Omega_P^{cc}$, where the latter uses \cref{mem:prop:Cramer_domain}. Thus, there exists $\theta \in \inte\Theta_P$ such that $y=\nabla \psi_P(\theta)=\mathbb{E}_{P_\theta}$. 
	Applying \cref{mem:lem:mem_bounds} with $Q=P_\theta$ yields
	\begin{gather}
		\label{mem:eq:mem_is_cramer_in_interior}
		\kappa_P(y) = \psi_P^*(y)\qquad (y\in\inte \Omega_P^{cc}).
	\end{gather}
	\noindent Due to \cref{mem:cor:equiv_domain_of_mem_and_cramer}, this establishes the result when $\Omega_P$ is uncountable. To complete the proof, we only need to address the case when $y\in \bdr\Omega_P^{cc}$ under assumption (ii).
	By \cref{mem:cor:equiv_domain_of_mem_and_cramer}, in this case $\dom \kappa_P=\dom \psi_P^*=\Omega_P^{cc}$ and $\QQQ_P(y)\neq\emptyset$ for $y\in\bdr\Omega_P^{cc}$. Consider any $Q\in\QQQ_P(y)$, then, by definition of $\kappa_P$, we have that
	\begin{gather}
		\label{mem:eq:kappa_is_inf}
		\kappa_P(y)\leq \KL{Q}{P}<+\infty.
	\end{gather}
	Choose any $\hat y\in \inte \Omega_P^{cc}$ and set $\hat \theta=\nabla\psi_P^*(\hat y)$ (i.e., $\hat y=\nabla\psi(\hat \theta)$). For any $\lambda\in[0,1)$ consider $Q_\lambda = \lambda Q+(1-\lambda)P_{\hat\theta}$. % and the associated random variable $Y_\lambda$. 
	Then, by  linearity of $Q\mapsto \mathbb E_Q$ \cite[Lemma 2]{rioux2020maximum}, 
	we obtain   
	\begin{gather*}
		y_\lambda:=\mathbb{E}_{Q_\lambda}%=E_{\lambda Q+(1-\lambda)P_{\hat\theta}} 
		= \lambda \mathbb{E}_Q+(1-\lambda) \mathbb{E}_{P_{\hat\theta}}=\lambda y+(1-\lambda)\hat y.
	\end{gather*}
	By convexity of $\Omega_P^{cc}$ and the line segment principle \cite[Lemma 6.28]{beck2014introduction} we  conclude that ${y_\lambda\in\inte\Omega_P^{cc}}$. Set $\theta_\lambda:=\nabla\psi_P^*(y_\lambda)$ and observe that, by \cref{mem:lem:mem_bounds} and the nonnegativity of the Bregman distance, it holds that
	\begin{gather}
		\label{mem:eq:sandwich}
		\psi_P^*(y)\leq\kappa_P(y)%\leq\psi_P^*(y) + \Delta(Q,\theta_\lambda)
		\leq \psi_P^*(y) + \KL{Q}{Q_\lambda}.
	\end{gather}
	In addition, due to \eqref{mem:eq:kappa_is_inf} and the fact that  $Q\ll P\ll P_{\hat\theta}$, we  conclude that ${\KL{Q}{P_{\hat\theta}}<\infty}$. Thus, by \eqref{mem:eq:sandwich} and convexity of $\KL{Q}{\cdot}$, we obtain
	\begin{gather*}
		\KL{Q}{Q_\lambda}\leq \lambda\KL{Q}{Q}+(1-\lambda)\KL{Q}{P_{\hat\theta}}\rightarrow 0\quad \text{as}\quad\lambda \rightarrow 1.
	\end{gather*}
\end{proof}

\noindent
We refer to a solution of the optimization problem \eqref{mem:eq:mem} as the \emph{MEM distribution} and denote it as $Q_{\scriptscriptstyle MEM}$. By similar arguments to the ones used in order to establish the lower bound in \cref{mem:lem:mem_bounds}, one can show that, when $y\in \inte(\dom \kappa_P)=\inte (\conv\Omega_P)$,  the MEM distribution is a particular member of the exponential family generated by the reference distribution $P$. More precisely, it holds that $Q_{\scriptscriptstyle MEM}=P_{\theta}$ where $\theta=\nabla\psi_P^*(y)$ and consequently
\begin{gather*}
	f_{Q_{\scriptscriptstyle MEM}}(x) =\frac{dP_\theta}{dP}(x) =  \exp\left(\inner{x}{\theta}-\log\int_{\Omega}\exp(\inner{\cdot}{\theta})dP\right)= \frac{\exp(\inner{x}{\theta})}{\int_{\Omega}\exp(\inner{\cdot}{\theta})dP}.
\end{gather*}
\noindent 
This, again, highlights the intimate connection between the MEM function and exponential families.
The case $y\in\bdr(\dom \kappa_P)$ is more subtle and will be the topic of future research. \withsmallskip

In what follows, we  assume that the reference distribution of the MEM function satisfies the conditions stated in \cref{mem:thrm:Cramer_and_MEM_equiv}, that is:

\begin{assumption}\label{mem:asmp:Cramer_and_MEM_equiv_cond}
	The  distribution $P\in\PPP(\Omega)$ 
	%defining the MEM function \eqref{mem:eq:mem} %generates a minimal and steep exponential family $\FFF_P$ and 
	satisfies one of the following conditions:	
	\begin{enumerate}%[(i)]
		\item[(i)] $\Omega_P$ is uncountable.		
		\item[(ii)] $\Omega_P$ is countable and $\conv\Omega_P$ is closed (as is the case when $\Omega_P$ is finite).
	\end{enumerate}
\end{assumption}

Under \cref{mem:asmp:blanket_minimal_and_steep,mem:asmp:Cramer_and_MEM_equiv_cond},
the MEM function and the Cram\'er rate function coincide. 
As an immediate consequence, we obtain that the MEM function $\kappa_P$ is of Legendre type. More importantly, we will see that the alternative representation by means of Cram\'er's rate function is more tractable compared to the original definition given in \eqref{mem:eq:mem}.

\begin{theorem}[Properties of the MEM function]
	\label{mem:thrm:MEM_properties}
	Let $P\in\PPP(\Omega)$ satisfy  \cref{mem:asmp:blanket_minimal_and_steep,mem:asmp:Cramer_and_MEM_equiv_cond}.
	Then the following hold:
	\begin{enumerate}%[(a)]
		\item[(a)] $\kappa_{P}(y)\geq 0$ and equality holds if and only if $y=\mathbb{E}_P$.
		\item[(b)] $\kappa_P$ is of Legendre type.
		\item[(c)] $\kappa_P$ is coercive in the sense that $\lim_{\|y\|\to\infty} \kappa_P(y) =+\infty$ \cite[Definition 11.10]{bauschke2011convex}.
		% \begin{gather*}
		% 	\lim_{\|y\|\to\infty} \kappa_P(y) =+\infty.
		% 	%\liminf_{\|y\|\to\infty} \frac{\kappa_P(y)}{\|y\|} > 0.
		% 	%		\displaystyle{\lim\inf}_{\|x\|\to\infty} \frac{\kappa_P(x)}{\|x\|} = +\infty.
		% \end{gather*}
		In particular, $\kappa_P(y)$ is level bounded.
		\item[(d)] If $M_P$ is finite (which holds, in particular, when $\Omega_P$ is bounded), then $\kappa_P$ is supercoercive in the sense that $\lim_{\|y\|\to\infty} \kappa_P(y)/\|y\| =+\infty$ \cite[Definition 11.10]{bauschke2011convex}.
		% \begin{gather*}
		% \lim_{\|y\|\to\infty} \frac{\kappa_P(y)}{\|y\|} =+\infty.
		% \end{gather*}		 
	\end{enumerate}
\end{theorem}

\begin{proof}
	Part (a) is evident from the definition of $\kappa_P$ as given in \eqref{mem:eq:mem} and \cite[Proposition 6.2]{brown1986fundamentals}. Part (b) follows directly from the equivalence to the Cram\'er rate function $\psi_P^*$ and \cref{pre:cor:Legendre}. To see (c), observe that (a) implies that $\kappa_P$ admits a unique minimizer $\mathbb{E}_P$ which combined with the fact that $\kappa_P$ is closed, proper and convex (since $\kappa_P$ is of Legendre type due to (b)) establishes the result by \cite[Proposition 3.1.3]{auslender2006asymptotic}. Lastly, if the moment generating function is finite, then so is $\psi_P$, and the supercoercivity of $\kappa_P=\psi_{P}^*$ follows from \cite[Theorem 11.8(d)]{rockafellar2009variational}.\footnote{The definition of supercoercive convex functions we use here follows \cite[Definition 11.10]{bauschke2011convex}. In \cite{rockafellar2009variational} the authors refer to such functions as coercive (see \cite[Definition 3.25]{rockafellar2009variational}).} If $\Omega_P$ is bounded then $\dom{\kappa_P}$ is bounded due to \cref{mem:lem:mem_domain}. In this case, $\kappa_P=\psi_P^*$ is trivially supercoercive and the claim that $\psi_P$ is finite follows from \cite[Theorem 11.8(d)]{rockafellar2009variational}.
\end{proof}

\noindent The results presented in the remainder of this work are established under \cref{mem:asmp:blanket_minimal_and_steep,mem:asmp:Cramer_and_MEM_equiv_cond} which, in particular, ensure the equivalence between the MEM and Cram\'er rate functions. 
For this reason, we take this opportunity to standardize our nomenclature: between the two options ($\kappa_P$ or $\psi_P^*$) we will opt for the one that corresponds to the Cram\'er rate function $\psi_P^*$. This choice is motivated by our intent to emphasize the more computationally appealing definition and the connection to the log-normalizer function $\psi_P$. Nevertheless, in the definition of some new concepts defined by means of Cram\'er's rate function, we will adopt the MEM terminology in order to emphasize the motivation in the context of estimation.\withsmallskip

If the reference distribution belongs to an exponential family generated by some measure $P\in\MMM(\Omega)$, i.e., if for some $\hat\theta\in\Theta_P$ we consider a new exponential family generated by the probability measure $P_{\hat\theta}$,\footnote{Recall from the definition of $\FFF_P$ that % according to \eqref{pre:eq:exp_family} 
$P_{\hat\theta}$ is the probability measure with 
%\begin{gather*}
	$\frac{dP_{\hat\theta}}{dP}(y) = \exp(\inner{y}{\hat\theta}-\psi_P(\hat\theta))$.
%\end{gather*}
} then the corresponding moment-generating function takes the form 
\begin{gather}
	\label{mem:eq:moment_generating_of_exp_fam}
	M_{P_{\hat \theta}}[\theta]=\exp\left(\psi_P(\hat\theta+\theta)-\psi_P(\hat\theta)\right).
\end{gather} 
\noindent In this case, the Cram\'er rate functions that corresponds to $P_{\hat\theta}$ and $P$ share a useful relation summarized in the following lemma. %We include the simple proof in \cref{apndx:sec:proofs_and_tables}.

\begin{lemma}
	\label{mem:lem:relative_kramer}
	Let $\FFF_P$ be a minimal and steep exponential family generated by ${P\in\MMM(\Omega)}$ and assume further that, for any $\theta\in\inte\Theta_P$, \cref{mem:asmp:Cramer_and_MEM_equiv_cond} holds  for $P_\theta\in\PPP(\Omega)$. 
	Then, for any $\hat{\theta}\in\inte\Theta_P$ and $y\in\dom{\psi_P^*}$,  we have
	$\psi^*_{P_{\hat{\theta}}}(y) = D_{\psi_P^*}(y,\hat{y})$
	% \begin{gather*}
	% 	\psi^*_{P_{\hat{\theta}}}(y) = D_{\psi_P^*}(y,\hat{y})\qquad  (y\in\dom{\psi_P^*}),
	% \end{gather*}
	where $\hat{y}:=\nabla \psi_P(\hat{\theta})\in \inte\Omega_P^{cc}$.
\end{lemma}

\begin{proof}%[Proof (for \cref{mem:lem:relative_kramer})]
	For $y\in \dom \psi_P^*$, we have
		\begin{gather*}
			\def\arraystretch{2}
			\begin{array}{rl}
				\psi^*_{P_{\hat{\theta}}}(y) &\overset{\eqref{mem:eq:log_Laplace}}{=}\sup\left\{\inner{y}{\theta}-\log\left(M_{P_{\hat\theta}}[\theta]\right):\theta\in \real^d\right\}\\
				&\overset{\eqref{mem:eq:moment_generating_of_exp_fam}}{=}\sup\left\{ \inner{y}{\theta}-[\psi_P(\hat{\theta}+\theta)-\psi_P(\hat{\theta})]:\theta\in \real^d \right\}\\
			%	&=\sup\left\{\inner{y}{\hat{\theta}+\theta}-\psi_P(\hat{\theta}+\theta):\theta\in \real^d \right\}+\psi_P(\hat{\theta})-\inner{y}{\hat{\theta}}\\
				&=\psi^*_P(y)+\psi_P(\hat{\theta})-\inner{y}{\hat{\theta}}.
			\end{array}
		\end{gather*}
	The result follows from the definition of the Bregman distance, \eqref{pre:eq:fechel_in_equality_equivs} and  $\hat{\theta}\in\idom{\psi_P}$.
\end{proof}

\noindent We list in \cref{mem:tbl:cramer}  below a number of examples of Cram\'er rate functions that correspond to most of the popular distributions (i.e. choices of the reference distribution $P\in\PPP(\Omega)$). Some of the functions admit a closed form expression while others are given implicitly.\footnote{One can evaluate Cram\'er's rate function value at a point of interest by solving a nonlinear system.} The derivations and further details are included in \cref{apndx:sec:Cramer}. %as a supplementary material. 
Observe that all cases considered below satisfy \cref{mem:asmp:blanket_minimal_and_steep,mem:asmp:Cramer_and_MEM_equiv_cond} which guarantees the equivalence established in \cref{mem:thrm:Cramer_and_MEM_equiv}:  indeed, with some exceptions, all the distributions in \cref{mem:tbl:cramer} are minimal with a natural parameter space $\Theta_P$ open which implies steepness. These  exceptions are: the multinomial distribution which is minimal under an appropriate reformulation and the multivariate normal-inverse Gaussian which is steep (see \cref{apndx:sec:Cramer}%supplementary material
). Here, we provide the Cram\'er rate function of the multinomial distribution in minimal form. Thus, \cref{mem:asmp:blanket_minimal_and_steep} holds for all the distributions given in \cref{mem:tbl:cramer}. This comprehensive list complements and extends some previously established formulas  \cite{le1999new, wainwright2008graphical}. \withsmallskip

Many computations are facilitated in the presence of separability as described in the following remark.

\begin{remark}[Separability of $\psi_P^*$]
\label{mem:rmrk:Cramer_seperability}
In most examples, the reference distribution $P\in\PPP(\Omega)$ admits a separable structure of the form $P(y)=P_1(y_1)P_2(y_2)\cdots P_d(y_d)$  where $P_i\in\PPP(\Omega_i)$,  $\Omega_i\subset\real$,  
i.e., each component corresponds to an i.i.d. random variable. 
In this case, since $\mathbb{M}_{P}[\theta]=\prod_{i=1}^{d}\mathbb{M}_{P_i}[\theta_i]$ \cite[Section 4.4]{rohatgi2015introduction}, 
% \begin{gather*}
% 	\mathbb{M}_{P}[\theta]=\prod_{i=1}^{d}\mathbb{M}_{P_i}[\theta_i],
% \end{gather*}
 we  have 
\[
	\psi^*_P(y) 
%	\sup\{\inner{y}{\theta}-\psi_P(\theta):\theta\in\real^d\} \\
	= \sup\left\{\inner{y}{\theta}-\log\left(\mathbb{M}_{P }[\theta]\right):\theta\in\real^d\right\}=\sum_{i=1}^d\sup\left\{y_i\theta_i-\log\left(\mathbb{M}_{P_i}[\theta_i]\right):\theta_i\in\real\right\}.
\]
Hence, in most of our examples below we will consider only the case $d=1$.\rqed
\end{remark}
 %In the following $\Delta_d:=\{y\in\real^d_+:\inner{e}{y}=1\}$ stands for the unit simplex and for a given vector $p\in\real^d$ we denote $I(p):=\{y\in\real^d:y_i=0~(p_i=0)\}$. 
 In \cref{mem:tbl:cramer} we employ the convention that $0\log(0)=0$ and define 
 \[
\Delta_{(d)}:=\left\{y\in\real^d_+:\sum_{i=1}^d y_i\leq 1\right\} \quad\text{and}\quad I(p):=\{y\in\real^d:y_i=0~(p_i=0)\}\quad (p\in  \real^d).
 \]

%\begin{table}[H]
{ 	 
% \fontsize{7}{10}\selectfont 
\small	
\centering
	\def\arraystretch{2}
%	\begin{tabular}{@{}lcc@{}}
	\begin{longtable}{@{}lcc@{}}		
		\toprule
		\multicolumn{1}{c}{Reference Distribution ($P$)}                                                                                                                           & Cram\'er Rate Function ($\psi_P^*(y)$)                                                                                                                                                                                                                                                                                                                               & $\dom \psi_P^*$                                                                                                                      \\ \midrule
		\endfirsthead
		\multicolumn{3}{@{}l}{\ldots continued}\\\toprule
				\multicolumn{1}{c}{Reference Distribution ($P$)}                                                                                                                           & Cram\'er Rate Function ($\psi_P^*(y)$)                                                                                                                                                                                                                                                                                                                                & $\dom \psi_P^*$                                                                                                                    \\ \midrule
		\endhead % all the lines above this will be repeated on every page
		\midrule%<--
		\multicolumn{3}{r@{}}{continued \ldots}%<--
		\endfoot%<--
		%\caption{Cram\'er rate functions for ubiquitous distributions.}%<--
		\endlastfoot%		
		{\color[HTML]{333333} \def\arraystretch{1}\begin{tabular}[c]{@{}l@{}}Multivariate Normal\\ $(\mu\in\real^d, \Sigma\in\bbS^d:\Sigma\succ 0)$\end{tabular}}                                & {\color[HTML]{333333} $\frac{1}{2}(y-\mu)^T\Sigma^{-1}(y-\mu)$}                                                                                                                                                                                                                                                                                                & $\real^d$                                                                                                                            \medskip\\
		%Chi-squared $(k\in\nn)$                                                                                                                                              & $\frac{1}{2}\left(y-k+k\log(k/y)\right)$                                                                                                                                                                                                                                                                                                                       & $\real_{++}$                                                                                                                         \medskip\\
		\def\arraystretch{1}\begin{tabular}[c]{@{}l@{}}Multivar. Normal-inverse Gaussian \\ 
		 $\big(\mu,\beta\in\real^d,~\alpha,\delta\in\real,\Sigma\in\real^{d\times d}$:\\ 
		$\delta>0,~\Sigma\succ0, \alpha\geq\sqrt{\beta^T\Sigma\beta}\big)$\\
		$\gamma:=\sqrt{\alpha^2-\beta^T\Sigma\beta}$\end{tabular} 
		
		& $\alpha\sqrt{\delta^2+(y-\mu)^T\Sigma^{-1}(y-\mu)}-\beta^T(y-\mu)-\delta\gamma$                                                                                                                                                                                                                                                                                                    & $\real^d$                                                                                                                              \medskip\\
		Gamma $(\alpha,\beta\in\real_{++})$                                                                                                                                 & $\beta y-\alpha+\alpha\log\left(\frac{\alpha}{\beta y}\right)$                                                                                                                                                                                                                                                                                                                                  & $\real_{++}$                                                                                                                         \medskip\\
		%Exponential $(\lambda\in\real_{++})$                                                                                                                                 & $\lambda y-1-\log(\lambda y)$                                                                                                                                                                                                                                                                                                                                  & $\real_{++}$                                                                                                                         \medskip\\
		Laplace $(\mu\in\real,b\in\real_{++})$                                                                                                                                       & \def\arraystretch{1}\begin{tabular}[t]{@{}c@{}}$\begin{cases} 0, & y=\mu, \\ \sqrt{1+\rho(y)^2}-1+\log\left(\frac{\sqrt{1+\rho(y)^2}-1}{\rho(y)^{2}/2}\right), & y\neq \mu, \end{cases}$\withsmallskip \\\withsmallskip $\left(\rho(y):=(y-\mu)/b\right)$\end{tabular}                                                                                                                                                                                                                      & $\real$                                                                                                                              \medskip\\
		Poisson $(\lambda\in\real_{++})$                                                                                                                                     & $y\log(y/\lambda)- y+ \lambda$                                                                                                                                                                                                                                                                                                                                 & $\real_{+}$                                                                                                                          \medskip\\
		\def\arraystretch{1}\begin{tabular}[c]{@{}l@{}}Multinomial $(n\in\nn, p\in\Delta_{(d)}$: \\
		$\sum_{i=1}^d p_i<1)$
		\end{tabular}

		& $\sum_{i=1}^{d} y_i\log\left(\frac{y_i}{np_i}\right)+ \left(n - \sum_{i=1}^d y_i\right)\log\left(\frac{n - \sum_{i=1}^d y_i}{n(1 - \sum_{i=1}^d p_i)}\right)$                                                                                                                                                                                                                                                                                                            & $n\Delta_{(d)}\cap I(p)$ \medskip\\
		\def\arraystretch{1}\begin{tabular}[c]{@{}l@{}}Negative Multinomial $(p\in[0,1)^d,$\\
			$y_0\in\real_{++},~p_0:=1-\sum_{i=1}^d p_i>0)$	\end{tabular}
		 &       $\sum_{i=0}^dy_i\log\left(\frac{y_i}{p_i\bar{y}}\right)$\quad$(\bar{y}:=\sum_{i=0}^dy_i)$                                                                                                                                                                                                                                                                                                                                                        &  $\real_{+}^d\cap I(p)$                                                                                                                                      \medskip\\
%		Degenerate $(a\in\real^d)$                                                                                                                                           & $\delta_{\{a\}}(y)$                                                                                                                                                                                                                                                                                                                                            & $\{a\}$                                                                                                                              \medskip\\
		\def\arraystretch{1}\begin{tabular}[c]{@{}l@{}}Discrete Uniform\\ $\big(a,b\in\intg: a\leq b, $ \\
			$~~\mu:=(a+b)/2,n:=b-a+1\big)$\\ \end{tabular}                                                                         & \def\arraystretch{1}\begin{tabular}[t]{@{}l@{}}$\begin{cases} 		0, & y=\mu,\\ 		(y-\mu)\theta-\log\left(\frac{e^{(b-\mu+1)\theta}-e^{(a-\mu)\theta}}{n(e^{\theta}-1)}\right), & y\neq\mu, 	\end{cases}$\medskip\\ where $\theta\in\real:~  y+\frac{e^{\theta}}{e^{\theta}-1}=\frac{(b+1)e^{(b+1)\theta}-ae^{a\theta}}{e^{(b+1)\theta}-e^{a\theta}}$ %is the root of:\withsmallskip\\ $\qquad\qquad  y-\frac{be^{b\theta}-ae^{a\theta}}{e^{b\theta}-e^{a\theta}}+\frac{1}{\theta}=0 $
		\end{tabular} & $[a,b]$                                                                                      
	\medskip\\
	\def\arraystretch{1}\begin{tabular}[c]{@{}l@{}}Continuous Uniform\\ $\big(a,b\in\real:a<b, \mu:=(a+b)/2\big)$\\ \end{tabular}                                                                         & \def\arraystretch{1}\begin{tabular}[t]{@{}l@{}}$\begin{cases} 		0, & y=\mu,\\ 		(y-\mu)\theta-\log\left(\frac{e^{(b-\mu)\theta}-e^{(a-\mu)\theta}}{(b-a)\theta}\right), & y\neq\mu, 	\end{cases}$\medskip\\ where $\theta\in\real:~  y+\frac{1}{\theta}=\frac{be^{b\theta}-ae^{a\theta}}{e^{b\theta}-e^{a\theta}}$ %is the root of:\withsmallskip\\ $\qquad\qquad  y-\frac{be^{b\theta}-ae^{a\theta}}{e^{b\theta}-e^{a\theta}}+\frac{1}{\theta}=0 $
	\end{tabular} & $(a,b)$                                                                                                                                                                     \medskip\\
		Logistic %$(\mu,s\in\real:s>0)$:
		$(\mu\in\real,s\in\real_{++})$                                                                                                                                     & \def\arraystretch{1}\begin{tabular}[t]{@{}l@{}}$\begin{cases} 		0, & y=\mu,\\ 		(y-\mu)\theta-\log\left(B(1-s\theta,1+s\theta)\right), & y\neq \mu,\\ 	\end{cases}$ \medskip\\ where $\theta\in\real_+:~y-\mu =\frac{1}{\theta} + \frac{\pi s}{\tan{(-\pi s\theta)}}$ %is the nonzero root of: \withsmallskip\\ $\qquad\qquad\quad{\ed\frac{\pi s\theta}{(y-\mu)\theta-1} = \tan(-\pi s\theta)}$
		\end{tabular}                        &      $\real$                                                                                                                                \\ \bottomrule
%	\end{tabular}
\caption{Cram\'er rate functions for popular distributions.}
	\vspace{-2mm}
\label{mem:tbl:cramer}
	\end{longtable}}

%\end{table}
\begin{remark}[On \cref{mem:tbl:cramer}] 
\label{mem:rmrk:special_cases_and_interpretation}
We provide some additional comments on \cref{mem:tbl:cramer} here.
	\begin{itemize}
		\item[(a)] (Special cases) 
		\begin{itemize}
			\item As special cases of the Gamma distribution we obtain Chi-squared with parameter $k$ ($\alpha=k/2$, $\beta=1/2$), Erlang ($\alpha$ positive integer), and exponential ($\alpha=1$) distributions.
			\item As special cases of the multinomial distribution, we obtain binomial ($d=1$, $n>1$), Bernoulli ($d=1$, $n=1$), and categorical ($d>1$, $n=1$) distributions.
			\item  As special cases of the negative multinomial distribution we obtain the negative binomial ($d=1$) and (shifted) geometric ($d=1$, $y_0=1$) distributions.			
%			\item Chi-squared, Erlang and exponential distributions are special cases of the Gamma distribution.
%			\item Binomial, Bernoulli and categorical distributions are special cases of the multinomial distribution.
%			\item Negative binomial and (shifted) geometric distributions are special cases of the negative multinomial distribution.
		\end{itemize}
		\item[(b)] (Statistical interpretation) For many reference distributions, 
		$\psi^*_P$ recovers well-known functions from information theory and related areas. Here, the MEM provides an information-driven, statistical interpretation for these functions. Examples include the squared Mahalanobis distance (multivariate normal), pseudo-Huber loss (multivariate normal-inverse Gaussian), Itakura-Saito distance (Gamma), Burg entropy (exponential), Fermi-Dirac entropy (Bernoulli), and the generalized cross entropy (Poisson). 
	\end{itemize}\rqed
\end{remark}

\section{The MEM Estimator and Models for Inverse Problems}
\label{sec:model}

In this section, we show how the MEM function can be used in various modeling paradigms. We start by presenting the MEM estimator and exploring some of its properties. We then discuss its (primal and dual) analogy to the maximum likelihood (ML) estimator. Finally, we will illustrate its efficacy by considering a class of linear models involving a regularization term. 

\subsection{The Maximum Entropy on the Mean Estimator}

\noindent The maximum entropy on the mean (MEM) function gives rise to an information-driven criterion for measuring the compliance of given data with a prior distribution. Based on this function, we can define the MEM estimator as given in \cref{model:def:MEM_estimator} below. First, we introduce some additional terminology and notation that will be used in the sequel. Let $\Omega\subseteq\real^d$ and let  %$F_\Theta:=\{P_\theta:\theta\in\Theta\subseteq\real^d\}\subset\PPP(\Omega)$ 
$F_\Lambda=\{P_\lambda:\lambda\in\Lambda\subseteq\real^d\}\subset\PPP(\Omega)$ 
be a parameterized family of distributions indexed by $\lambda\in\Lambda$ such that $\mathbb{E}_{P_{\lambda_1}}=\mathbb{E}_{P_{\lambda_2}}$ if and only if $\lambda_1=\lambda_2$. We call $F_\Lambda$ as the \emph{reference family} and say that it satisfies \cref{mem:asmp:blanket_minimal_and_steep,mem:asmp:Cramer_and_MEM_equiv_cond} if they hold for each  $P_\lambda\in F_\Lambda$. When  $F_\Lambda$ is an exponential family (in this case $\Lambda$ is the natural parameter space $\Theta_P$ for some $P\in\MMM(\Omega)$) the MEM estimator was studied in \cite[Chapter 6]{brown1986fundamentals}. We stress that, in our presentation, $F_\Lambda$ need {\em not} be an exponential family.

\begin{definition}[MEM estimator]
	\label{model:def:MEM_estimator}
	Let $F_\Lambda\subset\PPP(\Omega)$ be a reference family satisfying \cref{mem:asmp:blanket_minimal_and_steep,mem:asmp:Cramer_and_MEM_equiv_cond} and assume  that $\mathbb{E}_{P_{\lambda_1}}=\mathbb{E}_{P_{\lambda_2}}$ if and only if $\lambda_1=\lambda_2$. 
	For an observation $\hat y\in\real^d$, let %$\hat\theta\in\Theta$ 
	$P_{\hat{\lambda}}\in F_\Lambda$ be such that $\hat y=\mathbb{E}_{P_{\hat\lambda}}$,  and let $S^*\subseteq\real^d$ be (nonempty) closed. The MEM estimator %of $P$ over $S^*$, 
	is defined as
	\begin{gather*}
		y_{\scriptscriptstyle MEM}(\hat{y},F_\Lambda,S^*):=\argmin\{\psi^*_{P_{\hat\lambda}}(y):y\in S^*\}.
	\end{gather*}
\end{definition}
\noindent
In order to simplify notation, in what follows, we will write $y_{\scriptscriptstyle MEM}:=y_{\scriptscriptstyle MEM}(\hat{y},F_\Lambda,S^*)$ when the dependence on the triple $(\hat{y},F_\Lambda,S^*)$ is clear from the context.

\begin{remark}[The observation vector and its domain] %Domain of the observations vector]
	\label{model:rmrk:observations_domain}
	In \cref{model:def:MEM_estimator}, the condition that %$\hat\theta\in\Theta$
	$P_{\hat{\lambda}}\in F_\Lambda$ is chosen such that $\hat y=\mathbb{E}_{P_{\hat\lambda}}$ implies that the reference distribution is indexed by the observation vector $\hat y$. This condition combined with \cref{mem:asmp:blanket_minimal_and_steep}
	 entails that $\hat y\in \inte{\Omega_{P_{\hat\lambda}}^{cc}}$ must hold due to \cref{mem:lem:expected_val_in_ri}. \rqed
\end{remark}
\noindent In order to establish the well-definedness of the MEM estimator, we will use the following extension of \cite[Lemma 5.4]{brown1986fundamentals}. %The proof is included in \cref{apndx:sec:proofs_and_tables}.

\begin{lemma}
	\label{model:lem:when_mem_est_in_inetrior}
	Let $\phi:\real^d\rightarrow\erl$ be closed and  Legendre-type, let  ${\varphi:\real^d\rightarrow\erl}$ be proper, closed and convex such that $\idom \phi\cap\dom \varphi\neq\emptyset$.  Assume that one of the functions is coercive while the other is bounded from below.
	Then there exists a unique solution $y^*\in\real^d$ to $\min\{\phi(y)+\varphi(y):y\in\real^d\}$,
	% \begin{gather*}
	% 	\min\{\phi(y)+\varphi(y):y\in\real^d\},
	% \end{gather*}
	which also satisfies $y^*\in\idom \phi\cap\dom \varphi$.
\end{lemma}

\begin{proof}%[Proof (for \cref{model:lem:when_mem_est_in_inetrior})]
	The existence and uniqueness of the solution follow from \cite[Corollary 11.15]{bauschke2011convex}. It remains to show that $y^*\in\idom\phi\cap\dom\varphi$. Evidently, $y^*\in\dom\phi\cap\dom\varphi$ thus it is sufficient to show that $y^*\in\idom\phi$. Using \cite[Theorem 16.2]{bauschke2011convex} and \cite[Corollary 16.38]{bauschke2011convex}
	we have $0\in \partial \phi(y^*)+\partial \varphi(y^*)$,  in particular $\partial \phi(y^*)\neq\emptyset$. Since $\phi$ is of Legendre type we  conclude that $y^*\in\idom{\phi}$ \cite[Theorem 26.1]{rockafellar1970convex}. 
\end{proof}

% \begin{proof}
% 	The existence and uniqueness of the solution follows from \cite[Corollary 11.15]{bauschke2011convex}. It remains to show that $y^*\in\idom\phi\cap\dom\varphi$. Evidently, $y^*\in\dom\phi\cap\dom\varphi$ thus it is sufficient to show that $y^*\in\idom\phi$. Using \cite[Theorem 16.2]{bauschke2011convex} and \cite[Corollary 16.38]{bauschke2011convex}
% 	it must hold that $0\in \partial \phi(y^*)+\partial \varphi(y^*)$ and in particular $\partial \phi(y^*)\neq\emptyset$. Since $\phi$ is of Legendre type we can thus conclude that $y^*\in\idom{\phi}$ \cite[Theorem 26.1]{rockafellar1970convex}. 
% \end{proof}

\begin{theorem}[Well-definedness of the MEM estimator]
	\label{model:eq:MEM_estimator_well_def}	
	Let $F_\Lambda\subset\PPP(\Omega)$ be a reference family satisfying \cref{mem:asmp:blanket_minimal_and_steep,mem:asmp:Cramer_and_MEM_equiv_cond}.
	For $\hat y\in\real^d$, let $P_{\hat\lambda}\in F_\Lambda$ such that $\hat y=E_{P_{\hat\lambda}}$, and let $S^*\subseteq\real^d$ be  closed  with  $S^*\cap \dom \psi^*_{P_{\hat{\lambda}}}\neq \emptyset$. Then, the MEM estimator $y_{\scriptscriptstyle MEM}$ exists. If, in addition,  $S^*$ is convex and $\idom {\psi^*_{P_{\hat\lambda}}}\cap S^*\neq\emptyset$,  $y_{\scriptscriptstyle MEM}$   is unique and  in  $\idom {\psi^*_{P_{\hat\lambda}}}\cap S^*$.
\end{theorem}

\begin{proof}
	Recall that, by \cref{mem:thrm:MEM_properties},  $\psi^*_{P_{\hat\lambda}}$ is coercive and of Legendre type (proper, closed, steep and strictly convex on the interior of its domain). Observe that $S^*\subset \real^d$ is  closed and  $S^*\cap \dom \psi^*_{P_{\hat\lambda}}\neq \emptyset$. Thus, the function $\psi^*_{P_{\hat\lambda}}+\delta_{S^*}$ is proper, closed and coercive. Hence, the existence of the MEM estimator follows from \cite[Remark 3.4.1, Theorem 3.4.1]{auslender2006asymptotic}.
	The case when $S^*$ is convex and $\idom {\psi^*_{P_{\hat\lambda}}}\cap S^*\neq\emptyset$ follows from \cref{model:lem:when_mem_est_in_inetrior} with $\phi=\psi^*_{P_{\hat\lambda}}$ and $\varphi=\delta_S$ due to the coercivity of $\psi^*_{P_{\hat\lambda}}$ and the fact that $\delta_S$ is bounded from below.
\end{proof}

\subsubsection{Analogy Between MEM and ML (for Exponential Families)}

\noindent \emph{Maximum likelihood} (ML) is arguably the most popular principle for statistical estimation. Here, the estimated parameters are chosen as the most likely to produce a given sample of observed data while satisfying model assumptions. More precisely, for some $\Omega\subseteq\real^d$, the model is defined by means of a nonempty, closed set $S\subseteq\real^d$ of admissible parameters and a parameterized family of distributions %$\{P_\theta:\theta\in\Theta\subset\real^m\}$ 
$F_\Lambda=\{P_\lambda:\lambda\in\Lambda\subseteq\real^m\}\subset\PPP(\Omega)$ with densities $f_{P_\lambda}$. Given a sample of observed data $\hat y\in\real^d$, the ML estimator $\lambda_{ML}(\hat{y},F_\Lambda,S)$ is defined as 
\begin{gather*}
	\lambda_{\scriptscriptstyle ML}(\hat{y},F_\Lambda,S):=\argmax\{\log f_{P_\lambda}(\hat y):\lambda\in S\cap\Lambda\}.
\end{gather*}
\noindent In order to simplify notation, we will write $\lambda_{\scriptscriptstyle ML}:=\lambda_{\scriptscriptstyle ML}(\hat{y},F_\Lambda,S)$ when the dependence on the triple $(\hat{y},F_\Lambda,S)$ is clear from the context.\withsmallskip

An intriguing connection between the ML and MEM estimator comes to light when $\Lambda$ is the natural parameter space  $\Theta_P$ of an exponential family induced by $P\in\MMM(\Omega)$. 
The MEM estimator can then be retrieved by solving one of two alternative optimization problems each of which has a closely related problem that yields the ML estimator. One problem is driven by information-theoretic arguments, while the other emphasizes a connection motivated by convex duality. These connections were previously observed in \cite[Chapter 6]{brown1986fundamentals} (also \cite{ben1988role}) and are summarized in the following theorem. % whose proof is in \cref{apndx:sec:proofs_and_tables}. 
For consistency, we denote the ML estimator as $\theta_{\scriptscriptstyle ML}$.
\begin{theorem}[MEM and ML estimator analogy]
	\label{model:thrm:MEM_ML_analogy}
	Let $\FFF_P$ be a minimal and steep exponential family generated by ${P\in\MMM(\Omega)}$ and assume  that, for any $\theta\in\inte\Theta_P$, \cref{mem:asmp:Cramer_and_MEM_equiv_cond} holds with respect to $P_\theta\in\PPP(\Omega)$.
	Let $S,S^*\subseteq\real^d$  such that  $S\cap\dom\psi_P\neq\emptyset$ and $S^*\cap\dom\psi_P^*\neq\emptyset$. Finally, let $\hat y\in\inte\Omega_P^{cc}$ and set $\hat{\theta}:=\nabla \psi^*_{P}(\hat{y})$. Then the following hold:
	\vspace{1mm}
	\begin{enumerate}%[(a)]
		%correspondence
		\item[(a)] (Primal analogy) If 
		$S^*\cap \idom {\psi^*_P}\neq \emptyset$ and $\nabla \psi_P^*(S^*\cap\idom{\psi_P^*}) = S\cap\idom{\psi_P}$, then $y_{\scriptscriptstyle MEM}=\nabla \psi_P(\theta_{\scriptscriptstyle MEM})$ where
		\begin{gather}
			\label{model:eq:mle_and_mem_reverse_KL}
			\theta_{\scriptscriptstyle MEM}\in\argmin\{\KL{P_\theta}{P_{\hat\theta}}:\theta\in S\} \quad\text{and}\quad 
			\theta_{\scriptscriptstyle ML} \in\argmin\{\KL{P_{\hat\theta}}{P_\theta}:\theta\in S\}.
		\end{gather}
		% \begin{gather}
		% 	\label{model:eq:mle_and_mem_reverse_KL}
		% 	\def\arraystretch{2}
		% 	\text{and}\qquad
		% 	\begin{array}{ll}
		% 	\theta_{\scriptscriptstyle MEM}&\displaystyle\in\argmin\{\KL{P_\theta}{P_{\hat\theta}}:\theta\in S\},\\ 
		% 	\theta_{\scriptscriptstyle ML}&\displaystyle \in\argmin\{\KL{P_{\hat\theta}}{P_\theta}:\theta\in S\}.
		% 	\end{array}
		% \end{gather}
		\item[(b)] (Dual analogy): We have
		\begin{gather}
			\label{model:eq:mle_and_mem_reverse}				
				 y_{\scriptscriptstyle MEM}\in\argmin\{D_{\psi_P^*}(y,\hat y):y\in S^*\}\quad\text{and}\quad
				\theta_{\scriptscriptstyle ML} \in\argmin\{D_{\psi_P}(\theta,\hat \theta):\theta\in S\}.
		\end{gather}
		% \begin{gather}
		% 	\label{model:eq:mle_and_mem_reverse}
		% 			\def\arraystretch{2}
		% 	\text{and} \qquad
		% 	\begin{array}{ll}
		% 		 y_{\scriptscriptstyle MEM}&\displaystyle\in\argmin\{D_{\psi_P^*}(y,\hat y):y\in S^*\},\\
		% 		\theta_{\scriptscriptstyle ML}&\displaystyle \in\argmin\{D_{\psi_P}(\theta,\hat \theta):\theta\in S\}.
		% 	\end{array}
		% \end{gather}
	%where $S$ and $S^*$ satisfy ${\nabla \psi_P(S\cap \Theta_P) = S^*\cap \inte\conv\Omega_{\ed P_\theta}}$. 
	\end{enumerate}
\end{theorem}

\begin{proof}%[Proof (for \cref{model:thrm:MEM_ML_analogy})]
	Since $\FFF_P$ is assumed to be minimal and steep,  it is easy to verify (recall \eqref{mem:eq:moment_generating_of_exp_fam}) that $P_\theta$ satisfies \cref{mem:asmp:blanket_minimal_and_steep} for any $\theta\in\inte\Theta_P$.	As we assume $S\cap\dom\psi_P\neq\emptyset$ and $S^*\cap\dom\psi_P^*\neq\emptyset$, the MEM and ML estimator exist due to \cref{model:eq:MEM_estimator_well_def} and \cite[Theorem 5.7]{brown1986fundamentals}, respectively. We now prove (b). Since   $\FFF_P$ is an exponential family, we have $\log f_{P_\theta}(\hat y)=\inner{\hat y}{\theta}-\psi_P(\theta)$ and the ML estimator is a solution to
	\begin{gather*}
		%\label{model:eq:mle_ef_in_bregman_form}
		\def\arraystretch{2}
		\begin{array}{rl}
			\displaystyle\max\{\log f_{P_\theta}(\hat y):\theta\in S\} &\displaystyle=\max\{\inner{\hat y}{\theta}-\psi_P(\theta) :\theta\in S\}\\
			&\displaystyle=-\min\{D_{\psi_P}(\theta,\nabla \psi_P^*(\hat y)):\theta\in S\}-\psi_P(\nabla \psi_P^*(\hat y))+\inner{\hat y}{\nabla \psi_P^*(\hat y)}.
		\end{array}
	\end{gather*}
	Omitting terms independent of the minimization and using that $\hat \theta = \nabla \psi_{P}^*(\hat y)$, the formulation for the ML estimator follows. To obtain the formulation for the MEM estimator, observe that, due to %Theorem \ref{mem:thrm:Cramer_and_MEM_equiv} and
	\cref{mem:lem:relative_kramer}, we have
	\begin{gather*}
		%\min\{\psi^*_{P_{\hat\theta}}(y):y\in S^*\}= 
		\min\{\psi_{P_{\hat\theta}}^*(y):y\in S^*\}= \min\{D_{\psi_P^*}(y,\nabla \psi_P(\hat\theta)):y\in S^*\}.%\\
	\end{gather*} 
	Thus, the result follows by recalling that $\hat y = \nabla \psi_{P}(\hat \theta)$. 
	
	We now turn to prove (a). Since ${S^*\cap \idom {\psi^*_P}\neq \emptyset}$ we obtain by \cref{model:eq:MEM_estimator_well_def} that $y_{\scriptscriptstyle MEM}\in S^*\cap\idom{\psi^*_P}$. This fact combined with the assumption $\nabla\psi_P^*(S^*\cap\idom{\psi_P^*}) = S\cap\idom{\psi_P}$ implies that $\nabla \psi_P^*(y_{\scriptscriptstyle MEM})\in S\cap\idom{\psi_P}$. Thus, (a) follows from (b)  due to the Bregman distance dual representation property \eqref{pre:eq:Bregman_duality} and \cref{pre:rmrk:KL_special_cases}. 
\end{proof}

\noindent The primal and dual analogy between the MEM and ML estimator for exponential families clarifies that the two are symmetric principles.

\subsection{Examples - Linear Models}

To illustrate the versatility of the  MEM estimation framework, we will consider the broad class of linear models which are among the most popular paradigms in statistical estimation with applications in numerous fields such as image processing, bio-informatics, machine learning etc.\withsmallskip

%The linear model which we consider here are defined by means of 
We assume that the set $S^*$ of admissible mean value parameters is the image of a convex set $X\subseteq\real^d$ under a linear mapping defined by a measurement matrix $A\in \real^{m\times d}$. In many practical scenarios, this matrix satisfies some application-related properties, which in combination with the set $X$ restricts the image space to a subset of $\real^m$. We will denote by $\CCC$ the set of all matrices that satisfy such a condition for the application in question. The second component in the model is $F_\Lambda=\{P_\lambda:\lambda\in\Lambda\subseteq\real^m\}\subset\PPP(\Omega)$, a reference family indexed by $\lambda\in\Lambda$ such that $\mathbb{E}_{P_{\lambda_1}}=\mathbb{E}_{P_{\lambda_2}}$ if and only if $\lambda_1=\lambda_2$. The reference distribution is specified from this family by means of the observation vector $\hat y$. From \cref{model:rmrk:observations_domain} it follows that such a family of distributions must satisfy $\hat y\in\inte{\Omega_{P_{\hat\lambda}}^{cc}}$ for $\hat{\lambda}$ such that $\mathbb{E}_{P_{\hat\lambda}}=\hat y$. In some cases, this condition imposes additional assumptions that must be satisfied by the measurement vector. We will denote the set of measurement vectors that satisfy such an assumption with respect to the family of distributions under consideration by $D:=\{y\in\real^m:\mathbb{E}_{P_{\lambda}}=y~(\lambda\in\Lambda)\}$. To summarize, an MEM estimator of the linear model outlined above is obtained by solving 
\begin{gather}
	\label{model:eq:linear_model}
	\min\left\{\psi^*_{P_{\hat \lambda}}(Ax):x\in X\right\}\qquad (\hat\lambda\in\Lambda: \mathbb{E}_{P_{\hat\lambda}}=\hat y),
\end{gather} 
under the following set of assumptions:
\begin{assumption}[MEM estimation for linear models]~
	\label{model:asmp:linear_model}
	\begin{enumerate}
		\item The reference family $F_\Lambda$ satisfies \cref{mem:asmp:blanket_minimal_and_steep,mem:asmp:Cramer_and_MEM_equiv_cond}.
		\item The set $X\subseteq\real^d$ is nonempty and convex. 
		\item $A\in\CCC$ and for any $x\in X$ it holds that $Ax\in\dom\psi_P^*$.
		\item The observation vector satisfies $\hat{y}\in D$. 
%		The vector $\hat\theta\in\Theta$ satisfies $E_{P_{\hat \theta}}=\hat y$.
	\end{enumerate}
\end{assumption}
\noindent In the following table, we present some examples of MEM linear models that correspond to particular choices of a reference family. In all cases, we assume that the reference family admits a separable structure as outlined in \cref{mem:rmrk:Cramer_seperability}. 
%, namely each component corresponds to an independently distributed random variable. 
%The corresponding ML models are included for comparison. 
The vectors $a_i\;(i=1,\dots,m)$ stand for the $i$th row of the matrix $A$. We set %$e:=(1,1,\dots,1)^T\in\real^d$ and 
\[
\CCC_{0}:= \{A\in\real^{m\times d}_+: \text{A has no zero rows or columns}\}. 
\]
{ \small
\begin{table}[H]\small
	\def\arraystretch{2}
	\centering
	\begin{tabular}{@{}lccccc@{}}
		\toprule
		 Reference family & Objective function ($\psi_{P_{\hat\lambda}}^*\circ A$)                                                                     & $\CCC$              & $X$            & $D$                \\ 
		\midrule
		Normal                                                           &  $\displaystyle\frac{1}{2}\|Ax-\hat{y}\|_2^2$                                                                         & $\real^{m\times d}$ & $\real^d$      & $\real^m$          \\ \medskip
		Poisson                                                          &  $\displaystyle\sum_{i=1}^m\left[\inner{a_i}{x}\log\left(\inner{a_i}{x}/\hat{y}_i\right)-\inner{a_i}{x}+\hat{y}_i\right]$ & $\CCC_{0}$             & $\real^d_{+}$  & $ \real^m_{++}$ \\ \medskip
		Gamma ($\beta=1$) & 
	$\displaystyle\sum_{i=1}^m\left[\inner{a_i}{x}-\hat{y}_i\log\left(\inner{a_i}{x}\right) -\left(\hat{y}_i-\hat{y}_i\log\left(\hat y_i\right)\right) \right]$
	& $\CCC_{0}$ & $\real^d_{++}$ & $\real^m_{+}$ \\
% \midrule
% \multicolumn{5}{l}{Definitions: $\qquad\CCC_{0}:= \{A\in\real^{m\times d}_+:\text{with no zero rows or columns}\}$} \\
 \bottomrule
	\end{tabular}
	\caption{Linear models under the MEM estimation framework for various reference families.}
	\label{model:tbl:linear_models}
\end{table}
}

\begin{remark}
%	\begin{enumerate}
%		\item 
		Additional models are readily available by choosing any of the reference distributions presented in \cref{mem:tbl:cramer}. Alternatively, one may consider a family of
		linear models where the natural parameters are the ones restricted to the image of a convex set under a linear mapping. This class of models is commonly referred to as \emph{generalized linear models} with a \emph{canonical link function} \cite{nelder1972generalized}. \rqed
%	\end{enumerate}
\end{remark}
\noindent
The MEM linear model with reference family that corresponds to the normal distribution coincides with its ML counterpart, resulting in the celebrated least-squares model \cite{bjorck1996numerical}.
This phenomenon is unique for the normal distribution and is a direct consequence of the fact that the squared Euclidean norm is the only self-conjugate function \cite[Section 12]{rockafellar1970convex}.\withsmallskip

Linear inverse models under the Poisson noise assumption have been successfully applied in various disciplines including fluorescence microscopy, optical/infrared astronomy, and medical applications such as positron emission tomography (PET) (see, for example, \cite{ben1988role,vardi1985statistical}). 
The MEM linear model with Poisson reference distribution outlined in \cref{model:tbl:linear_models} was previously suggested in \cite[Subsection 5.3]{bauschke2017descent} as an example for the algorithmic setting considered in that work (see further details in \Cref{sec:algos} where we expand on the framework considered in \cite{bauschke2017descent}).\withsmallskip

If, for example, $X=\real^d$ and $\text{rge}{A}=\real^m$ with $m<d$, then  $x\in\real^d$ such that $y_{\scriptscriptstyle ML}=y_{\scriptscriptstyle MEM}=Ax=\hat{y}$. This outcome is not a result of a deep statistical characteristic but a simple consequence of the model's ill-posedness, a situation when the desired solution is not uniquely characterized by the model. Situations like this are among the reasons which motivate the use of \emph{regularizers} which allow for the incorporation of some additional (prior) knowledge of the solution. This approach gives rise to the following extended version of model \eqref{model:eq:linear_model} 
\begin{gather}
	\label{model:eq:regularized_linear_model}
	\min\left\{\psi^*_{P_{\hat \lambda}}(Ax)+ \varphi(x):x\in X\right\}\qquad (\hat\lambda\in\Lambda: \mathbb{E}_{P_{\hat\lambda}}=\hat y),
\end{gather} 
where, in our setting, $\varphi:\real^d\rightarrow\erl$ stands for a proper, closed, and convex function. In \eqref{model:eq:regularized_linear_model}, the optimization formulation is designed to find a solution (model estimator) that balances between two criteria represented by the \emph{fidelity} term $\psi^*_{P_{\hat \lambda}}\circ A$ and the \emph{regularization} term $\varphi$. 
While 
%The terminology is motivated by the fact that, commonly, 
the fidelity term penalizes the violation between the model and observations, the regularization term incorporates prior information (belief) on the solution, and in many cases, when the problem with the fidelity term alone is ill-posed, it also serves as a regularizer. In the context of MEM, the Cram\'er rate function can be used to penalize violations of the solution vector $x\in\real^d$ with respect to some prior reference measure $R\in\PPP(\Omega)$ that satisfies \cref{mem:asmp:blanket_minimal_and_steep,mem:asmp:Cramer_and_MEM_equiv_cond}. In other words, we can set $\varphi(x)=\psi_R^*(x)$.\withsmallskip

In many applications, the desired reference distribution of the regularizer will admit a separable structure (\`{a} la \cref{mem:rmrk:Cramer_seperability}). While this is advantageous from an algorithmic perspective (cf. \cref{algo:rmrk:Breg_operator_separability}), other alternatives are viable. Non-separable priors can be considered in order to promote desirable correlations between the entries of the solution to problem \eqref{model:eq:regularized_linear_model}. E.g., by considering the multinomial, negative multinomial, multivariate normal inverse Gaussian or multivariate normal (with non-diagonal correlation matrix in the latter) reference distributions intrinsically give rise to non-separable modeling. But there are other options that involve separable reference distributions with a composite structure such as 
\begin{gather}
	\label{model:eq:structured_regularizer}
	\varphi(x)=\psi_R^*(Lx)\qquad\text{or}\qquad\varphi(x)=\sum_{i=1}^d\psi_R^*(L_ix),
\end{gather}
where $L\in\real^{r\times d}, L_i\in\real^{r\times d}$. For example, new variants of the well-known (discrete) \emph{total variation} (TV) regularizer \cite{rudin1992nonlinear} can be considered by replacing the norm appearing in the original definition with a Cram\'er rate function while keeping the first-order finite difference matrix (further details are given at the end of \Cref{sec:algos}). %\ref{subsec:algos_exmaple_psuedohuber}). 
Different reference distributions might be used to promote desirable, application-specific, properties of the solution. Nevertheless, for all choices of reference distribution, the resulting function will admit some desirable properties, including convexity, differentiability, and coerciveness as established in \cref{mem:thrm:MEM_properties}. As we will see in the following section, these properties allow us to consider a unified algorithmic approach for tackling problem \eqref{model:eq:regularized_linear_model}.

\section{Algorithms}
\label{sec:algos}

The optimization formulations of statistical estimation problems as presented in the previous section are solved by optimization algorithms. Customized methods, such as the ones we consider here, allow us to leverage the structure of a given problem, thus resulting in a significant efficiency improvement compared to general-purpose solvers. The structure of problems which are of interest to us is given by the 
%Under our setting, the structure of the problems 
\emph{additive composite model} 
%optimization problems are formulated as the \emph{additive composite model} 
\begin{gather}\label{algo:eq:additive_composite}
	\min \{f(x)+g(x):x\in\real^d\},
\end{gather}
where $f,g:\real^d\rightarrow \erl$ %and $g:\real^d\rightarrow\erl$ 
are proper, closed, and convex.\withsmallskip

We will assume that both the fidelity and regularization terms, represented by $f$ and $g$, respectively, are continuously differentiable on the interior of their domain. %Since the MEM function is of Legendre type (Theorem \ref{mem:thrm:MEM_properties}), 
This assumption holds for all the modeling paradigms discussed in the previous section. In particular, model \eqref{model:eq:regularized_linear_model} is recovered with $f=\psi_{P}^*\circ A$ and $g=\psi_R^*$. Our focus on this type of problem is for convenience only as our goal is merely to illustrate how modern first-order methods can be used for computing MEM estimators, much like their popular ML counterparts. We point out that we are not limited to this setting. Other models can be considered as well, e.g., by blending a fidelity term originating from an MEM modeling paradigm with a traditional regularizer or vice versa. In this case, similar algorithms are applicable under suitable adjustments.\withsmallskip

The method we consider is the \emph{Bregman proximal gradient} (BPG) method. This first-order iterative algorithm admits a comparably mild per-iteration complexity and as such, it is particularly suitable for contemporary large-scale applications. It is important to notice that many other methods, including second-order and primal-dual decomposition methods, can be also considered in some scenarios and can benefit from the operators derived in this work. Before we  present the BPG method, we need to define its fundamental components
%, a \emph{smooth adaptable kernel} and the \emph{Bregman proximal operator}
\cite{bauschke2017descent,bolte2018first}.

\vspace{1mm}
\noindent
\noindent{\bf Smooth adaptable kernel:} Let $f:\real^d\rightarrow \erl$ be  proper,  closed and continuously differentiable on $\idom f$. Then $h:\real^d\rightarrow \erl$ of Legendre type is a \emph{smooth adaptable kernel} with respect to $f$ if $\dom h\subseteq \dom f$ and there exists $L>0$ such that $Lh-f$ is convex. 

\vspace{1mm}
\noindent
{\bf Bregman proximal operator:}  Let $g:\real^d\rightarrow \erl$ be  closed and proper and $h:\real^d\rightarrow \erl$ of Legendre type. Then the \emph{Bregman proximal operator} is defined as 
\begin{gather}
	\label{algo:eq:Bref_prox}
	\prox^{h}_{g}\left(\bar{x}\right) := \argmin\left\{g(x)+D_h(x,\bar{x}):x\in\real^n\right\} \qquad (\bar{x}\in \idom{h}).
\end{gather}

% \begin{description}
% 	\item[Smooth adaptable kernel] 	Let $f:\real^d\rightarrow \erl$ be a proper,  closed and continuously differentiable over $\idom f$. Then $h:\real^d\rightarrow \erl$ of Legendre type is a \emph{smooth adaptable kernel} with respect to $f$ if $\dom h\subseteq \dom f$ and there exists $L>0$ such that $Lh-f$ is convex. 
% 	\item[Bregman proximal operator]  Let $g:\real^d\rightarrow \erl$ closed and proper and $h:\real^d\rightarrow \erl$ of Legendre type. Then the \emph{Bregman proximal operator} is defined as 
% 	\begin{gather}
% 		\label{algo:eq:Bref_prox}
% 		\prox^{h}_{g}\left(\bar{x}\right) := \argmin\left\{g(x)+D_h(x,\bar{x}):x\in\real^n\right\} \qquad (\bar{x}\in \idom{h}).
% 	\end{gather}
% \end{description}

\noindent The BPG method is applicable under the following assumption. 

\begin{assumption}
	\label{algo:asmp:Bregman_conditions}
	Consider problem \eqref{algo:eq:additive_composite} and assume that there exists a function of Legendre type  $h:\real^d\rightarrow\erl$ such that:
	\begin{enumerate}
		\item $h$ is a smooth adaptable kernel with respect to $f$.
		\item $h$ induces a computationally efficient Bregman proximal operator with respect to $g$.
	\end{enumerate}
\end{assumption}

\noindent The BPG method reads:

\medskip

\fcolorbox{black}{white}{\parbox{14.5cm}{{\bf (BPG Method)} 
		% {\bf Initialization.} Pick $t\in(0,1/L]$ and $x^0\in\idom{h}$.
		% \withsmallskip\\
		% {\bf Procedure.} For $k=0,1,2,\dots$:
		Pick $t\in(0,1/L]$ and $x^0\in\idom{h}$. For $k=0,1,2,\dots$ compute
		\begin{gather*}
			\def\arraystretch{2}
			\begin{array}{rl}
				x^{k+1} = & \prox^{h}_{t g}\left(\nabla h^* \left(\nabla h(x^k)-t \nabla f(x^k) \right)\right).
				%p_k = & \nabla h_f^* \left(\nabla h_f(x^k)-\lambda \nabla f(x^k) \right),\\
				%x^{k^+1} = & \prox^{h_f}_{\lambda g}(p_k).
			\end{array}
		\end{gather*}
}}

\medskip
\noindent
For $h=(1/2)\|\cdot\|_2^2$ and $f$ convex, $Lh-f$ is convex if and only if $\nabla f$ is $L$-Lipschitz. In this case, the Bregman proximal operator reduces to the classical proximal operator and the BPG method is the well-known proximal gradient algorithm \cite{beck2017first}.\withsmallskip 

The BPG method for solving \eqref{algo:eq:additive_composite} exhibits a sublinear convergence rate \cite{bauschke2017descent}. Under suitable assumptions, the convergence improves to linear \cite{bauschke2019linear}. Accelerated variants, which improve practical performance and have superior theoretical guarantees under additional assumptions, are also available \cite{auslender2006interior,beck2009fast}. For simplicity's sake, we confine ourselves to the basic BPG scheme, but the operators to be presented can be readily applied to the enhanced algorithms.\withsmallskip 

In order to customize the method to a particular instance of problem \eqref{algo:eq:additive_composite}, a smooth adaptable kernel and corresponding Bregman proximal operator must be specified. To illustrate this idea for  MEM estimation, we focus on the linear models discussed in the previous section. In particular, we consider the model \eqref{model:eq:regularized_linear_model} where $\varphi=\psi_R^*$. 
%$f=\psi_{P}^*\circ A$ and $g=\psi_R^*$. 
We assume that \cref{model:asmp:linear_model} holds and that the prior reference measure $R\in\PPP(\Omega)$ satisfies \cref{mem:asmp:blanket_minimal_and_steep,mem:asmp:Cramer_and_MEM_equiv_cond}. Furthermore, we assume that $\dom\psi_R\subseteq X$ which allows us to disregard the constraint $x\in X$. The latter assumption holds in many practical situations and we assume it here for simplicity. Otherwise, one can simply apply the BPG method with $g=\psi_R^*+\delta_X$ (under the appropriate adjustments to the proximal operator).
In \cref{algo:tbl:smooth_adaptable_kernels} below, we summarize the smooth adaptable kernels suitable for the models described in the previous section, see \cref{model:tbl:linear_models}. In all cases, the smooth adaptable function admits a separable structure of the form $h(x)=\sum_{j=1}^d h_j(x_j)$ where $h_j:\real\rightarrow\erl\;(j=1,\dots,d)$ is a (univariate) function of Legendre type. As we will see in what follows,  this property is very desirable as it gives rise to a computationally efficient implementation of the Bregman proximal operator. For completeness, we include the explicit formulas for the operators involved in the BPG method. 
{\small
\begin{table}[H]\small
	\centering
	\def\arraystretch{2}
	\begin{tabular}{@{}lcccc@{}}
		\toprule
		Reference family                                                                                                                                                       & Kernel ($h_j$)               & Constant ($L$)   & $[\nabla h(x)]_j$           & $[\nabla h^*(z)]_j$                                      \\ \midrule
		Normal                                         &  $\displaystyle (1/2)x_j^2$   & $\displaystyle\|A\|_2 := \sqrt{\lambda_{\max}(A^TA)}$  & $\displaystyle x_j$         & $\displaystyle z_j$                 \\
		Poisson   & $\displaystyle x_j\log(x_j)$ &  %$\displaystyle\max_{1\leq j\leq n}\sum_{i=1}^n a_{ij}$
		$\displaystyle \|A\|_1:=\max_{j=1,2,\dots,d}\sum_{i=1}^m|A_{i,j}|$ & $\displaystyle \log(x_j)+1$ & $\displaystyle \exp(z_j-1)$\\
		Gamma ($\beta=1$) & $\displaystyle -\log(x_j)$  & $\displaystyle \|\hat{y}\|_1:=\sum_{i=1}^m|\hat{y}_i|$ & $\displaystyle-1/x_j$ & $\displaystyle-1/z_j$ \\
		\bottomrule
	\end{tabular}
	\caption{Smooth adaptable kernels and related operators that correspond to the objective function ($f=\psi_{P_{\hat{\theta}}}^*\circ A$) of the linear models listed in \cref{model:tbl:linear_models}.}
	\label{algo:tbl:smooth_adaptable_kernels}
\end{table}
}
\noindent The kernel and related constant that corresponds to the normal reference family is a well-known consequence due to the Lipschitz gradient continuity, a special case of the smooth adaptability property considered here.\footnote{More precisely, the equivalence holds for convex functions such as the ones considered here. For the nonconvex case see an extension of the smooth adaptability condition presented in \cite{bolte2018first}.} The kernel and related constant that corresponds to the Poisson reference family is due to \cite[Lemma 8]{bauschke2017descent}. 
The kernel and related constant that corresponds to the Gamma distribution follows from \cite[Lemma 7]{bauschke2017descent}.\withsmallskip

We now discuss the special form of the Bregman proximal operator in the setting of the linear model \eqref{model:eq:regularized_linear_model} with $\varphi=\psi_R^*$. According to \eqref{algo:eq:Bref_prox}, for any $t>0$, the Bregman proximal operator is defined by the smooth adaptable kernel $h$ and the regularizer $g=\psi_R^*$ as follows:
\begin{gather}
	\label{algo:eq:Bregman_prox_operator_for_Cramer}
	\prox^{h}_{t\psi_{R}^*}\left(\bar{x}\right)=\argmin\left\{t\psi_{R}^*(u)+D_{h}(u,\bar{x}):u\in\real^d\right\}.
\end{gather}
The following theorem records that, in our setting, the above operator is well-defined. 
\begin{theorem}[Well-definedness of the Bregman proximal operator]
	\label{algo:lem:breg_prox_operator_well_defined}
	Let $h:\real^d\rightarrow\erl$ be of Legendre type and let $R\in\PPP(\Omega)$ be a reference distribution satisfying the conditions in \cref{mem:asmp:blanket_minimal_and_steep,mem:asmp:Cramer_and_MEM_equiv_cond}. Assume further that $\idom h\cap \dom \psi_R^*\neq\emptyset$. Then, for any $t>0$ and $\bar{x}\in\idom{h}$, the Bregman proximal operator defined in \eqref{algo:eq:Bregman_prox_operator_for_Cramer} produces a unique point
	in $\idom{h}\cap\dom{\psi_R^*}$.
\end{theorem}
\begin{proof}
Since $\bar{x}\in\idom{h}$, the function $D_h(\cdot,\bar{x})$ is proper. In addition, since $h$ is of Legendre type, so is $D_h(\cdot,\bar{x})$. Finally, $D_h(\cdot,\bar{x})$ is bounded below (by zero) by the convexity of $h$. The result follows from \cref{model:lem:when_mem_est_in_inetrior} with $\phi=D_h$ and $\varphi=t\psi_R^*$ due to the aforementioned properties of $D_h$ and the coercivity of $t\psi_R^*$ (\cref{mem:thrm:MEM_properties} and $t>0$). 
\end{proof}
\noindent We now show that this operator is also computationally tractable. For many reference distributions, this fact stems from the following separability property.% of the resulting subproblems.
\begin{remark}[Separability of the Bregman proximal operator]
	\label{algo:rmrk:Breg_operator_separability}
	In all cases under consideration, the smooth adaptable kernel $h:\real^d\rightarrow\erl$ admits a separable structure $h(x)=\sum_{j=1}^d h_j(x_j)$. Therefore, by \eqref{pre:eq:Bregman_linear_additivity}, the induced Bregman distance satisfies: 
	$D_h(x,y) = \sum_{i=1}^d D_{h_i}(x_i,y_i).$
	If, in addition, the Cram\'er rate function admits a separable structure $\psi_{R}^*=\sum_{i=1}^{d}\psi_{R_i}^*$ (cf. \cref{mem:rmrk:Cramer_seperability}),  then the optimization problem defining the Bregman proximal operator is separable and can be evaluated for each component of $\bar{x}$.  \rqed
\end{remark}
Given a particular instance of problem \eqref{algo:eq:additive_composite}, with fidelity term $f=\psi_{P_{\hat\lambda}}^*\circ A$ and regularizer $g=\psi_R^*$, one can derive a formula for the corresponding Bregman proximal operator. These formulas are summarized in \cref{algo:tbl:breg_prox_Normal,algo:tbl:breg_prox_Poisson,algo:tbl:breg_prox_Gamma} for each of the combinations of linear models (by using a compatible kernel generating distance from \cref{algo:tbl:smooth_adaptable_kernels}) and regularizers from \cref{mem:tbl:cramer}.
Some formulas are given in a closed form, others must be evaluated numerically through a solution of a nonlinear system.\footnote{The solution of the nonlinear system can be efficiently approximated by various methods. In our implementation, building upon the fact that the systems involve monotonic functions (since they stem from the optimality conditions of a convex problem), we used a variant of safeguarded Newton-Raphson method.} Due to \cref{algo:rmrk:Breg_operator_separability}, for most of the regularizer reference distributions (excluding only the multivariate normal, multinomial and negative multinomial) the resulting subproblem is separable. Thus, for the sake of simplicity and without loss of generality, we assume that $d=1$, i.e., the resulting formulas correspond to one entry of the vector produced by the operator. The general case follows by applying the operator components-wise on all the elements of a vector $\bar{x}\in\real^d$. An implementation of the operators along with selected algorithms, applications, and detailed derivations of the operators can be found under:
\\
\begin{center}
\url{https://github.com/yakov-vaisbourd/MEMshared}.
\end{center}
\bigskip
The following table lists the formulas of Bregman proximal operators for the normal linear family. In this case, the operator reduces to the classical proximal operator \cite{moreau1965proximite}. 
\medskip

%\begin{table}[H]
{\small
	\centering
	\def\arraystretch{2}
%	\begin{tabular}{@{}L{5.8cm}c @{}}
	\begin{longtable}{lc}%{| p{.20\textwidth} | p{.80\textwidth} |} 
	%\begin{tabular}{lc}
		\toprule 
		\multicolumn{1}{c}{Reference Distribution ($R$)}                                                                                                                           & Proximal Operator $(x^+=\prox_{t\psi_R^*}(\bar x))$                                                                                                                                                                                                                                                                                                                                                                    \\ \midrule
%		\endhead
		\endfirsthead
		\multicolumn{2}{@{}l}{\ldots continued}\\\toprule
		\multicolumn{1}{c}{Reference Distribution ($R$)}                                                                                                                           & Proximal Operator $(x^+=\prox_{t\psi_R^*}(\bar x))$                                                                                                                                                                                                                                                                                                                                                                  \\ \midrule
		\endhead % all the lines above this will be repeated on every page
		\midrule
		\multicolumn{2}{r@{}}{continued \ldots}\\

		\endfoot
%		\bottomrule
%		\endfoot
		%\bottomrule
		%\caption{Proximal operators - Normal linear model.}
		\endlastfoot
		{\color[HTML]{333333} \def\arraystretch{1}\begin{tabular}[c]{@{}l@{}}Multivariate Normal\\ $(\mu\in\real^d, \Sigma\in\bbS^d:\Sigma\succ 0)$\end{tabular}}                                & {\color[HTML]{333333} $x^+=(tI+\Sigma)^{-1}(\Sigma\bar{x}+t\mu)$}                                                                                                                                                                                                                                          \medskip\\
		\def\arraystretch{1}\begin{tabular}[c]{@{}l@{}}Multivariate Normal-inverse \\ 
			Gaussian $\big(\mu,\beta\in\real^d,~\alpha,\delta\in\real,$\\ 
			$\Sigma\in\real^{d\times d}:\delta>0,~\Sigma\succ0,$\\
			$\alpha^2\geq\beta^T\Sigma\beta$,
			$\gamma:=\sqrt{\alpha^2-\beta^T\Sigma\beta}\big)$\end{tabular}  &     \def\arraystretch{1}\begin{tabular}[t]{@{}l@{}}$x^+=\left(I+\rho\Sigma^{-1}\right)^{-1}\left(t\beta+\bar{x}+\rho\Sigma^{-1}\mu\right),$ where $\rho\in\real_+:$\medskip\\
			%where $	\rho\in\real:
			$\qquad(\rho\delta)^2+\|\left(\rho^{-1}I+\Sigma^{-1}\right)^{-1}\left(t\beta+\bar{x}-\mu \right)\|_{\Sigma^{-1}}^2=(\alpha t)^2$
		\end{tabular}	 \medskip\\
%		\def\arraystretch{1}\begin{tabular}[c]{@{}l@{}}Normal-inverse Gaussian \\ $(\alpha,\beta,\gamma,\delta\in\real: \alpha\geq|\beta|,$\\
%			$\qquad\delta>0,\gamma=\sqrt{\alpha^2-\beta^2})$\end{tabular} &     $x^+\in\real:\frac{\alpha t(x^+-\mu)}{\sqrt{\delta^2+(x^+-\mu)^2}}=\beta t-(x^+-\bar{x})$                                                                                                                                                                                                                                                                                                                                                                                   \medskip\\
		Gamma $(\alpha,\beta\in\real_{++})$                                                                                                                                                               &         $x^+=\left(\bar{x}-t\beta+\sqrt{(\bar{x}-t\beta)^2+4t\alpha}\right)/2$                                                                                                                                                                                                                                                                                                                                                                               \medskip\\
		Laplace $(\mu\in\real,~b\in\real_{++})$                                                                                                                                       & \def\arraystretch{1}\begin{tabular}[t]{@{}l@{}}$\hspace{3.3cm}x^+=\begin{cases}\mu,& \bar{x}=\mu,\\\mu+b\rho,& \bar{x}\neq\mu,\end{cases}$\medskip\\ where $\rho\in\real:\quad\alpha_1\rho^3+\alpha_2\rho^2+\alpha_3\rho+\alpha_4=0$,\medskip\\ with $\alpha_1=(b/t)^2b^2,~\alpha_2=2(b/t)^2b(\mu-\bar{x}),$\medskip\\ 
		$\quad\alpha_3=(b/t)^2(\mu-\bar{x})^2-2(b/t)b-1,~\alpha_4=-2(b/t)(\mu-\bar{x})$	
			%$\begin{array}{ll}\\ 	\alpha_1=(b/t)^2b^2,& \alpha_3=(b/t)^2(\mu-\bar{x})^2-2(b/t)b-1\\\\ 	\alpha_2=2(b/t)^2b(\mu-\bar{x}),& \alpha_4=-2(b/t)(\mu-\bar{x}).\\ 	\end{array}$
		\end{tabular} \medskip\\
		Poisson\footnotemark{}
		$(\lambda\in\real_{++})$                                                                                                                                     &   $x^+ = tW\left(\frac{\lambda e^{\bar{x}/t}}{t}\right)    $                     %$x^+\in\real_+:~\log(x^+/\lambda)+(x^+-\bar{x})/t=0$                                                                                                                                                                                                                                                                                                                                                            
		\medskip\\
		\def\arraystretch{1}\begin{tabular}[c]{@{}l@{}}Multinomial $(n\in\nn, p\in\Delta_{(d)}$: \\
		$\sum_{i=1}^d p_i<1)$
		\end{tabular}                                                                                                   &    
		$x^+\in\real^d_{+}\cap I(p):\quad (x^+_i-\bar{x}_i)/t+\log\left(\frac{x^+_i(1-\sum_{j=1}^dp_j)}{p_i(n-\sum_{j=1}^dx^+_j)}\right)=0$
%		\def\arraystretch{1}\begin{tabular}[t]{@{}l@{}} $x^+\in\real^d_+\cap I(p):~\sum_{i=1}^d x^+_i=n$	\medskip\\ 
%		and  $\quad(x^+_i-\bar{x}_i)/t+\log\left(\frac{x^+_i}{np_i}\right)+1 +\lambda=0$\end{tabular}	                                                                       
      \medskip\\
		\def\arraystretch{1}\begin{tabular}[c]{@{}l@{}} Negative Multinomial $(p\in[0,1)^d,$\\
		$x_0\in\real_{++},~p_0:=1-\sum_{i=1}^d p_i>0)$	
		\end{tabular}                                                                                                                                                 &

	$x^+\in\real^d_+\cap I(p):~
	(x^+_i-\bar{x}_i)/t+\log\left(\frac{x^+_i}{p_i(x_0+\sum_{j=1}^dx^+_j)}\right)=0,
	$

\medskip\\
%		Degenerate $(a\in\real^d)$                                                                                                                                           &              $x^+=a$                                                                                                                                                                                                                                                                                                                                                                          \medskip\\
		\def\arraystretch{1}\begin{tabular}[c]{@{}l@{}} Discrete Uniform \\ $(a,b\in\real:a<b)$      \end{tabular}                                                                  &       \def\arraystretch{1}\begin{tabular}[t]{@{}l@{}}$x^+=\bar{x}-t\theta^+$
		where $\theta^+=0$ if $\bar{x}=(a+b)/2$, \medskip\\
		otherwise: $\theta^+\in\real\setminus\{0\}$:\medskip\\
		$\quad t(\theta^+-\bar{x}/t)+ \frac{(b+1)e^{(b+1)\theta^+}-ae^{a\theta^+}}{e^{(b+1)\theta^+}-e^{a\theta^+}}=\frac{e^{\theta^+}}{e^{\theta^+}-1}$
		                    \end{tabular}                                                                                                                                                                                                                                                                                \medskip\\
	                    \def\arraystretch{1}\begin{tabular}[c]{@{}l@{}}Continuous Uniform \\ $(a,b\in\real:a\leq b)$ \end{tabular}                                                                         &       \def\arraystretch{1}\begin{tabular}[t]{@{}l@{}}$x^+=\bar{x}-t\theta^+$
	                    where $\theta^+=0$ if $\bar{x}=(a+b)/2$, \medskip\\
	                    otherwise: $\theta^+\in\real\setminus\{0\}$:\medskip\\
	                    $\qquad t(\theta^+-\bar{x}/t)+ \frac{b e^{b \theta^+}-a e^{a \theta^+}}{e^{b \theta^+}-e^{a \theta^+}}=\frac{1}{\theta^+}$
                    \end{tabular}                                                                                                                                                                                                                                                                                                                                                                                                                                     \medskip\\
		Logistic $(\mu\in\real,~s\in\real_{++})$:                                                                                                                                      & \def\arraystretch{1}\begin{tabular}[t]{@{}l@{}}$x^+=\bar{x}-t\theta^+$
			where $\theta^+=0$ if $\bar{x}=\mu$, \medskip\\
			otherwise: $\theta^+\in\real\setminus\{0\}$:\medskip\\
			$\qquad t\theta^++\frac{1}{\theta^+} + \frac{\pi s}{\tan{(-\pi s\theta^+)}}=\bar{x}-\mu$
		\end{tabular}                                                                              
	 \\ \bottomrule
	%\end{tabular}
	%\hline
	\caption{Bregman Proximal Operators - Normal Linear Model  ($h=\frac{1}{2}\|\cdot\|^2$).}
	%\endlastfoot
	\label{algo:tbl:breg_prox_Normal}
	\end{longtable}
}
 %\addcounter{footnote}
%\footnotetext{For a symmetric positive definite matrix $M\in\real^{d\times d}$ and a vector $z\in\real^d$ we denote $\|z\|_M^2:=\inner{z}{M z}$.}
%\end{table}
%\addtocounter{footnote}{-1}
%\footnotetext{For a symmetric positive definite matrix $M\in\real^{d\times d}$ and a vector $z\in\real^d$ we denote $\|z\|_M^2:=\inner{z}{M z}$.}
%\addtocounter{footnote}{1}
%\stepcounter{footnote}
\footnotetext{We denote by $W:\real\rightarrow\real$ the Lambert $W$ function (see, for example, \cite{corless1996lambertw}).}

\noindent Recall that the Cram\'er rate function induced by a uniform (discrete/continuous) or logistic reference distribution does not admit a closed form. To compute their proximal operator we appeal to the corresponding dual of the subproblem in \eqref{algo:eq:Bregman_prox_operator_for_Cramer}. This is done via Moreau decomposition (see, e.g., \cite[Theorem 6.45]{beck2017first}) which applies when the  Bregman proximal operator \eqref{algo:eq:Bregman_prox_operator_for_Cramer} reduces to the classical proximal operator (i.e., when $h=(1/2)\|\cdot\|_2^2$). For the general case, we will employ a result summarized in \cref{algo:lem:Moreau_gen} and \cref{algo:cor:Moreau_gen_and_trivial_case} below. %The proofs of both results can be found in \cref{apndx:sec:proofs_and_tables}. 
Some notation is needed:  for a function $g:\real^d\rightarrow\erl$  proper, closed and convex and  of  $h:\real^d\rightarrow\erl$ of Legendre type we set
\begin{gather}
	\label{algo:eq:inf_conv_sol}
\text{iconv}^h_g(\bar{x}):=\argmin\left\{g(x)+h(\bar{x}-x):x\in\real^d \right\}.
\end{gather}
This is the (possibly empty) solution of the optimization problem defining the \emph{infimal convolution} $(g\square h)(\bar{x}): = \inf\left\{g(x)+h(\bar{x}-x):x\in\real^d \right\}$.
% \begin{gather*}
% 	(g\square h)(\bar{x}) : = \inf\left\{g(x)+h(\bar{x}-x):x\in\real^d \right\}.
% \end{gather*}
\begin{lemma}
	\label{algo:lem:Moreau_gen}
Let $g:\real^d\rightarrow\erl$ be proper, closed, and convex, and let $h:\real^d\rightarrow\erl$ be of Legendre type. Let $\bar{x}\in\idom{h}$ and assume that there exists a unique point $x^+:=\prox^h_{g}(\bar{x})$ satisfying $x^+\in\idom{h}\cap\dom{g}$. Then, $y^+:=\text{\emph{iconv}}^{h^*}_{g^*}(\nabla h(\bar{x}))$ exists and it holds that $\nabla h(x^+)+y^+=\nabla h(\bar{x})$. 
% \begin{gather*}
% 	\nabla h(x^+)+y^+=\nabla h(\bar{x}).
% \end{gather*}
\end{lemma}

\begin{proof}%[Proof (for \cref{algo:lem:Moreau_gen})]	
	By the optimality condition of the optimization problem in the definition of the Bregman proximal operator \eqref{algo:eq:Bref_prox} we obtain that 
	\begin{gather*}
	%	\label{algo:eq:fo_Breg_prox}
		\nabla h(\bar{x})-\nabla h(x^+)\in\partial g(x^+).
	\end{gather*}
	Since $g$ is assumed to be proper, closed and convex, \eqref{pre:eq:fechel_in_equality_equivs} yields
	\begin{gather}
		\label{algo:eq:fo_Breg_prox_dual}
		x^+\in\partial g^*\left(\nabla h(\bar{x})-\nabla h(x^+)\right).
	\end{gather}
	Setting $\tilde{y}:=\nabla h(\bar{x})-\nabla h(x^+)$ and observing that $x^+=\nabla h^*(\nabla h(\bar{x})-\tilde{y})$ we can rewrite \eqref{algo:eq:fo_Breg_prox_dual} as 
	\begin{gather*}
	%	\label{algo:eq:fo_inf_conv}
		\nabla h^*(\nabla h(\bar{x})-\tilde y)\in\partial g^*(\tilde y). 
	\end{gather*}
	It is now easy to verify that the above is nothing else but the optimality condition for $\bar y$, thus, $\tilde y=y^+$ and we can conclude that $\nabla h(x^+)+y^+=\nabla h(\bar{x})$, establishing the desired result.
\end{proof}

\noindent 
The following corollary adapts the above lemma to the setting considered in our study. Furthermore, we complement this result with a simple observation which is particularly useful for Bregman proximal operator computations.

\begin{corollary}
	\label{algo:cor:Moreau_gen_and_trivial_case}
	Let $h:\real^d\rightarrow\erl$ be of Legendre type and let $R\in\PPP(\Omega)$ satisfy  \cref{mem:asmp:blanket_minimal_and_steep,mem:asmp:Cramer_and_MEM_equiv_cond}. Assume further that $\idom h\cap \dom \psi_R^*\neq\emptyset$. For $t>0$ and $\bar{x}\in\idom{h}$, let $x^+:=\prox^{h}_{t\psi_R^*}(\bar{x})$ and $\theta^+:=\text{\emph{iconv}}^{h^*}_{t\psi_R(\cdot/t)}(\bar{x})$. Then, $\nabla h(x^+)+\theta^+=\nabla h(\bar{x})$.
	% \begin{gather*}
	% 	\nabla h(x^+)+\theta^+=\nabla h(\bar{x}).
	% \end{gather*}
	In particular, $\theta^+=0$ (and $x^+=\bar{x}$) if and only if $\bar{x}=\mathbb E_R$.
\end{corollary}

\begin{proof}%[Proof (for \cref{algo:cor:Moreau_gen_and_trivial_case})]
	By \cref{mem:thrm:MEM_properties} we have that $\psi^*_R$ is proper, closed and convex and thus $\psi^{**}_R=\psi_R$ due to \cite[Theorem 4.8]{beck2017first}. By \cref{algo:lem:breg_prox_operator_well_defined} we know that $x^+$ is well-defined. The proof of the first part  then follows directly from \cref{algo:lem:Moreau_gen} (with $g=t\psi^*_R$ and $y^+=\theta^+$) and \cite[Theorem 4.14(a)]{beck2017first}. To see that $\theta^+=0$ if and only if $\bar{x}=\mathbb E_R$, observe that the objective function in the subproblem defining the Bregman proximal operator \eqref{algo:eq:Bregman_prox_operator_for_Cramer} is greater equal than zero, and equality holds if and only if $\bar{x}=\mathbb E_R$ with $x^+=\bar{x}$. Thus, the statement holds true in view of the first part of the current corollary. 
\end{proof}

\noindent

%The formulas of Bregman proximal operators for the Poisson and Gamma ($\beta=1$) linear families are included in \cref{apndx:sec:proofs_and_tables}. 

The following tables list the formulas of Bregman proximal operators for the Poisson and Gamma ($\beta=1$) linear families, respectively. Observe that by \cref{algo:lem:breg_prox_operator_well_defined} the Bregman proximal operator is well defined 
if $\idom h\cap \dom \psi_R^*\neq\emptyset$. Since $\idom h=\real^d_{++}$ this implies that for the multinomial and negative multinomial distributions we must assume that $p_i>0$ for all $i=1,2,\dots,d$. Furthermore, for the sake of simplicity, we include the normal and normal inverse-Gaussian distributions. The multivariate variants can be found in the software documentation along with further explanations. 

% \borderline
% \here{\pr update the above sentence and the table below. Check if you should consolidate the univariate and multivariate NiG.}
% \\
% \borderline

{\small
%\begin{table}[H]
	\centering
	\def\arraystretch{2}
	%	\begin{tabular}{@{}L{5.8cm}c @{}}
		\begin{longtable}{lc}%{| p{.20\textwidth} | p{.80\textwidth} |} 
			%\begin{tabular}{lc}
			\toprule 
			\multicolumn{1}{c}{Reference Distribution ($R$)}                                                                                                                           & Bregman Proximal Operator $(x^+=\prox_{t\psi_R^*}^h(\bar x))$                                                                                                                                                                                                                                                                                                                                                                    \\ \midrule
			%		\endhead
			\endfirsthead
			\multicolumn{2}{@{}l}{\ldots continued}\\\toprule
			\multicolumn{1}{c}{Reference Distribution ($R$)}                                                                                                                           & Proximal Operator                                                                                                                                                                                                                                                                                                                                                                      \\ \midrule
			\endhead % all the lines above this will be repeated on every page
			\midrule
			\multicolumn{2}{r@{}}{continued \ldots}\\
			\endfoot
			%		\bottomrule
			%		\endfoot
			%\bottomrule
		%	\caption{Bregman proximal operators - Poisson linear model.}
			\endlastfoot
			{\color[HTML]{333333} \def\arraystretch{1}\begin{tabular}[c]{@{}l@{}} Normal\\ $(\mu,\sigma\in\real:~\sigma>0)$\end{tabular}}                                & {\color[HTML]{333333} $x^+= \frac{\sigma}{t}W\left( \frac{t}{\sigma}\bar{x} e^{\frac{t\mu}{\sigma}} \right)$}                                                                                                                                                                                                                                          \medskip\\
			\def\arraystretch{1}\begin{tabular}[c]{@{}l@{}}Normal-inverse Gaussian\\ $\big(\mu,\alpha,\beta,\delta\in\real:~\delta>0$,\\ 
				$\quad\alpha\geq|\beta|,
				~\gamma:=\sqrt{\alpha^2-\beta^2}\big)$\end{tabular}  &     				  \def\arraystretch{1}\begin{tabular}[t]{@{}l@{}}
				$x^+\in\real_{++}:$ \\$(t\alpha/\sigma)(x^+-\mu)=\left(t\beta-\log(x^+/\bar{x})\right)\sqrt{\delta^2+(x^+-\mu)^2/\sigma}$%\medskip\\
				%				%where $	\rho\in\real:
				%				$\qquad(\rho\delta)^2+\|\left(\rho^{-1}I+\Sigma^{-1}\right)^{-1}\left(t\beta+\bar{x}-\mu \right)\|_{\Sigma^{-1}}^2=(\alpha t)^2$
			\end{tabular}	 \medskip\\
			%		\def\arraystretch{1}\begin{tabular}[c]{@{}l@{}}Normal-inverse Gaussian \\ $(\alpha,\beta,\gamma,\delta\in\real: \alpha\geq|\beta|,$\\
				%			$\qquad\delta>0,\gamma=\sqrt{\alpha^2-\beta^2})$\end{tabular} &     $x^+\in\real:\frac{\alpha t(x^+-\mu)}{\sqrt{\delta^2+(x^+-\mu)^2}}=\beta t-(x^+-\bar{x})$                                                                                                                                                                                                                                                                                                                                                                                   \medskip\\
			Gamma $(\alpha,\beta\in\real_{++})$                                                                                                                                                               &         $x^+
			= \frac{\alpha t}{W\left( \frac{\alpha t \exp\left({t\beta}\right)}{\bar{x}} \right)}$                                                                                                                                                                                                                                                                                                                                                                               \medskip\\
			Laplace $(\mu\in\real,~b\in\real_{++})$                                                                                                                                       & \def\arraystretch{1}\begin{tabular}[t]{@{}l@{}}$\hspace{3.3cm}x^+=\begin{cases}\mu,& \bar{x}=\mu,\\\mu+b\rho,& \bar{x}\neq\mu,\end{cases}$\medskip\\  where $\rho\in\real:\quad\rho +\frac{2b}{t}\log\left( \frac{\mu + b\rho}{\bar{x}} \right)  =  \frac{b^2 \rho}{t^2}\log^2\left( \frac{\mu + b\rho}{\bar{x}} \right)$	
				%$\begin{array}{ll}\\ 	\alpha_1=(b/t)^2b^2,& \alpha_3=(b/t)^2(\mu-\bar{x})^2-2(b/t)b-1\\\\ 	\alpha_2=2(b/t)^2b(\mu-\bar{x}),& \alpha_4=-2(b/t)(\mu-\bar{x}).\\ 	\end{array}$
			\end{tabular} \medskip\\
			Poisson $(\lambda\in\real_{++})$                                                                                                                                     &                            $x^+ = \bar{x}^{1-\tau}\lambda^{\tau}\qquad(\tau:=\frac{t}{t+1})$
			 \medskip\\
			Multinomial $(n\in\nn, p\in\inte\Delta_{(d)})$                                                                                                       & 
			\def\arraystretch{1}\begin{tabular}[t]{@{}l@{}}
				$\qquad x_i^+=\gamma_i\left(n-\rho\right)^{\tau}\quad \left(\tau:=\frac{t}{t+1},~\gamma_i := \left[\frac{p_i\bar{x_i}^{1/t}}{1-\sum_{j=1}^dp_j}\right]^{\tau}\right)
			    $
			\medskip\\
				$\qquad$where $\rho\in\real:~\rho = (n-\rho)^{\frac{t}{t+1}}\left(\sum_{i=1}^{d}\gamma_i\right)$
			\end{tabular}
                                         \medskip\\
			\def\arraystretch{1}\begin{tabular}[c]{@{}l@{}} Negative Multinomial $(p\in(0,1)^d,$\\
				$x_0\in\real_{++},~p_0:=1-\sum_{i=1}^d p_i>0)$	
			\end{tabular}                                                                                                                                                 &

				$x^+\in\real^d_+\cap I(p):~
				\log\left( \frac{x_i^+}{\bar{x}_i} \right)+t\log\left( \frac{x_i^+}{p_i(x_0+\sum_{j=1}^dx^+_j)} \right) = 0,
				$
			\medskip\\
%			Degenerate $(a\in\real^d)$                                                                                                                                           &              $x^+=a$                                                                                                                                                                                                                                                                                                                                                                          \medskip\\
			\def\arraystretch{1}\begin{tabular}[c]{@{}l@{}} Discrete Uniform \\ $(a,b\in\real:a<b)$      \end{tabular}                                                                  &       \def\arraystretch{1}\begin{tabular}[t]{@{}l@{}}$x^+=\bar{x}e^{-t\theta^+}$
				where $\theta^+=0$ if $\bar{x}=(a+b)/2$, \medskip\\
				otherwise: $\theta^+\in\real\setminus\{0\}$:\medskip\\
				$\quad \frac{(b+1)\mathrm{exp}((b+1)\theta^+) - a\mathrm{exp}(a\theta^+)}{\mathrm{exp}((b+1)\theta^+) - \mathrm{exp}(a\theta^+)}
				= \frac{\mathrm{exp}(\theta^+)}{\mathrm{exp}(\theta^+)-1} + \mathrm{exp}(\bar{x} - t\theta^+ -1)$
			\end{tabular}                                                                                                                                                                                                                                                                                \medskip\\
			\def\arraystretch{1}\begin{tabular}[c]{@{}l@{}}Continuous Uniform \\ $(a,b\in\real:a\leq b)$ \end{tabular}                                                                         &       \def\arraystretch{1}\begin{tabular}[t]{@{}l@{}}$x^+=\bar{x}e^{-t\theta^+}$
				where $\theta^+=0$ if $\bar{x}=(a+b)/2$, \medskip\\
				otherwise: $\theta^+\in\real\setminus\{0\}$:\medskip\\
				$\qquad \frac{b\mathrm{exp}(b\theta^+) - a\mathrm{exp}(a\theta^+)}{\mathrm{exp}(b\theta^+) - \mathrm{exp}(a\theta^+)}
				= \frac{1}{\theta^+} + \mathrm{exp}(\bar{x} - t\theta^+ -1)$
			\end{tabular}                                                                                                                                                                                                                                                                                                                                                                                                                                     \medskip\\
			Logistic $(\mu\in\real,~s\in\real_{++})$:                                                                                                                                      & \def\arraystretch{1}\begin{tabular}[t]{@{}l@{}}$x^+=\bar{x}e^{-t\theta^+}$
				where $\theta^+=0$ if $\bar{x}=\mu$, \medskip\\
				otherwise: $\theta^+\in\real\setminus\{0\}$:\medskip\\
				$\qquad \frac{1}{\theta^+} + \frac{\pi s}{\tan(-\pi s\theta^+)} + \mu = \exp\left({\bar{x} - t\theta^+ -1}\right)$
			\end{tabular}                                                                              
			\\ \bottomrule
			%\end{tabular}
			%\hline
		\caption{Bregman Proximal Operators - Poisson Linear Model $(h_j(x)=x_j\log x_j)$}
			%\endlastfoot
			\label{algo:tbl:breg_prox_Poisson}
		\end{longtable}
	
	%\end{table}
}

% \footnotetext{We denote by $W:\real\rightarrow\real$ the Lambert $W$ function (see, for example, \cite{corless1996lambertw}).}

%\begin{table}[H]
{\footnotesize
	\centering
	\def\arraystretch{2}
	%	\begin{tabular}{@{}L{5.8cm}c @{}}
		\begin{longtable}{lc}%{| p{.20\textwidth} | p{.80\textwidth} |} 
			%\begin{tabular}{lc}
			\toprule 
			\multicolumn{1}{c}{Reference Distribution ($R$)}                                                                                                                           & Bregman Proximal Operator $(x^+=\prox_{t\psi_R^*}^h(\bar x))$                                                                                                                                                                                                                                                                                                                                                                   \\ \midrule
			%		\endhead
			\endfirsthead

			%		\bottomrule
			%		\endfoot
			%\bottomrule
		%	\caption{Bregman proximal operators - Gamma ($\beta=1$) linear model.}
	
			{\color[HTML]{333333} \def\arraystretch{1}\begin{tabular}[c]{@{}l@{}} Normal\\ $(\mu,\sigma\in\real:~\sigma>0)$\end{tabular}}                                & {\color[HTML]{333333} $x^+=\left((t/\sigma)\mu-1/\bar{x}+\sqrt{((t/\sigma)\mu-1/\bar{x})^2+4(t/\sigma)}\right)/(2t/\sigma)$}                                           
	          \medskip\\
			\def\arraystretch{1}\begin{tabular}[c]{@{}l@{}}Normal-inverse Gaussian\\ $\big(\mu,\alpha,\beta,\delta\in\real:~\delta>0$,\\ 
				$\quad\alpha\geq|\beta|,
				~\gamma:=\sqrt{\alpha^2-\beta^2}\big)$\end{tabular}  &   
				  \def\arraystretch{1}\begin{tabular}[t]{@{}l@{}}
				  	$x^+\in\real_{++}:$ \\$t\alpha(x^+-\mu)x^+=\left((t\beta-1/\bar{x})x^++1\right)\sqrt{\delta^2+(x^+-\mu)^2}$%\medskip\\
%				%where $	\rho\in\real:
%				$\qquad(\rho\delta)^2+\|\left(\rho^{-1}I+\Sigma^{-1}\right)^{-1}\left(t\beta+\bar{x}-\mu \right)\|_{\Sigma^{-1}}^2=(\alpha t)^2$
			\end{tabular}	 
\medskip\\
\def\arraystretch{1}\begin{tabular}[c]{@{}l@{}}Multivariate Normal-inverse \\ 
	Gaussian $\big(\mu,\beta\in\real^d,~\alpha,\delta\in\real,$\\ 
	$\Sigma = \sigma I, \sigma > 0:\delta>0,~\Sigma\succ0,$\\
	$\alpha^2\geq\beta^T\Sigma\beta$,
	$\gamma:=\sqrt{\alpha^2-\beta^T\Sigma\beta}\big)$\end{tabular}  &     \def\arraystretch{1}\begin{tabular}[t]{@{}l@{}}\hspace{1.5cm}$x^+_i=(w_i+\rho \mu_i+\sqrt{(w_i+\rho \mu_i)^2+4\rho})/(2\rho)$, \medskip\\ with $w_i=t\beta_i-1/\bar{x}_i$ and $\rho\in\real_+:$\medskip\\
	%where $	\rho\in\real:
%	where $\rho\in\real_+:$ $\qquad	(\rho\delta)^2 - (\alpha t/\sigma)^2=$
%	\medskip\\
%	$\qquad	\frac{1}{4\sigma}\sum_{i=1}^d\left(t\beta_i-\frac{1}{\bar{x}_i}+\sqrt{(t\beta_i-1/\bar{x}_i+\mu_i\rho)^2+4\rho}\right)^2$
	$\qquad(\rho\delta)^2+\frac{1}{4\sigma}\sum_{i=1}^d\left(w_i+\sqrt{(w_i+\mu_i\rho)^2+4\rho}\right)^2  =(\alpha t/\sigma)^2$
%	$\qquad	(\rho\delta)^2+\frac{1}{4\sigma}\sum_{i=1}^d\left(t\beta_i-\frac{1}{\bar{x}_i}+\sqrt{(t\beta_i-1/\bar{x}_i+\mu_i\rho)^2+4\rho}\right)^2  =(\alpha t/\sigma)^2$
\end{tabular}	 \medskip\\
			%		\def\arraystretch{1}\begin{tabular}[c]{@{}l@{}}Normal-inverse Gaussian \\ $(\alpha,\beta,\gamma,\delta\in\real: \alpha\geq|\beta|,$\\
				%			$\qquad\delta>0,\gamma=\sqrt{\alpha^2-\beta^2})$\end{tabular} &     $x^+\in\real:\frac{\alpha t(x^+-\mu)}{\sqrt{\delta^2+(x^+-\mu)^2}}=\beta t-(x^+-\bar{x})$                                                                                                                                                                                                                                                                                                                                                                                   \medskip\\
			Gamma $(\alpha,\beta\in\real_{++})$                                                                                                                                                               &        
			 $x^+=\bar{x}(t\alpha+1)/(\bar{x}t\beta+1)$                                                             \medskip\\
			Laplace $(\mu\in\real,~b\in\real_{++})$                                                                                                                                       & \def\arraystretch{1}\begin{tabular}[t]{@{}l@{}}$\hspace{3.3cm}x^+=\begin{cases}\mu,& \bar{x}=\mu,\\\mu+b\rho,& \bar{x}\neq\mu,\end{cases}$\medskip\\ where $\rho\in\real%{\ed\real_{++}-\{\mu/b\}}
				:\quad\alpha_1\rho^3+\alpha_2\rho^2+\alpha_3\rho+\alpha_4=0$,\medskip\\ with $\alpha_1=b^2((b/\bar{x})^2-t^2),~\alpha_2=2b(\mu((b/\bar{x})^2-t^2)-b^2(t+1)/\bar{x}),$\medskip\\ 
				$\alpha_3=b^2((1-\mu/\bar{x})^2+2t(1-2\mu/\bar{x}))-t^2\mu^2,~\alpha_4=2tb\mu(1-\mu/\bar{x})$	
			\end{tabular} \medskip\\
			Poisson $(\lambda\in\real_{++})$                                                                                                                                     &                            $x^+\in\real_+:~t\log\left(\frac{x^+}{\lambda}\right)=\frac{1}{x^+}-\frac{1}{\bar{x}}$     \medskip\\
			Multinomial $( n\in\nn, p\in\ri{\Delta_{(d)}})$                                                                                                                                &    $x^+\in \ri{n\Delta_{(d)}}:~ t\log\left(\frac{x^+_i(1-\sum_{j=1}^dp_j)}{p_i(n-\sum_{j=1}^dx^+_j)}\right)=\frac{1}{x_i^+}-\frac{1}{\bar{x}_i}$		
%			\def\arraystretch{1}\begin{tabular}[t]{@{}l@{}} $x^+\in\real^d_+\cap I(p):~\sum_{i=1}^d x^+_i=n$	\medskip\\ 
%				and  $\quadt\left[\log\left(\frac{x^+_i}{np_i}\right)+1\right]=\lambda+\frac{1}{x_i^+}-\frac{1}{\bar{x}_i}$\end{tabular}	                                                                                                                                                                              
\medskip\\
			\def\arraystretch{1}\begin{tabular}[c]{@{}l@{}} Negative Multinomial $(p\in (0,1)^d,$\\
				$x_0\in\real_{++},~p_0:=1-\sum_{i=1}^d p_i>0)$	
			\end{tabular}                                                                                                                                                 &

				$x^+\in\real^d_{++}:~
				t\log\left(\frac{x^+_i}{p_i(x_0+\sum_{i=j}^dx^+_j)}\right)=\frac{1}{x_i^+}-\frac{1}{\bar{x}_i},
				$
			\medskip\\
%			Degenerate $(a\in\real^d)$                                                                                                                                           &              $x^+=a$                                                                                                                                                                                                                                                                                                                                                                          \medskip\\
			\def\arraystretch{1}\begin{tabular}[c]{@{}l@{}} Discrete Uniform \\ $(a,b\in\real:a<b)$      \end{tabular}                                                                  &       \def\arraystretch{1}\begin{tabular}[t]{@{}l@{}}$x^+=\bar{x}/(\bar{x}t\theta^++1)$
				where $\theta^+=0$ if $\bar{x}=(a+b)/2$, \medskip\\
				otherwise: $\theta^+\in\real\setminus\{0\}$:\medskip\\
				$\quad \frac{(b+1)\exp\left((b+1)\theta\right)-a\exp\left(a\theta\right)}{(\exp\left((b+1)\theta\right)-\exp\left(a\theta\right)}=\frac{\exp\left(\theta\right)}{\exp\left(\theta\right)-1}+\frac{\bar{x}}{t\bar{x}\theta^++1}$
			\end{tabular}                                                                                                                                                                                                                                                                                \medskip\\
			\def\arraystretch{1}\begin{tabular}[c]{@{}l@{}}Continuous Uniform \\ $(a,b\in\real:a\leq b)$ \end{tabular}                                                                         &       \def\arraystretch{1}\begin{tabular}[t]{@{}l@{}}$x^+=\bar{x}/(\bar{x}t\theta^++1)$
				where $\theta^+=0$ if $\bar{x}=(a+b)/2$, \medskip\\
				otherwise: $\theta^+\in\real\setminus\{0\}$:\medskip\\
					$\qquad\frac{b \exp(b \theta^+)-a \exp(a \theta^+)}{\exp(b \theta^+)-\exp(a \theta^+)}=\frac{1}{\theta^+}+\frac{\bar{x}}{t\bar{x}\theta^+ + 1}$
			\end{tabular}                                                                          \medskip\\
			Logistic $(\mu\in\real,~s\in\real_{++})$:                                                                                                                                      & \def\arraystretch{1}\begin{tabular}[t]{@{}l@{}}$x^+=\bar{x}/(\bar{x}t\theta^++1)$
				where $\theta^+=0$ if $\bar{x}=\mu$, \medskip\\
				otherwise: $\theta^+\in\real\setminus\{0\}$:\medskip\\
				$\qquad\frac{1}{\theta^+}+\frac{\pi s}{\tan{(-\pi s\theta^+)}}
				+\mu=\frac{\bar{x}}{\bar{x}t\theta^++1}$
			\end{tabular}                                                                              
			\\ \bottomrule
			%\end{tabular}
			%\hline
			\caption{Bregman Proximal Operators - Gamma ($\beta=1$) Linear Model ($h_j(x)=-\log(x_j)$)}
			%\endlastfoot
			\label{algo:tbl:breg_prox_Gamma}
		\end{longtable}
	%\end{table}
}

We close our study with particular models and algorithms.
\medskip 

\noindent{\bf Barcode Image Deblurring.} Restoration of a blurred and noisy image represented by a vector $\hat{y}\in\real^d$ can be cast as the following optimization problem:
\begin{gather}
	\label{algos:eq:barcode_model}
	\min\left\{\frac{1}{2}\|Ax-\hat{y}\|_2^2+\tau\varphi_R^*(x):x\in\real^d\right\}.
\end{gather}
$A\in\real^{d\times d}$ is the  blurring operator and $\tau>0$ is a regularization parameter. The noise is assumed to be Gaussian which explains the least-squares fidelity term which can be justified from the viewpoint of both the ML and, as we know from our study, the  MEM framework.
%The reference measure $R\in\PPP(\Omega)$ that induces the regularization term $\varphi_R^*$. 
If the original image is a 2D barcode, a natural choice for the reference measure $R\in\PPP(\Omega)$ inducing $\varphi_R^*$ is a separable Bernoulli distribution with $p=1/2$ due to the binary nature of each pixel and no preference at each pixel to take either value.\footnote{As mentioned in \cref{mem:rmrk:special_cases_and_interpretation}, Bernoulli is a special case of the multinomial distribution. This, one dimensional, distribution is used to form a $d$-dimensional i.i.d as described in \cref{mem:rmrk:Cramer_seperability}.}  Additional information (symbology)  can be easily incorporated  by an appropriate adjustment of the parameter for each known pixel (see \cite{rioux2019blind}). Using the appropriate proximal operator from \cref{algo:tbl:breg_prox_Normal}, the BPG method for solving the model takes the form 
\begin{gather*}
	x^{k+1}_i\in\real: \quad x^{k+1}_i+t\tau\log\left(\frac{x^{k+1}_i}{1-x^{k+1}_i}\right)=x^k_i-t[A^T(Ax^k-\hat{y})]_i,\quad(i=1,2,\dots,d).
\end{gather*}
As mentioned above, our focus on the Bregman proximal gradient method is only for illustration purposes. Favorable accelerated algorithms that employ the proximal operators derived in this work are readily available and should be used in practice. The acceleration scheme applicable here is known as the Fast Iterative Shrinkage Thresholding Algorithm (FISTA) \cite{beck2009fast}.
\\

% \subsubsection{Natural Image Deblurring}
% \label{subsec:algos_exmaple_psuedohuber}

\noindent{\bf Natural Image Deblurring.} For natural image deblurring there is no 
obvious structure such as the binary one for barcodes. However, it is customary to assume that the image is piecewise smooth. A popular model that promotes piecewise constant restoration is the Rudin, Osher, and Fatemi (ROF) model \cite{rudin1992nonlinear} based on the total variation (TV) regularizer $\sum_{i=1}^dg(L_ix)$. 
Here,  $L_i\in\real^{2\times d}$ extracts the difference between the pixel $i$ and two adjacent pixels while $g$ stands for either the $l_1$ (isotropic TV) or $l_2$ (anisotropic TV) norm.  Variants that admit the same structure with other choices of  $g$ are also considered in the literature: in \cite[Subsection 6.2.3]{chambolle2011first}, a model with the Huber norm for  $g$ was shown to promote restoration prone to artificial flat areas. Alternatively, one may consider the pseudo-Huber norm that corresponds to an MEM regularizer induced by the multivariate normal inverse-Gaussian reference distribution with parameters $\mu=\beta=0,~\alpha=1$, and $\Sigma=I$. The resulting model is similar to \eqref{algos:eq:barcode_model} where the regularization term is substituted by $\sum_{i=1}^d\psi_R^*(L_ix)$. This model can be tackled by a primal-dual decomposition method that employs the appropriate proximal operator from \cref{algo:tbl:breg_prox_Normal}. For example, using the separability of the proximal operator \cite[Theorem 6.6]{beck2017first} and the extended Moreau decomposition \cite[Theorem 6.45]{beck2017first}, the update formula of the Chambolle-Pock algorithm \cite[Algorithm 1]{chambolle2011first} reads %for $i=1,2,\dots,d$:
% \[
% y_i^{k+1} 
% 	=\frac{\rho_i}{1+\rho_i}(y^k+sL_i z_i^k) \quad\text{with}\quad \rho_i\in\real_+:\rho_i^2(s\delta)^2+\left(\frac{\rho_i}{1+\rho_i}\right)^2\|y_i^k+sL_i z_i^k\|_2^2=1,
% \]
% \[
% x^{k+1} 
% 	&=(I+\tau A^TA)^{-1}\left(x^k-\tau(L^Ty^{k+1}-A^T\hat{y})\right),\quad z^{k+1} = 2x^{k+1}-x^k,
% \]
\begin{gather*}
	\def\arraystretch{2}
	\begin{array}{rll}
		y_i^{k+1} 
	&=\frac{\rho_i}{1+\rho_i}(y^k+sL_i z^k) &(i=1,2,\dots,d),\\ \text{with}&\rho_i\in\real_+:\rho_i^2(s\delta)^2+\left(\frac{\rho_i}{1+\rho_i}\right)^2\|y_i^k+sL_i z^k\|_2^2=1,&\\	
	x^{k+1} 
	&=(I+\tau A^TA)^{-1}\left(x^k-\tau(L^Ty^{k+1}-A^T\hat{y})\right),&\\ z^{k+1} &= 2x^{k+1}-x^k,&
	\end{array}
\end{gather*}
% \begin{gather*}
% 	\def\arraystretch{2}
% 	\begin{array}{rl}
% 		y_i^{k+1} 
% 	&=\frac{\rho_i}{1+\rho_i}(y^k+sL_i z^k) \quad\text{with}\quad \rho_i\in\real_+:\rho_i^2(s\delta)^2+\left(\frac{\rho_i}{1+\rho_i}\right)^2\|y_i^k+sL_i z^k\|_2^2=1,\\	
% 	x^{k+1} 
% 	&=(I+\tau A^TA)^{-1}\left(x^k-\tau(L^Ty^{k+1}-A^T\hat{y})\right),\\ z^{k+1} &= 2x^{k+1}-x^k,
% 	\end{array}
% \end{gather*}
where $L^T=[L_1^T,\dots,L_d^T]\in\real^{d\times 2d}$, $y^k\in\real^{2d}:(y^k)^T=[(y_1^k)^T, \dots, (y_d^k)^T]$ with $y_i^k\in\real^2$ for all $i=1,2,\dots,d$) and $s,\tau$ are some positive step-sizes satisfying $s\tau\|L\|_2^2<1$.\withsmallskip

We point out that an efficient implementation of the above algorithm that takes into account the sparse and structured nature of the matrices $L$ and $A$, respectively, will result in a per-iteration complexity of the order $O(d\log{d})$. The same statement is true with regard to the BPG method in the previous and following examples. \\

% \subsubsection{Poisson Linear Inverse Problem}

\noindent {\bf Poisson Linear Inverse Problem.} Poisson linear inverse problems play a prominent role in various physical and medical imaging applications. The linear model proposed in \cite[Subsection 5.3]{bauschke2017descent} is simply the  MEM linear model with Poisson reference distribution. The authors of \cite{bauschke2017descent} suggest  $l_1$-regularization to deploy their BPG method. Alternatively, one may consider the MEM function induced by the  Laplace distribution with parameters $\mu=0$ and $b=1$. This setting leads to the following update formula of the BPG method. For $i=1,2,\dots,d$:
\begin{gather*}
\def\arraystretch{2}
	\begin{array}{rl}
\bar{x}^{k+1}_i  &\displaystyle=\exp\left(\log(x_i^k)-t\sum_{j=1}^ma_{ji}\log(\inner{a_j}{x^k}/\hat{y}_j)\right),    \\
x^{k+1}_i\in\real:  &\displaystyle t^2x_i^{k+1}+2t\log\left(\frac{x_i^{k+1}}{\bar{x}_i^{k+1}}\right) = x_i^{k+1}\left[\log\left(\frac{x_i^{k+1}}{\bar{x}_i^{k+1}}\right)\right]^2.
\end{array}
\end{gather*}

% \[
% \bar{x}^{k+1}_i  \displaystyle=\left[(x^k_i)^{-1}+t\left(\sum_{j=1}^ma_j\left(1-\frac{\hat{y}_j}{\inner{a_j}{x^k}}\right)\right)_i~\right]^{-1}
% \]

% \[
% x^{k+1}_i=\left\{
% 	\begin{array}{ll}\displaystyle
% 		\frac{\bar{x}^{k+1}_i\left[1+t\left(1+\sqrt{(2t+1)(\bar{x}^{k+1}_i)^2+1}\right)\right]}{1-t^2(\bar{x}^{k+1}_i)^2}, & \bar{x}^{k+1}_i<1/t,\\\displaystyle
% 		\frac{(\bar{x}^{k+1}_i)(2t+1)}{2(t+1)}, & \bar{x}^{k+1}_i=1/t,\\\displaystyle
% 		\frac{\bar{x}^{k+1}_i\left[1+t\left(1-\sqrt{(2t+1)(\bar{x}^{k+1}_i)^2+1}\right)\right]}{1-t^2(\bar{x}^{k+1}_i)^2}, & \bar{x}^{k+1}_i>1/t.
% 	\end{array}\right.
% \]

%\begin{gather*}
%	\small
%	\def\arraystretch{2}
%	\begin{array}{rl}
%		\bar{x}^{k+1}_i & \displaystyle=\left[(x^k_i)^{-1}+t\left(\sum_{j=1}^ma_j\left(1-\frac{\hat{y}_j}{\inner{a_j}{x^k}}\right)\right)_i~\right]^{-1},\bigskip\\
%	\def\arraystretch{2.8}
%	x^{k+1}_i&=\left\{
%	\begin{array}{ll}\displaystyle
%		\frac{\bar{x}^{k+1}_i\left[1+t\left(1+\sqrt{(2t+1)(\bar{x}^{k+1}_i)^2+1}\right)\right]}{1-t^2(\bar{x}^{k+1}_i)^2}, & \vbar{x}^{k+1}_i<1/t,\\\displaystyle
%		\frac{(\bar{x}^{k+1}_i)(2t+1)}{2(t+1)}, & \bar{x}^{k+1}_i=1/t,\\\displaystyle
%		\frac{\bar{x}^{k+1}_i\left[1+t\left(1-\sqrt{(2t+1)(\bar{x}^{k+1}_i)^2+1}\right)\right]}{1-t^2(\bar{x}^{k+1}_i)^2}, & \bar{x}^{k+1}_i>1/t.
%	\end{array}\right.	
%	\end{array}
%\end{gather*}

\bibliographystyle{plain}
\bibliography{MEM_references}

\appendix

\section{Cram\'er Rate Functions}
\label{apndx:sec:Cramer}
We present here the computations of all Cram\'er rate functions provided during our study. To this end, recall that the Cram\'er rate function $\psi_P^*$ is the  conjugate 
\[
\psi_P^*(y):=\sup\{\inner{y}{\theta}-\psi_P(\theta): \theta\in \real^d\}
\]
of the cumulant-generating function
\begin{gather*}
	\psi_P(\theta):=\log  M_P[\theta],
\end{gather*}
where $M_P[\theta]$ is the moment-generating function of the reference distribution $P\in\PPP(\Omega)$ (which one can simply look up at various places in the literature for the distributions considered here).

\subsubsection*{Multivariate Normal} For a normal distribution  with mean $\mu$ and covariance $\Sigma\succ 0$, its moment generating function is $M_P[\theta]=\exp(\inner{\mu}{\theta} +{\frac {1}{2}}\inner{\theta}{\Sigma \theta})$. Therefore, we find

\begin{gather*}
	\def\arraystretch{2}
	\begin{array}{rl}
		\psi_P^*(y)	
		&= \sup\left\{\inner{y}{\theta}-\log\left(\exp({\inner{\mu}{\theta} +{\frac {1}{2}}\inner{\theta}{\Sigma \theta}})\right):\theta\in\real^d\right\}\\
		&= \sup\left\{\inner{y}{\theta}-\inner{\mu}{\theta} -{\frac {1}{2}}\inner{\theta}{\Sigma \theta}:\theta\in\real^d\right\}.
	\end{array}
\end{gather*}
The maximumizer of the above quadratic optimization problem is  $\theta^* = \Sigma^{-1}(y-\mu)$, hence
\begin{gather*}
	\psi_P^*(y) = \frac{1}{2}(y-\mu)^T\Sigma^{-1}(y-\mu).
\end{gather*}

\subsubsection*{Multivariate Normal-inverse Gaussian}
The Multivariate Normal-inverse Gaussian distribution is defined by means of location ($\mu\in\real^d$), tail heaviness ($\alpha\in\real$), asymmetry ($\beta\in\real^d$), and scale ($\delta\in\real,~\Sigma\in\real^{d\times d}$) parameters satisfying $\alpha\geq\sqrt{\inner{\beta}{\Sigma\beta}}$, $\delta>0$ and $\Sigma\succ0$ \cite{barndorff1997normal}. In addition, let $\gamma:=\sqrt{\alpha^2-\inner{\beta}{\Sigma\beta}}$. Its moment-generating function is
\[
M_{P}[\theta]=\exp\left(\inner{\mu}{\theta}+\delta(\gamma-\sqrt{\alpha^2-\inner{\beta+\theta}{\Sigma(\beta+\theta)}})\right)\quad (\theta\in B_\alpha),
\]
for the ellipsoid $B_\alpha=\{\theta\in \real^d:\sqrt{\inner{\beta+\theta}{\Sigma(\beta+\theta)}}\leq\alpha\}.$  Observe that in this case $\psi_P(\theta)=\log\left(M_P[\theta]\right)$ is indeed steep and minimal.\withsmallskip

Now, in order to compute the Cram\'er rate function, we find that
\begin{equation}\label{app:eq:Cramer_NIG}
\psi_P^*(y)=\sup\left\{\inner{y-\mu}{\theta}-\delta(\gamma-\sqrt{\alpha^2-\inner{\beta+\theta}{\Sigma(\beta+\theta)}}):\theta\in B_\alpha\}\right.
\end{equation}
We consider two cases: if $y=\mu$, then it is evident that the optimal solution of the problem above is given by $\theta=-\beta$ and thus $\psi_P^*(\mu)=\delta(\alpha-\gamma)$. Consider the case $y\neq\mu$. Disregarding the feasibility constraints (which will be justified in the sequel), the first-order optimality condition is given by
\begin{gather*}
	y-\mu=\frac{\delta\Sigma(\beta+\theta)}{\sqrt{\alpha^2-\inner{\beta+\theta}{\Sigma(\beta+\theta)}}}.%\quad \Rightarrow\quad 
	%\theta = -\beta+\frac{\alpha(y-\mu)}{\sqrt{\delta^2+(y-\mu)^2}}.
\end{gather*}
From the above, we can derive 
\begin{gather*}
\inner{\beta+\theta}{\Sigma(\beta+\theta)} = \frac{\alpha^2\inner{y-\mu}{\Sigma^{-1}(y-\mu)}}{\delta^2+\inner{y-\mu}{\Sigma^{-1}(y-\mu)}}\quad\text{and}\quad\theta=-\beta+\frac{\alpha\Sigma^{-1}(y-\mu)}{\sqrt{\delta^2+\inner{y-\mu}{\Sigma^{-1}(y-\mu)}}}.
\end{gather*}
It is straightforward to verify that $\theta\in \inte B_\alpha$,  which retroactively justifies our choice to disregard the constraint before. Now, we can write the Cram\'er rate function as
\begin{gather*}
	\def\arraystretch{2}
	\begin{array}{rl}	&\psi_P^*(y)\\
	&=\displaystyle \inner{y-\mu}{-\beta+\frac{\alpha\Sigma^{-1}(y-\mu)}{\sqrt{\delta^2+\inner{y-\mu}{\Sigma^{-1}(y-\mu)}}}}-\delta\left(\gamma-\sqrt{\alpha^2-\frac{\alpha^2\inner{y-\mu}{\Sigma^{-1}(y-\mu)}}{\delta^2+\inner{y-\mu}{\Sigma^{-1}(y-\mu)}}}\right)\\
		&\displaystyle=\alpha\sqrt{\delta^2+(y-\mu)^T\Sigma^{-1}(y-\mu)}-\inner{\beta}{y-\mu}-\delta\gamma.
%		&\displaystyle=\alpha\sqrt{\delta^2+\inner{y-\mu}{\Sigma^{-1}(y-\mu)}}-\inner{\beta}{y-\mu}-\delta\gamma.
	\end{array}
\end{gather*}

\subsubsection*{Gamma} The Gamma distribution is parametrized by $\alpha,\beta>0$ and its moment generating function  is given by 
\[
M_{P}[\theta]=\left[1-\frac{\theta}{\beta}\right]^{-\alpha}\quad (\theta< \beta).
\]
Hence, its Cram\'er rate function reads
\begin{gather*}
	\def\arraystretch{2}
	\begin{array}{rl}
		\psi_P^*(y)
		&= \sup\left\{y\theta-\log\left(\left[1-\frac{\theta}{\beta}\right]^{-\alpha}\right):\theta<\beta\right\}\\
		&= \sup\left\{y\theta+\alpha\log\left(1-\frac{\theta}{\beta}\right):\theta<\beta\right\}.
	\end{array}
\end{gather*}
If $y\leq0$, then $\psi_P^*(y)=+\infty$ (with $\theta\rightarrow-\infty$). If $y>0$ then the first-order optimality conditions imply
\begin{gather*}
	y-\frac{\alpha}{\beta}\left(1-\frac{\theta}{\beta}\right)^{-1}=0\quad \Rightarrow \quad \theta=\beta-\frac{\alpha}{y}.
\end{gather*}
Thus, 
\begin{gather*}
	\psi_P^*(y) = \beta y-\alpha+\alpha \log\left(\frac{\alpha}{\beta y}\right), \quad y\in\real_{++}.
\end{gather*}

\subsubsection*{Laplace} The Laplace distribution is parameterized by its mean $\mu\in \real$ and scale $b>0$. Its MGF reads
\[
M_P[\theta]=\frac{\exp(\mu \theta)}{1-b^2\theta^2}\quad (|\theta|<1/b)
\]
Hence, its Cram\'er rate function reads
\[
\psi_P^*(y)=\sup\left\{(y-\mu)\theta+\log\left(1-b^2\theta^2\right):|\theta|<1/b\right\}.
\]
It is easy to see that $\log\left(1-b^2\theta^2\right)\leq 0$ for any $\theta$ such that $|\theta|<1/b$ and that $\log\left(1-b^2\theta^2\right)\rightarrow-\infty$ when $|\theta|\rightarrow 1/b$. Thus, we can conclude that 
$\psi_P^*(\mu)=0$ and for any $y\neq\mu$ the maximum of the above problem is attained at some point in the open interval $(0,1/b)$ for $y>\mu$ or in $(-1/b,0)$ for $y<\mu$. The first-order optimality conditions boil down to the quadratic equation
\[
\theta^2+\left(\frac{2}{y-\mu}\right)\theta-\frac{1}{b^2}=0
\]
Evaluating the roots of the resulting quadratic equation we conclude that the optimal solution is
\begin{gather*}
	\theta=\frac{1}{y-\mu}\left(\sqrt{1+\left(\frac{y-\mu}{b}\right)^2}-1\right)=\frac{1}{b\rho}\left(\sqrt{1+\rho^2}-1\right),
\end{gather*}
%\begin{gather*}
%	y=\begin{cases}
%		\frac{1}{x-\mu}\left(\sqrt{1+\left(\frac{x-\mu}{b}\right)^2}-1\right), & x>\mu,\\
%		\frac{1}{\mu-x}\left(1-\sqrt{1+\left(\frac{x-\mu}{b}\right)^2}\right), & x<\mu.
%	\end{cases}
%\end{gather*}
where we set $\rho:=\frac{y-\mu}{b}$. Evidently, $|\theta|<1/b$ holds for the solution we just derived.
%\begin{gather*}
%	-\frac{1}{x-\mu}-\sqrt{\frac{1}{(x-\mu)^2}+\frac{1}{b^2}}=-\frac{1}{x-\mu}-\frac{1}{|x-u|}\sqrt{1+\left(\frac{x-\mu}{b}\right)^2}
%\end{gather*}
Thus
\begin{gather*}
	\def\arraystretch{2}
	\begin{array}{rl}
		\psi_P^*(y) &= (y-\mu)\theta+\log\left(1-(b\theta)^2\right) \\
		&= \rho(b\theta)+\log\left(1-(b\theta)^2\right) \\
		&= \sqrt{1+\rho^2}-1+\log\left(1-\frac{1}{\rho^2}(\sqrt{1+\rho^2}-1)^2\right) \\
		&=\sqrt{1+\rho^2}-1+\log\left(1-\frac{1}{\rho^2}(1+\rho^2+1-2\sqrt{1+\rho^2})\right)\\
		&=\sqrt{1+\rho^2}-1+\log\left(\frac{2}{\rho^2}(\sqrt{1+\rho^2}-1)\right),
		%&= \sqrt{1+\rho^2}-1+\log\left(1-(\sqrt{\rho^{-2}+1}-|\rho|^{-1})^2\right),
	\end{array}				 
\end{gather*}
and we can conclude that
\begin{gather*}
	\psi_P^*(y) = \begin{cases}
		0,& y=\mu,\\
		\sqrt{1+\left(\frac{y-\mu}{b}\right)^2}-1+\log\left(2\left(\frac{y-\mu}{b}\right)^{-2}\left[\sqrt{1+\left(\frac{y-\mu}{b}\right)^2}-1\right]\right),& y\neq \mu.		
		%+\log\left(1-\left[\sqrt{\left(\frac{x-\mu}{b}\right)^{-2}+1}-\left|\frac{x-\mu}{b}\right|^{-1}\right]^2\right),& x\neq \mu.
	\end{cases}
\end{gather*}

\subsubsection*{Poisson} The Poisson distribution is parameterized by its rate $\lambda>0$. Its MGF reads
\[
M_P[\theta]=\exp(\lambda(\exp(t)-1)
\]
Consequently, its Cram\'er rate function is given by
\[
\psi_P^*(y)=\sup\left\{y\theta-\lambda(\exp(\theta)-1):\theta\in\real\right\}.
\]
If $y<0$ then it is evident from the above that $\psi_P^*(y)=+\infty$ (indeed, take $\theta\rightarrow -\infty$). 
Similarly, we can see that $\psi_P^*(0)=\lambda$. Otherwise, due to the first-order optimality conditions
\begin{gather*}
	y=\lambda \exp(\theta) \quad \Rightarrow\quad \theta=\log(y/\lambda),
\end{gather*}
we obtain that $\psi_P^*(y) = y\log(y/\lambda)- y+ \lambda$.

\subsubsection*{Multinomial}

We will use the following notation. The $i$th canonical unit vector is denoted by $e_i$ and the vector of all ones is denoted by $e$. The unit simplex is given by $\Delta_d:=\{y\in\real^d_+:\inner{e}{y}=1\}$.\withsmallskip

For $n\in\nn$ and $p\in\Delta_{d+1}$ we can write 
\begin{gather*}
	\def\arraystretch{2}
	\begin{array}{rl}
		\psi_P^*(\theta)&\displaystyle=\sup\left\{l(y,\theta):=\inner{y}{\theta}-\log\left(M_P[\theta]\right):\theta\in\real^{d+1}\right\}\\	
		&\displaystyle=\sup\left\{\inner{y}{\theta}-n\log\left(\sum_{i=1}^{d+1}p_i\exp(\theta_i)\right):\theta\in\real^{d+1}\right\}.	
	\end{array}
\end{gather*}
Let $I(p):=\left\{y\in\real^{d+1}: y_i=0~(p_i=0,~i=1,2,\dots,d+1)\right\}$. 
We can see that $\dom{\psi_P^*}=n\Delta_d\cap I(p)$. 
Indeed, if there exists $k\in\{1,2,\dots,d+1\}$ such that $y_k<0$ then by setting $\theta=-\alpha e_k$ we obtain that 
\begin{gather*}
	l(y,\theta) = \alpha|y_k| - n\log\left(p_k\exp(-\alpha)+\sum_{i\neq k}p_i\right).
\end{gather*}
If, $y\in\real^{d+1}$ but $\inner{e}{y}\neq n$ then by choosing $\theta=\alpha\sigma e$ where $\sigma = ~\text{sign}(\inner{e}{y}-n)$ we obtain that
\begin{gather*}
	l(y,\theta) = \alpha\sigma \inner{e}{y} - n\log\left(\exp(\alpha\sigma)\inner{e}{p}\right)=\alpha|\inner{e}{y} - n|.
\end{gather*}
If there exists $k\in \{i\in\{1,2,\dots,d+1\}:p_i=0\}$ such that $y_k>0$ then by setting $\theta=\alpha e_k$ we obtain
\begin{gather*}
	l(y,\theta) = \alpha y_k - n\log\left(\sum_{i\neq k}p_i\right).
\end{gather*}
In all cases, by taking $\alpha\rightarrow\infty$ it is evident that the problem is unbounded. \withsmallskip

We now address the case when $y\in\dom{\psi_P^*}=n\Delta_{d+1}\cap I(p)$. 
From the first-order optimality condition, we can deduce that for any $j=1,\dots,d+1$ such that $p_j>0$
\begin{gather*}
	y_j = \frac{np_j\exp(\theta_j)}{\sum_{i=1}^{d+1}p_i\exp(\theta_i)} \quad \Rightarrow\quad \theta_j = \log\left(\frac{y_j}{np_j}\right), 
\end{gather*}
for all $j=1,2,\dots,d+1$. Thus, under the convention that $0/0=1$, we can conclude that for $y\in n\Delta_{d+1}\cap I(p)$
\begin{gather*}
	\psi_P^*(y) =\sum_{i=1}^{d+1} y_i\log\left(\frac{y_i}{np_i}\right).
\end{gather*}

Cram\'er's rate function that corresponds to the multinomial distribution after reduction to a minimal form can be obtained from the above by eliminating one component of the vectors $y\in\real^{d+1}$ and $p\in\real^{d+1}$. Assuming, without the loss of generality, that $p_{d+1}>0$ we can plug in the above
\begin{gather*}
	y_{d+1} = n - \sum_{i=1}^{d} y_i, \qquad\text{and}\qquad p_{d+1} = 1 - \sum_{i=1}^{d}p_i, 
\end{gather*}
in order to obtain the Cram\'er rate function $\psi_P^*:\real^{d}\rightarrow\erl$. Hence, for $y\in\real^d$ and $p\in\Delta_{(d)}:=\{z\in\real^d_+:\inner{e}{z}\leq 1\}$ such that $\inner{e}{p}<1$ 
\begin{gather*}
	\psi_P^*(y) = \sum_{i=1}^{d} y_i\log\left(\frac{y_i}{np_i}\right)+ \left(n - \inner{e}{y}\right)\log\left(\frac{n - \inner{e}{y}}{n(1 - \inner{e}{p})}\right),
\end{gather*}
where, in this case, $\dom \psi_P^* = I(p)\cap\Delta_{(d)}$. 
%Alternatively, we could have obtained the above expression by applying \cref{mem:lem:relative_kramer} with the log-normalizer %of the multinomial exponential family in minimal form (see, \cite[Example 1.3]{brown1986fundamentals}).

\subsubsection*{Negative Multinomial}

Observing that $\Theta_P:=\{\theta\in\real^d:\sum_{i=1}^d p_i\exp(\theta_i)<1\}$ and using the definition of Cram\'er's rate function we can write 
\begin{gather*}
	\def\arraystretch{2.5}
	\begin{array}{rl}
		\psi_P^*(\theta)&\displaystyle=\sup\left\{l(y,\theta):=\inner{y}{\theta}-\log\left(M_P[\theta]\right):\theta\in\real^d\right\}\\	
		&\displaystyle=\sup\left\{\inner{y}{\theta}-\log\left(\left[\frac{p_0}{1-\sum_{i=1}^dp_i\exp(\theta_i)}\right]^{y_0}\right):\theta\in\Theta_P\right\}\\
		&\displaystyle=\sup\left\{\inner{y}{\theta}+y_0\log\left(1-\sum_{i=1}^dp_i\exp(\theta_i)\right):\theta\in\Theta_P\right\}-y_0\log(p_0).
	\end{array}
\end{gather*}

%$L(p):=\left\{y\in\real^d: y_i=0~\left(i\in I(p):=\left\{i\in\{1,2,\dots,d\}: p_i=0\right\}\right)\right\}$

%For a given vector $p\in\real^d$ denote $I(p):=\left\{i\in\{1,2,\dots,d\}: p_i=0\right\}$ and $L(p):=\left\{y\in\real^d: y_i=0~\left(i\in I(p)\right)\right\}$

Let $I(p):=\left\{y\in\real^d: y_i=0~(p_i=0,~i=1,2,\dots,d)\right\}$. We can see that $\dom{\psi_P^*}=\real_{+}^d\cap I(p)$. Indeed, if there exists $k\in\{1,\dots,d\}$ such that $y_k<0$ then by setting $\theta=-\alpha e_k$ (recall that $e_k$ stands for the $k$th canonical unit vector) we obtain that 
\begin{gather*}
	%\inner{y}{\theta}-n\log\left(\sum_{i=1}^dp_ie^{\theta_i}\right) 
	l(y,\theta) +  y_0\log(p_0) = \alpha|y_k| + y_0\log\left(1-p_k\exp(-\alpha)-\sum_{i\neq k}p_i\right).
\end{gather*}
If there exists $k\in \{i\in\{1,2,\dots,d\}:p_i=0\}$ such that $y_k>0$ then by setting $\theta=\alpha e_k$ we obtain that
\begin{gather*}
	%\inner{y}{\theta}-n\log\left(\sum_{i=1}^dp_ie^{\theta_i}\right) 
	l(y,\theta)+y_0\log(p_0) = \alpha y_k + y_0\log\left(1-\sum_{i\neq k}p_i\right).
\end{gather*}
In both cases, by taking $\alpha\rightarrow\infty$ it is evident that the problem is unbounded. \withsmallskip

We now address the case when $y\in\dom{\psi_P^*}=\real^d_+\cap I(p)$. From the first-order optimality condition, we can deduce that 
\begin{gather}
	\label{apndx:eq:neg_mult_stat}
	y_j=\frac{y_0p_j\exp(\theta_j)}{1-\sum_{i=1}^dp_i\exp(\theta_i)} \quad \Rightarrow\quad \frac{y_j}{y_0}\left(1-\sum_{i=1}^dp_i\exp(\theta_i)\right) = p_j\exp(\theta_j), 
\end{gather}
for all $j=1,2,\dots,d$. Denoting $\sigma:=\sum_{i=1}^dp_i\exp(\theta_i)$, $\bar y:=\sum_{i=0}^dy_i$ and summing \eqref{apndx:eq:neg_mult_stat} for ${j=1,2,\dots,d}$ yields
\begin{gather*}
	(\bar{y}-y_0)\left(\frac{1-\sigma}{y_0}\right) = \sigma \qquad\Rightarrow\qquad \sigma=\frac{\bar{y}-y_0}{\bar{y}}.
\end{gather*}
The above, combined with \eqref{apndx:eq:neg_mult_stat} we obtain that for any $j=1,2,\dots,d$ such that $p_j\neq 0$ 
\begin{gather*}
	\theta_j = \log\left(\frac{y_j}{p_j\bar{y}}\right).
\end{gather*}
Thus, we can conclude that for $y\in \real^d_+\cap I(p)$ %with $I:=\left\{i\in\{1,\dots,d\}:p_i=0\right\}$
\begin{gather*}
	\psi_P^*(y) = \sum_{i=1}^d y_i\log\left(\frac{y_i}{p_i\bar{y}}\right)+y_0\log\left(\frac{y_0}{\bar{y}}\right)-y_0\log(p_0)=\sum_{i=0}^d y_i\log\left(\frac{y_i}{p_i\bar{y}}\right).
\end{gather*}
It is important to note that in the above $y\in\real^d$ is the function variable while $y_0\in\real$ is a fixed parameter.

%\subsubsection*{Degenerate}
%
%By definition
%\begin{multline*}
%	\psi_P^*(y)=\sup\left\{\inner{y}{\theta}-\log\left(M_P[\theta]\right):\theta\in\real^d\right\}\\=\sup\left\{\inner{y}{\theta}-\log\left(e^{\inner{a}{ \theta}}\right):\theta\in\real^d\right\}=\sup\left\{\inner{y-a}{\theta}:\theta\in\real^d\right\} = \delta_{\{a\}}(y).
%\end{multline*}

\subsubsection*{Discrete Uniform} The discrete uniform distribution is parameterized by $a,b\in \mathbb Z$ with $a\leq b$. We set 
$\mu:=(a+b)/2$ and $n:=b-a+1$. Its MGF reads
\begin{gather*}
	M_P[\theta]=\begin{cases}
		\frac{\exp((b+1)\theta)-\exp(a\theta)}{n(\exp(\theta)-1)}, & \theta\neq 0,\\
		1, & \theta =0.
	\end{cases}
\end{gather*}
If $b=a$ then it is straightforward to verify that $\psi_P^*=\delta_{\{a\}}$ (degenerate distribution). We now turn to consider the case $b>a$. Since $M_P[\theta]$ is continuous at zero, we have
%Thus, we obtain
\begin{gather}
	\label{apndx:eq:Uniform_reformulations_discrete}
	\def\arraystretch{2.5}
	\begin{array}{rl}
		\psi_P^*(y)&=\sup\left\{y\theta-\log\left(\frac{\exp((b+1)\theta)-\exp(a\theta)}{n(\exp(\theta)-1)}\right):\theta\in\real\right\}\\
		&=\sup\left\{(y-b)\theta-\log\left(\frac{\exp(\theta)-\exp(-(b-a)\theta)}{n(\exp(\theta)-1)}\right):\theta\in\real\right\}\\
		&=\sup\left\{(y-a)\theta-\log\left(\frac{\exp((b-a+1)\theta)-1}{n(\exp(\theta)-1)}\right):\theta\in\real\right\}\\	&=\sup\left\{(y-\mu)\theta-\log\left(\frac{\exp((b-\mu+1)\theta)-\exp((a-\mu)\theta)}{n(\exp(\theta)-1)}\right):\theta\in\real\right\}.
	\end{array}
\end{gather}
If $y>b$ then from the second formulation above we can conclude that $\psi_P^*(y)=+\infty$ by taking $\theta\rightarrow+\infty$. Similarly, if, $y<a$, then from the third formulation above we can conclude that $\psi_P^*(y)=+\infty$ by taking $\theta\rightarrow-\infty$. If $y=\mu$ then %it is can be verified from 
the last formulation of \eqref{apndx:eq:Uniform_reformulations_discrete} %and Lemma \ref{apndx:lem:exp_ineq} with $\gamma=\theta(b-a)/2$ that $\psi_P^*(\mu)=0$ for $\theta=0$.
can be written as
\begin{gather*}
%	\label{apndx:eq:cont_uniform_discrete_1}
	\sup\left\{-\log\left(\frac{\exp\left(\gamma\theta\right)-\exp\left(-\gamma\theta\right)}{2\gamma\left(\exp(\theta/2)-\exp(-\theta/2)\right)}\right):\theta\in\real\right\}=-\log\left(\inf\left\{\phi(\theta):\theta\in\real\right\}\right),	
\end{gather*}
where $\gamma:=(b-a+1)/2>1/2$ and 
\begin{gather*}
	\phi(\theta):=\begin{cases}
		\frac{\exp\left(\gamma\theta\right)-\exp\left(-\gamma\theta\right)}{2\gamma\left(\exp(\theta/2)-\exp(-\theta/2)\right)}, & \theta\neq 0,\\
		1, & \theta=0.
	\end{cases}
\end{gather*}
By using L'H\^opital's rule and some straightforward arguments, it is easy to verify that 
\begin{gather*}
	\lim\limits_{|\theta|\rightarrow+\infty}\phi(\theta)=+\infty,\quad \lim\limits_{|\theta|\rightarrow 0}\phi(\theta)=1\quad\text{and}\quad \phi(\theta)=\phi(-\theta).
\end{gather*}
Thus, $\phi$ is continuous at zero (which justifies its definition), coercive and symmetric. Since the log-normalizer function $\psi_P(\theta)=\log\left(M_P[\theta]\right)$ is strictly convex, 
%(see \Cref{sec:pre}) 
we conclude that if a solution exists it must be unique. The coercivity of $\phi$ implies that a solution exists, and due to the symmetry of $\phi$ we can conclude that it must be zero. To summarize,  in this case, $\psi_P^*(\mu)=0$ (with $\theta=0$). If $y\neq \mu$ such that $a\leq y\leq b$ then the optimal solution to \eqref{apndx:eq:Uniform_reformulations_discrete} is nonzero and by the first-order optimality conditions it must satisfy 
\begin{gather}
	\label{apndx:eq:Uniform_stat_discrete}
	y-\frac{(b+1)\exp((b+1)\theta)-a\exp(a\theta)}{\exp((b+1)\theta)-\exp(a\theta)}+\frac{\exp(\theta)}{\exp(\theta)-1}=0.%,
\end{gather}
Therefore, using \eqref{apndx:eq:Uniform_reformulations_discrete} we can summarize that for $y\in[a,b]=\dom \psi_{P}^*$:
\begin{gather*}
	\psi_P^*(y) = \begin{cases}
		0, & y=\mu,\\
		(y-\mu)\theta-\log\left(\frac{\exp((b-\mu+1)\theta)-\exp((a-\mu)\theta)}{n(\exp(\theta)-1)}\right), & y\neq\mu,
	\end{cases}
\end{gather*}
where $\theta$ is the root of \eqref{apndx:eq:Uniform_stat_discrete}.

\subsubsection*{Continuous Uniform}

%We will use the following Lemma.
%\begin{lemma}
%	\label{apndx:lem:exp_ineq}
%	Let 
%	\begin{gather*}
%		\phi(\gamma):= \begin{cases}
%			\frac{\exp(\gamma)-\exp(-\gamma)}{2\gamma}, & \gamma\neq 0,\\
%			1, & \gamma =0.
%		\end{cases}
%	\end{gather*}
%	Then, $\phi$ is continuous function satisfying $\phi(\gamma)\geq 1$ where equality holds if and only if $\gamma=0$.
%%	$\gamma\in\real$, then $\frac{\exp(\gamma)-\exp(-\gamma)}{2\gamma}\geq 1$ and equality holds for $\gamma=0$ (i.e., when $\gamma\rightarrow 0$).
%\end{lemma}
%\begin{proof}
%	Denote $\varphi(\gamma):=\exp(\gamma)-\exp(-\gamma)$. Then $\varphi{''}=\varphi$ and thus we can conclude that $\varphi$ is strictly convex for $\gamma\geq 0$ and strictly concave for $\gamma\leq0$. Therefore, due to the gradient inequality we obtain
%	\begin{gather*}
%		\varphi(\gamma)=\exp(\gamma)-\exp(-\gamma)
%		\begin{cases}
%			> \varphi(0)+\varphi{'}(0)(\gamma-0)=2\gamma, & \gamma> 0,\\
%			< \varphi(0)+\varphi{'}(0)(\gamma-0)=2\gamma, & \gamma< 0,
%		\end{cases}
%	\end{gather*}
%where in the above we used $\varphi(0)=0$ and $\varphi'(0)=2$ from which we can conclude that \linebreak ${\phi(\gamma)=\varphi(\gamma)/(2\gamma)>1}$ for any $\gamma\neq 0$. The continuity of $\phi$ at zero 
%%can be verified by result follows from the above and the fact that 
%%\begin{gather*}
%%	\lim\limits_{\gamma\rightarrow 0}\frac{\exp(\gamma)-\exp(-\gamma)}{2\gamma}=1,
%%\end{gather*}
%%as 
%can be verified by L'H\^opital's rule.
%\end{proof}

By definition 
\begin{gather*}
	\psi_P^*(y)=\sup\left\{y\theta-\log\left(M_P[\theta]\right):\theta\in\real\right\},
\end{gather*}
where for $a<b$ we have that
\begin{gather*}
	M_P[\theta]=\begin{cases}
		\frac{\exp(b\theta)-\exp(a\theta)}{(b-a)\theta}, & \theta\neq 0,\\
		1, & \theta =0.
	\end{cases}
\end{gather*}
Since $M_P[\theta]$ is continuous at zero, then, without loss of generality, we obtain
%Thus, we obtain
\begin{gather}
	\label{apndx:eq:Uniform_reformulations}
	\def\arraystretch{2.5}
	\begin{array}{rl}
		\psi_P^*(y)&=\sup\left\{y\theta-\log\left(\frac{\exp(b\theta)-\exp(a\theta)}{(b-a)\theta}\right):\theta\in\real\right\}\\
		&=\sup\left\{(y-b)\theta-\log\left(\frac{1-\exp(-(b-a)\theta)}{(b-a)\theta}\right):\theta\in\real\right\}\\
		&=\sup\left\{(y-a)\theta-\log\left(\frac{\exp((b-a)\theta)-1}{(b-a)\theta}\right):\theta\in\real\right\}\\	&=\sup\left\{(y-\mu)\theta-\log\left(\frac{\exp((b-\mu)\theta)-\exp((a-\mu)\theta)}{(b-a)\theta}\right):\theta\in\real\right\}.
	\end{array}
\end{gather}
%Thus, we obtain
%\begin{gather}
%	\label{apndx:eq:Uniform_reformulations}
%	\def\arraystretch{2.5}
%	\begin{array}{rl}
%		\psi_P^*(y)&=\max\left\{\sup\left\{y\theta-\log\left(\frac{\exp(b\theta)-\exp(a\theta)}{(b-a)\theta}\right):\theta\in\real\setminus\{0\}\right\},0\right\}\\
%		&=\max\left\{\sup\left\{(y-b)\theta-\log\left(\frac{1-\exp(-(b-a)\theta)}{(b-a)\theta}\right):\theta\in\real\setminus\{0\}\right\},0\right\}\\
%		&=\max\left\{\sup\left\{(y-a)\theta-\log\left(\frac{\exp((b-a)\theta)-1}{(b-a)\theta}\right):\theta\in\real\setminus\{0\}\right\},0\right\}\\	&=\max\left\{\sup\left\{(y-\mu)\theta-\log\left(\frac{\exp((b-\mu)\theta)-\exp((a-\mu)\theta)}{(b-a)\theta}\right):\theta\in\real\setminus\{0\}\right\},0\right\}.
%	\end{array}
%\end{gather}
where $\mu = (a+b)/2$. If $y\geq b$ then from the second formulation above we can conclude that $\psi_P^*(y)=\infty$ by taking $\theta\rightarrow\infty$. Similarly, if, $y\leq a$, then from the third formulation above we can conclude that $\psi_P^*(y)=\infty$ by taking $\theta\rightarrow-\infty$. If $y=\mu$ then %it is can be verified from 
the last formulation of \eqref{apndx:eq:Uniform_reformulations} %and Lemma \ref{apndx:lem:exp_ineq} with $\gamma=\theta(b-a)/2$ that $\psi_P^*(\mu)=0$ for $\theta=0$.
can be written as
\begin{gather*}
%	\label{apndx:eq:cont_uniform_1}
%	\sup\left\{-\log\underbrace{\left(\frac{\exp(\gamma\theta)-\exp(-\gamma\theta)}{2\gamma\theta}\right)}_{=:\phi(\theta)}:\theta\in\real\right\},
	\sup\left\{-\log\left(\frac{\exp(\gamma\theta)-\exp(-\gamma\theta)}{2\gamma\theta}\right):\theta\in\real\right\}=-\log\left(\inf\left\{\phi(\theta):\theta\in\real\right\}\right),	
\end{gather*}
where $\gamma:=(b-a)/2>0$ and 
\begin{gather*}
	\phi(\theta):=\begin{cases}
		\frac{\exp(\gamma\theta)-\exp(-\gamma\theta)}{2\gamma\theta}, & \theta\neq 0,\\
		1, & \theta=0.
	\end{cases}
\end{gather*}
%Since the log-normalizer function $\psi_P(\theta)=\log\left(M_P[\theta]\right)$ is strictly convex (see section \ref{sec:pre}) we can conclude that if a solution exists it must be unique. The strict monotonicity of the logarithm justifies the equivalence in \eqref{apndx:eq:cont_uniform_1} and the conclusion that if a solution 
By using L'H\^opital's rule and some straightforward arguments, it is easy to verify that 
\begin{gather*}
	\lim\limits_{|\theta|\rightarrow+\infty}\phi(\theta)=+\infty,\quad \lim\limits_{|\theta|\rightarrow 0}\phi(\theta)=1\quad\text{and}\quad \phi(\theta)=\phi(-\theta).
\end{gather*}
Thus, $\phi$ is continuous at zero (which justifies its definition), coercive and symmetric. Since the log-normalizer function $\psi_P(\theta)=\log\left(M_P[\theta]\right)$ is strictly convex we can conclude that if a solution exists it must be unique. The coercivity of $\phi$ implies that a solution exists, and due to the symmetry of $\phi$ we can conclude that it must be zero. To summarize,  in this case, $\psi_P^*(\mu)=0$ (with $\theta=0$). If $y\neq \mu$ such that $a<y<b$ then the optimal solution to \eqref{apndx:eq:Uniform_reformulations} is nonzero and by the first-order optimality conditions it must satisfy 
\begin{gather}
	\label{apndx:eq:Uniform_stat}
	y-\frac{b\exp(b\theta)-a\exp(a\theta)}{\exp(b\theta)-\exp(a\theta)}+\frac{1}{\theta}=0.%,
\end{gather}
%which is equivalent to 
%\begin{gather}
%	\label{apndx:eq:Uniform_stat}
% \left[\theta(y-\mu)+1\right](\exp(b\theta)-\exp(a\theta))=\frac{\theta}{2}(b-a)(\exp(b\theta)+\exp(a\theta)).
%\end{gather}
%We can see that $\theta =0$ when $y=\mu$. 
Therefore, using \eqref{apndx:eq:Uniform_reformulations} we can summarize that for $y\in(a,b)=\dom \psi_{P}^*$:
\begin{gather*}
	\psi_P^*(y) = \begin{cases}
		0, & y=\mu,\\
		(y-\mu)\theta-\log\left(\frac{\exp((b-\mu)\theta)-\exp((a-\mu)\theta)}{(b-a)\theta}\right), & y\neq\mu,
	\end{cases}
\end{gather*}
where $\theta$ is the root of \eqref{apndx:eq:Uniform_stat}.

\subsubsection*{Logistic}

The moment generating function for Logistic distribution with location and scaling parameters $\mu$ and $s>0$, respectively, is given by 
\begin{gather*}
	M_P[\theta] = \exp(\mu y)B(1-s\theta,1+s\theta), \qquad s\theta\in(-1,1),
\end{gather*}
where $B(\cdot,\cdot)$ stands for the \emph{Beta function} 
\begin{gather*}
	B(\alpha,\beta) = \int_0^1t^{\alpha-1}(1-t)^{\beta-1}dt.
\end{gather*}
The beta function and the closely related \emph{gamma function}
\begin{gather*}
	\Gamma(\alpha) = \int_0^{\infty} t^{\alpha-1}\exp(-t)dt, \qquad \alpha>0,
\end{gather*}
%introduced by Leonard Euler 
share the following well-known relation
\begin{gather}
	\label{apndx:eq:beta_gamma_relation}
	B(\alpha,\beta) = \frac{\Gamma(\alpha)\Gamma(\beta)}{\Gamma(\alpha+\beta)}.	
\end{gather}  
The gamma function is an extension of the factorial as for a positive integer $\alpha$ 
it holds that \linebreak $\Gamma(\alpha)=(\alpha-1)!$. In the following, we will use the well-known function equations
\begin{gather}
	\label{apndx:eq:beta_func_1}
	B(\alpha+1,\beta) = B(\alpha,\beta)\frac{\alpha}{\alpha+\beta},
\end{gather}
and 
\begin{gather}
	\label{apndx:eq:beta_eulers_reflection}
	B(\alpha,1-\alpha) = \Gamma(1-\alpha)\Gamma(\alpha) =\frac{\pi}{\sin(\pi \alpha)},\qquad \alpha\notin\mathbb{Z}.
\end{gather}
The latter is known as Euler's reflection formula or Euler's function equation. Further details and proofs for both \eqref{apndx:eq:beta_func_1} and \eqref{apndx:eq:beta_eulers_reflection} can be found, for example, in \cite{artin1964gamma}. \withsmallskip

Since $s\theta\in(-1,1)$, the above relations imply that for any $\theta\neq 0$
\begin{gather*}	
	\phi_s(\theta):=B(1-s\theta,1+s\theta)\overset{\eqref{apndx:eq:beta_func_1}}{=}B(-s\theta,1+s\theta)\frac{-s\theta}{-s\theta+1+s\theta}\overset{\eqref{apndx:eq:beta_eulers_reflection}}{=}\frac{-\pi s\theta}{\sin(-\pi s\theta)}.
\end{gather*}
For $\theta=0$ we can verify by \eqref{apndx:eq:beta_gamma_relation} that 
\begin{gather*}
	\phi_s(\theta)=B_s(1-s\theta,1+s\theta)=1.
\end{gather*}
Thus, we can summarize
\begin{gather}
	\label{apndx:eq:beta_mgf_equiv}
	\phi_s(\theta)=B(1-s\theta,1+s\theta)=\begin{cases}
		1, & s\theta=0,\\
		\frac{-\pi s\theta}{\sin(-\pi s\theta)}, & s\theta\in(-1,1)\setminus\{0\}.
	\end{cases}
\end{gather}
Using L'H\^opital's rule we can verify that $\phi_s$ is continuous at $\theta=0$. Since $-\pi s\theta\geq\sin(-\pi s\theta)$ for all $s\theta\in(-1,1)$ we can conclude that $\phi_s(\theta)\geq 1$ for all $s\theta\in(-1,1)$ and equality ($\phi_s(\theta)= 1$) holds if and only if $s\theta=1$. Taking $|s\theta|\rightarrow 1$ it is evident that $\phi_s(\theta)\rightarrow\infty$. In addition, for any $\theta\neq 0$ the derivative of $\phi$ is given by 
\begin{gather*}
	\phi_s'(\theta) = -\pi s\left[\frac{\sin(-\pi s\theta)+\pi s\theta\cos(-\pi s\theta)}{\sin^2(-\pi s\theta)}\right],
\end{gather*}
and consequently
\begin{gather}
	\label{apndx:eq:phi_val_div_frac}
	\frac{\phi_s'(\theta)}{\phi_s(\theta)} = \frac{\sin(-\pi s\theta)+\pi s\theta\cos(-\pi s\theta)}{\theta\sin(-\pi s\theta)}.
\end{gather}
We are now ready to evaluate Cram\'er's rate function that corresponds to the logistic distribution. 
\begin{gather}
	\label{apndx:eq:Cramer_conj_def}
	\def\arraystretch{2}
	\begin{array}{rl}
		\psi_P^*(y)&=\sup\left\{y\theta-\log\left(M_P[\theta]\right):\theta\in\real\right\}\\	
		&= \sup\left\{(y-\mu)\theta-\log\left(\phi_s(\theta)\right):\theta\in\real\right\}.
	\end{array}
\end{gather}
If $y=\mu$ then the discussion that follows equation \eqref{apndx:eq:beta_mgf_equiv} implies that $\sup\{-\log(\phi_s(\theta)):\theta\in\real\}\leq 0$ where the upper bound is attained for $\theta=0$ (since $\phi_s(\theta)\geq1$ and $\phi_s(0)=1$). Thus, we can conclude that $\psi_P^*(\mu)=0$. If $y\neq\mu$ then the optimal solution to \eqref{apndx:eq:Cramer_conj_def} satisfies $\theta\neq 0$. Since, in addition, for $|s\theta|\rightarrow 1$ we have that $\phi_s(\theta)\rightarrow\infty$, and consequently, $-\log(\phi_s(\theta))\rightarrow -\infty$, an optimal solution to \eqref{apndx:eq:Cramer_conj_def} for the case $y\neq\mu$ must satisfy the first-order optimality conditions 
\begin{gather}
	\label{apndx:eq:Logistic_stat}
	0= y-\mu - \frac{\phi_s'(\theta)}{\phi_s(\theta)} = y-\mu -\frac{1}{\theta} - \frac{\pi s}{\tan{(-\pi s\theta)}},
	%\qquad\Rightarrow\qquad {\ed\frac{\pi s\theta}{(y-\mu)\theta-1} = \tan(-\pi s\theta)},
\end{gather}
where the above follows from \eqref{apndx:eq:phi_val_div_frac}. 
%Unfortunately, I cannot see how we can extract a closed form solution to the above equation.
To summarize, 
\begin{gather*}
	\psi_P^*(y)=\begin{cases}
		0, & y=\mu,\\
		(y-\mu)\theta-\log\left(B(1-s\theta,1+s\theta)\right), & y\neq \mu,
	\end{cases}
\end{gather*}
where $\theta\in\real$ is the nonzero root of \eqref{apndx:eq:Logistic_stat}.

\end{document}